\documentclass{lmcs}
\usepackage[utf8]{inputenc}
\pdfoutput=1

\usepackage{lastpage}
\lmcsdoi{16}{3}{8}
\lmcsheading{}{\pageref{LastPage}}{}{}%
{May~23,~2018}{Aug.~05,~2020}{}

\usepackage{amscd}
\usepackage{graphics}
\usepackage{hyperref}
\usepackage{color}
\usepackage{cancel}
\usepackage{stmaryrd}
\usepackage{tikz}
\usetikzlibrary{matrix,arrows}
\usepackage{amsmath}
\usepackage{amsthm}
\usepackage{amssymb}
\usepackage{proof}
\usepackage{verbatim,cmll}
\usepackage{setspace}
\usepackage{lscape}
\usepackage{latexsym,relsize}
\usepackage{amscd}
\usepackage{multicol}
\usepackage{bussproofs}
\usepackage{newtxmath}
\usepackage{tikz}
\usepackage[position=top]{subfig}
\usepackage{leqno}
\usepackage{marvosym}
\DeclareSymbolFont{extraup}{U}{zavm}{m}{n}
\DeclareMathSymbol{\varheart}{\mathalpha}{extraup}{86}
\DeclareMathSymbol{\vardiamond}{\mathalpha}{extraup}{87}
\DeclareMathSymbol{\vardiamond}{\mathalpha}{extraup}{87}
\usetikzlibrary{arrows}
\usetikzlibrary{matrix}
\usetikzlibrary{patterns}
\usetikzlibrary{shapes}

\newcommand{\jir}{J^{\infty}(\bba^{\delta})}
\newcommand{\mir}{M^{\infty}(\bba^{\delta})}

\newcommand{\commment}[1]{}

\renewcommand{\phi}{\varphi}
\renewcommand{\emptyset}{\varnothing}
\newcommand{\ca}{\mathbb{A}^\delta}

\renewcommand{\epsilon}{\varepsilon}

\newcommand{\nomi}{\mathbf{i}}
\newcommand{\nomj}{\mathbf{j}}

\newcommand{\nomh}{\mathbf{h}}
\newcommand{\cnomm}{\mathbf{m}}
\newcommand{\cnomn}{\mathbf{n}}

\newcommand{\blhd}{\blacktriangleleft_f}
\newcommand{\brhd}{\blacktriangleright_g}

\newcommand{\bba}{\mathbb{A}}
\newcommand{\bbb}{\mathbb{B}}
\newcommand{\bbA}{\mathbb{A}}
\newcommand{\bbB}{\mathbb{B}}
\newcommand{\bbas}{\bbA^{\delta}}
\newcommand{\bbbs}{\bbB^{\delta}}

\newcommand{\kbbas}{K(\mathbb{A}^\delta)}
\newcommand{\obbas}{O(\mathbb{A}^\delta)}
\newcommand{\obbbs}{O(\mathbb{B}^\delta)}
\newcommand{\kbbbs}{K(\mathbb{B}^\delta)}
\newcommand{\rhu}{\rightharpoonup}
\newcommand{\lhu}{\leftharpoonup}

\newcommand{\bigamp}{\mathop{\mbox{\Large \&}}}

\newcommand*\circled[1]{\tikz[baseline=(char.base)]{
		\node[shape=circle ,draw, minimum size=3.5mm, inner sep=0pt] (char) {#1};}}

\tikzset{
	treenode/.style = {align=center, inner sep=0pt, text centered},
	Ske/.style = {treenode, ellipse, double, draw=black,
		minimum width=6pt, thick},% arbre rouge noir, noeud noir
	PIA/.style = {treenode, ellipse, black, draw=black,
		minimum width=6pt},% arbre rouge noir, noeud rouge
	Crit/.style = {treenode, rectangle, draw=black,
		minimum width=0.5em, minimum height=0.5em}% arbre rouge noir, nil
}

\title{Constructive Canonicity of Inductive Inequalities}
\author[W.~Conradie]{Willem Conradie\rsuper{a}\rsuper{1}}
\author[A.~Palmigiano]{Alessandra Palmigiano\rsuper{b}\rsuper{c}\rsuper{2}}

\thanks{\lsuper{1} Supported by the National Research Foundation of South Africa, Grant number 81309, and the financial assistance of the Faculty of Science of the University of the Witwatersrand.}
\thanks{\lsuper{2} Supported by the NWO Vidi grant 016.138.314, by the NWO Aspasia grant 015.008.054, and by a Delft Technology Fellowship awarded in 2013.}

\address{\lsuper{a} School of Mathematics, University of the Witwatersrand, Johannesburg, South Africa}
\address{\lsuper{b} School of Business and Economics, Vrije Universiteit, Amsterdam, The Netherlands}
\address{\lsuper{c} Department of Mathematics and Applied Mathematics, University of Johannesburg, South Africa}


\begin{document}
\maketitle
%\begin{frontmatter}

  \begin{abstract}
  We prove the canonicity of inductive inequalities  in a constructive meta-theory, for  classes of logics algebraically captured by varieties of normal and regular lattice expansions. This result encompasses Ghilardi-Meloni's  and Suzuki's constructive canonicity results for Sahlqvist formulas and inequalities%\cite{GhMe97,Suzuki11a,Suzuki11b,Suzuki13}
  , and  is based on an application of the tools of unified correspondence theory. Specifically, we provide an alternative interpretation of the language of the algorithm ALBA for lattice expansions:  nominal and conominal variables are respectively interpreted  as closed and open elements of  canonical extensions of normal/regular lattice expansions, rather than as completely join-irreducible and meet-irreducible elements of perfect normal/regular lattice expansions. We show the correctness of ALBA with respect to this interpretation. From this fact,  the constructive canonicity of the inequalities on which ALBA succeeds follows by an
  %straightforward
  adaptation of the standard  argument.  The claimed result then follows as a consequence  of the  success of ALBA on  inductive inequalities.\\ %The methodology covers many different settings, e.g.\ distributive modal logic and substructural logics on a lattice base.
  %\begin{keyword}
%\end{keyword}
\end{abstract}
% \end{frontmatter}
%\marginnote{make sure that the title is exactly what appears in the references of the many papers in which it is cited}

\section{Introduction} %\marginnote{I just made a sketch. to be seriously redone. make sure that the title is exactly what appears in the references of the many papers in which it is cited}
Perhaps the most important uniform methodology for proving Kripke completeness for modal logics is the notion of canonicity, which, thanks to duality, can be studied both frame-theoretically and  algebraically. Frame-theoretically, canonicity can be formulated as {\em d-persistence}, i.e.\ preservation of validity from any given descriptive general frame to its underlying Kripke frame (or in other words, equivalence between validity w.r.t.\ admissible assignments and w.r.t.\ arbitrary assignments); algebraically, as preservation of validity from any given modal algebra to its canonical extension. The study of canonicity has been extended from classical normal modal logic to its many neighbouring logics, and has given rise to a rich literature. Particularly  relevant to the present paper are two general methods for canonicity, pioneered by Sambin and Vaccaro \cite{SaVa89} and by Ghilardi and Meloni \cite{GhMe97}. Sambin and Vaccaro obtain canonicity for  Sahlqvist formulas of classical modal logic in a frame-theoretic setting as a byproduct of correspondence. The core of their proof strategy is the observation that, whenever it exists, the first-order correspondent of a modal formula provides an equivalent rewriting of the modal formula with no occurring propositional variables, so that its validity w.r.t.\ admissible assignments is tantamount to its validity w.r.t.\ arbitrary assignments, which proves d-persistence. Sambin-Vaccaro's proof strategy, sometimes called {\em canonicity-via-correspondence}, has been imported to the algorithmic proof of canonicity given in \cite{CoGoVa06},  which achieves a uniform proof for the widest class known so far, the so-called {\em inductive} formulas, which significantly extends the class of Sahlqvist inequalities. Here it is the algorithm SQEMA which produces the equivalent rewriting (i.e.\ the first-order correspondent) of the modal formula, and this is something it does successfully for (at least) all inductive formulas.
%In the context of this algorithmic development, it is the algorithm (SQEMA) which fulfills the task of producing, whenever successful, the equivalent rewriting (i.e.\ the first-order correspondent) of the modal formula, and  is shown to succeed on all inductive formulas.

Ghilardi and Meloni's work \cite{GhMe97} shows that canonicity can be meaningfully investigated purely algebraically, in a constructive meta-theory where correspondence is not even defined, in general. Indeed, in \cite{GhMe97}, the canonical extension construction for certain bi-intuitionistic modal algebras, later applied also to general lattice and poset expansions in \cite{GeHa01,DunnGP05}, is formulated in terms of general filters and ideals, and does not depend on any form of the axiom of choice (such as the existence of `enough' optimal filter-ideal pairs). Thus, while the constructive canonical extension need not be perfect in the sense of \cite{JoTa52}, the canonical embedding map, sending the original algebra into its canonical extension, retains the properties of denseness and compactness (cf.\ Definition \ref{Def:Canon:Ext}). These properties make it possible for the authors of \cite{GhMe97} to identify a class of formulas which are {\em constructively canonical}  (i.e.\ the validity of which is preserved under constructive canonical extension). These inequalities are identified from the order-theoretic properties of the induced term-functions, and the preservation of their validity  from an algebra to its constructive canonical extension is shown in two steps: first, from the elements of a given algebra to the closed/open elements of its canonical extension, and then from the closed/open elements to arbitrary elements. This proof strategy is very similar to that of J\'{o}nsson \cite{Jonsson94}, but was developed independently. It should be noted however that, in terms of the classes of formulas to which they apply, the results of both \cite{GhMe97} and \cite{Jonsson94} fall within the scope of applicability of the canonicity-via-correspondence method.

The approaches of Sambin-Vaccaro, on the one hand, and Ghilardi-Meloni, %and J\'{o}nsson,
on the other, have been very influential, and have contributed to a certain binary divide detectable in the literature between correspondence and canonicity, namely: correspondence being typically developed on frames, and canonicity on algebras. Moreover, subsequent algebraic proofs of canonicity have mostly remained restricted to Sahlqvist formulas, rather than considering e.g.\ the wider class of inductive formulas \cite{GorankoV06}.  For the latter class, the first algebraic proof of canonicity in the style of J\'{o}nsson-Ghilardi-Meloni  appeared only very recently \cite{PaSoZh15a}. \emph{Unified correspondence theory} \cite{CoGhPa14}, to which the contributions of the present paper belong, bridges this divide in the sense that will be explained below, and by doing this, succeeds in importing Sambin-Vaccaro's proof strategy to the constructive setting of Ghilardi-Meloni, thus providing the conceptual unification of these very different perspectives which had remained a prominent open problem in the subsequent literature. For instance,  the intermediate step of \cite{GhMe97} can be recognized as the equivalent rewriting, independent of the evaluation of proposition variables, pursued in \cite{SaVa89}.

At the core of unified correspondence is an {\em algebraic} reformulation of {\em correspondence} theory \cite{ConPalSou}, with its ensuing algorithmic canonicity-via-correspondence argument \cite{ConPal12}. This reformulation makes it possible to construe the computation of first-order correspondents in two phases: reduction, and translation (cf.\ \cite{ConPal13} for an expanded discussion on this point).
%The modus operandi of unified correspondence
Formulas/inequalities are interpreted in the canonical extension $\bbas$ of a given algebra $\bba$, and a calculus of rules (captured by the ALBA algorithm) is  applied to rewrite them into equivalent expressions with no occurring propositional variables, called pure. The pure expressions may, however, contain (non-propositional) variables known as nominals and co-nominals. If successful, achieving pure expressions completes  the reduction phase. Pure expressions are already enough to implement Sambin-Vaccaro's canonicity strategy: indeed, the validity of pure expressions under assignments sending propositional variables into $\bba$ (identified with the {\em admissible} assignments on $\bbas$) is tantamount to their validity w.r.t.\ arbitrary assignments on $\bbas$, and this establishes the canonicity of the original formula or inequality.

In a non-constructive setting, the nominals and co-nominals are interpreted as ranging over the completely join-irreducible or meet-irreducible elements of $\bbas$. The soundness of the rewrite rules is based, in part, on the fact that the completely join-irreducible and meet-irreducible elements respectively join-generate and meet-generate the ambient algebra $\bbas$. Moreover, in this setting, completely meet- and join-irreducible elements correspond, via discrete duality, to first-order definable subsets of the dual relational semantics. Thus the first-order frame correspondent of the original formula or inequality can be obtained by simply applying the appropriate standard translation to the pure expressions. This is known as the translation phase.

In a constructive setting, the situation just described is changed by the fact that we can no longer rely on the completely join-irreducible and meet-irreducible elements to respectively join-generate and meet-generate $\bbas$. %the canonical extension.
However, we may fall back on the closed and open elements of $\bbas$ as complete join- and meet-generators, and adjust the interpretations of nominals and co-nominals accordingly. By doing this, the reduction phase remains sound in the non-constructive setting, and still yields canonicity. As expected, however, discrete duality is not available in general, and the possibility of translating to the relational semantics only applies modulo restriction to the setting in which discrete duality is available. %\marginnote{I changed the phrasing of this sentence and made them less negative-sounding} %Thus, by decoupling reduction and translation, unified correspondence makes it possible to import Sambin-Vaccaro's strategy into Ghilardi-Meloni's constructive setting.

In the present paper, we expand on the ideas outlined above  and give formal proofs of their  correctness. We will do this in the setting of arbitrary normal or regular lattice expansions, and  prove the constructive canonicity of inductive inequalities in the appropriate signature. The inductive inequalities significantly enlarge the set of Sahlqvist inequalities, accounting for important axiomatic instances\footnote{Well known formulas such as Frege's axiom are inductive but not Sahlqvist (cf.\ Example \ref{ex: frege as type 3}).}, and
the canonicity proof for inductive inequalities presented here subsumes and unifies those of \cite{GhMe97,SaVa89}. %Suzuki11a,Suzuki11b,Suzuki13}. %\marginnote{Not all of these Suzuki papers are relevant.}

%\marginnote{I entirely rewrote the paragraphs below}
The results of the present paper contribute to the realization of one of the core objectives of unified correspondence theory, which is to create an environment in which different proof techniques for canonicity and correspondence can be systematically compared and    connected  to each other. For instance, in \cite{CoPaZh15a}, the algorithmic route to correspondence and canonicity for distributive lattice based logics has been compared to the route via reduction to the Boolean setting, by means of G\"{o}del-type translations. In the specific case of canonicity, the results of the present paper build on techniques developed investigating  the phenomenon of canonicity via pseudocorrespondence \cite{CoPaSoZh15}, and concur to complete the picture begun in \cite{PaSoZh15a}, where  Sambin-Vaccaro's methodology for canonicity has been systematically connected to  J\'onsson's algebraic but non-constructive canonicity technique. These results also serve as base for proving constructive canonicity in the setting of mu-calculi \cite{CCPZ}, a result which bridges constructive canonicity with algebraic and algorithmic  correspondence and canonicity for mu-calculi \cite{CoCr14,CoFoPaSo15}. Moreover, the present results are of direct relevance to the problem of canonicity for possibility semantics in modal logic (cf.\ \cite{Ho16,YZ16}, more about this in the conclusions).

Recently, systematic connections have been established between  unified correspondence  and the theory of proper display  calculi for certain axiomatic extensions of basic normal DLE-logics  \cite{GMPTZ}. These connections have been also put to use in defining analytic Gentzen calculi for subintuitionistic logics \cite{MaZh16}. In particular, a uniform argument for proving that the resulting proper display calculi are sound and conservative crucially uses the canonicity of (a certain subclass of) inductive inequalities. From these and other results, a perspective in modern structural proof theory has emerged which is based on the systematic integration of results in algebraic logic into the design of analytic calculi; the results and insights of the present paper naturally fit also in this line of investigation (more about this in the conclusions).

It is interesting to observe that, while  the core tools of unified correspondence, and the algorithm ALBA in particular,   have proven their versatility in settings as diverse as hybrid logics \cite{CoRo14}, many valued logics \cite{LeRoux:MThesis:2016}, and monotone modal logic \cite{FrPaSa14}, through the development of  applications such as \cite{PaSoZh15a,GMPTZ,CoPaSoZh15}, these tools have acquired  novel conceptual significance, which cannot be reduced exclusively to their original purpose as  computational tools for correspondence theory. In this respect, the results of the present paper are yet another instance of the potential of ALBA to be used as a general-purpose computational tool, capable of meaningfully contributing to more general and different issues than pure correspondence.

\paragraph{Structure of the paper.} In Section \ref{Sec:Prelim} we provide some necessary preliminaries. Particularly, we introduce the logical languages we will consider, we provide them with algebraic semantics in the form of lattices expanded with normal and regular operations, and discuss the constructive canonical extensions of the latter. In Section \ref{Sec:Ind:Sahl:Ineq} we define the Inductive and Sahlqvist inequalities in the setting of lattices with normal and regular operations. Section \ref{sec:constructive ALBA} contains the specification of the constructive version of ALBA, and the partial correctness of this algorithm is proved in Section \ref{Sec:Crctns:ALBA}. The latter relies on some technical lemmas which have been relegated to Appendix \ref{sec:appendix}. We next show, in Section \ref{Sec:Complete:For:Inductive:Section}, that constructive ALBA successfully reduces all inductive and Sahlqvist inequalities. In Section \ref{sec:canonicity} we prove that all inequalities on which constructive ALBA succeeds are constructively canonical. From this our main theorem follows, i.e. that all inductive inequalities are constructively canonical. We conclude in Section \ref{Sec:Conclusions}.

\section{Preliminaries}\label{Sec:Prelim}

\subsection{Language}
Our base language is an unspecified but fixed language $\mathcal{L}_\mathrm{LE}$, to be interpreted over  lattice expansions of compatible similarity type.
We will make heavy use of the following auxiliary definition: an {\em order-type} over $n\in \mathbb{N}$ is an $n$-tuple $\epsilon\in \{1, \partial\}^n$. We will write $\varepsilon(i)$ or $\varepsilon_i$ for the $i$-th component of order type $\epsilon$ with $1 \leq i \leq n$. We denote its {\em opposite} order type by $\epsilon^\partial$, that is, $\epsilon^\partial_i = 1$ iff $\epsilon_i=\partial$ for every $1 \leq i \leq n$. Throughout the paper, order-types will be typically associated with arrays of variables $\vec p: = (p_1,\ldots, p_n)$. When the order of the variables in $\vec p$ is not specified, we will sometimes abuse notation and write $\varepsilon(p) = 1$ or $\varepsilon(p) = \partial$. For any lattice $\bba$, we let $\bba^1: = \bba$ and $\bba^\partial$ be the dual lattice, that is, the lattice associated with the converse partial order of $\bba$. Accordingly, we write $\leq^{\partial}$ for $\geq$, $\geq^{\partial}$ for $\leq$, $\wedge^{\partial}$ for $\vee$,  $\vee^{\partial}$ for $\wedge$, $\top^{\partial}$ for $\bot$ and $\bot^{\partial}$ for $\top$, while $\star^1$ denotes $\star$ for any symbol $\star \in \{\leq, \geq, \wedge, \vee, \top, \bot\}$. For any order type $\varepsilon$, we let $\bba^\varepsilon: = \Pi_{i = 1}^n \bba^{\varepsilon_i}$.
	
The language $\mathcal{L}_\mathrm{LE}(\mathcal{F}, \mathcal{G})$ (from now on abbreviated as $\mathcal{L}_\mathrm{LE}$) takes as parameters: 1) a denumerable set $\mathsf{PROP}$ of proposition letters, elements of which are denoted $p,q,r$, possibly with indexes; 2) disjoint sets of connectives $\mathcal{F}$ and  $\mathcal{G}$ such that $\mathcal{F}: = \mathcal{F}_r\uplus \mathcal{F}_n$ and $\mathcal{G}: = \mathcal{G}_r\uplus \mathcal{G}_n$.  Each $f\in \mathcal{F}$ and $g\in \mathcal{G}$ has arity $n_f\in \mathbb{N}$ (resp.\ $n_g\in \mathbb{N}$) and is associated with some order-type $\varepsilon_f$ over $n_f$ (resp.\ $\varepsilon_g$ over $n_g$).\footnote{Unary $f$ (resp.\ $g$) will be sometimes denoted as $\Diamond$ (resp.\ $\Box$) if the order-type is 1, and $\lhd$ (resp.\ $\rhd$) if the order-type is $\partial$.} Connectives belonging to $\mathcal{F}_r$ or $\mathcal{G}_r$ are always unary. 
	  
The motivation behind splitting the signature into these disjoint sets will become clearer in the next subsection but, in brief, the idea is that the algebraic interpretation of members $f \in \mathcal{F}$ (resp.\ $g \in \mathcal{G}$) preserve finite joins (resp.\ meets) in arguments $i$ with $\varepsilon_f(i) = 1$ (resp.\ $\varepsilon_g(i) = 1$) while turning finite meets (resp.\ joins) in argument $i$ into joins if $\varepsilon_f(i) = 1\partial$ (resp.\ $\varepsilon_g(i) = 1$). For $f \in \mathcal{F}_r$ and $g \in \mathcal{G}_r$ this only applies to non-empty finite meets and joins.

The terms (formulas) of $\mathcal{L}_\mathrm{LE}$ are defined recursively as follows:
	\[
	\phi ::= p \mid \bot \mid \top \mid \phi \wedge \phi \mid \phi \vee \phi \mid f(\overline{\phi}) \mid g(\overline{\phi})
	\]
	where $p \in \mathsf{PROP}$, $f \in \mathcal{F}$, $g \in \mathcal{G}$. Terms in $\mathcal{L}_\mathrm{LE}$ will be denoted either by $s,t$, or by lowercase Greek letters such as $\varphi, \psi, \gamma$ etc. %The set of all $\mathcal{L}_\mathrm{LE}$-inequalities $\phi \leq \psi$ where $\phi, \psi$ are $\mathcal{L}_\mathrm{LE}$-terms will be denoted  $\mathrm{LE}$. The set of all $\mathcal{L}_\mathrm{LE}$-quasi-inequalities, i.e., expressions of the form $(\phi_1 \leq \psi_1 \amp \cdots \amp \phi_n \leq \psi_n) \Rightarrow \phi \leq \psi$ where $\phi_1, \ldots, \phi_n, \psi_1, \ldots \psi_n, \phi, \psi \in \mathcal{L}_\mathrm{LE}$, will be denoted by $\mathcal{L}_\mathrm{LE}^{\mathit{quasi}}$.

%The formulas of \emph{distributive modal logic} (cf.\ \cite{GNV}, \cite{ALBAPaper}), denoted $\mathrm{DML}_{\mathit{term}}$, are obtained by instantiating $\mathcal{F}: = \{\Diamond, {\lhd}\}$ with $n_\Diamond = n_\lhd = 1$, $\varepsilon_\Diamond = 1$  and $\varepsilon_\lhd = \partial$, and   $\mathcal{G} = \{\Box, {\rhd}\}$ with $n_\Box = n_\rhd = 1$, $\varepsilon_\Box = 1$  and $\varepsilon_\rhd = \partial$. The formulas of the Full Lambek calculus \cite{la61} are obtained by instantiating $\mathcal{F}: = \{\circ\}$ with $n_\circ  = 2$, $\varepsilon_\circ = (1, 1)$   and   $\mathcal{G} = \{\backslash, /\}$ with $n_\backslash = n_/ = 2$, $\varepsilon_\backslash = (\partial, 1)$  and $\varepsilon_/ = (1, \partial)$. The formulas of the  Lambek-Grishin calculus (cf.\ \cite{Moortgat}) %\marginnote{add reference to Moortgat/Grishin} are obtained by instantiating $\mathcal{F}: = \{\circ, \starfor, \starback \}$ with $n_\circ = n_{\starback} = n_{\starfor} = 2$, $\varepsilon_\circ = (1, 1)$, $\varepsilon_{\starback} = (\partial, 1)$, $\varepsilon_{\starfor} = (1, \partial)$   and   $\mathcal{G} := \{\star, \circfor, \circback\}$ with $n_\star = n_{\circfor} = n_{\circback} = 2$, $\varepsilon_\star = (1, 1)$,  $\varepsilon_{\circback} = (\partial, 1)$, $\varepsilon_{\circfor} = (1, \partial)$.

\begin{rem}\label{rem:cmpctpres:signatures}
	The above presentation of the language $\mathcal{L}_{LE}$ follows the conventions established in a series of papers, including \cite{ConPal13, CoPaZh15a, GMPTZ}. However, as should be clear by the comments above, the intended interpretations of the connectives belonging to families $\mathcal{F}$ and $\mathcal{G}$ are order-dual to one another. As such it is possible to present them more compactly, viz.\ by specifying that a signature consists of a family $\mathcal{H}$ of connectives, such that each $h \in \mathcal{H}$ has arity $n_h \in \mathbb{N}$, and \emph{adjunction-type} $\alpha_{h} \in \{1, \partial\}$. (Relating this to the standard terminology, above, $\mathcal{F} = \{h \in \mathcal{H} \mid \alpha_h = 1 \}$ and $\mathcal{G} = \{h \in \mathcal{H} \mid \alpha_h = \partial \}$.) A subset, denoted $\mathcal{H}_n$, of the connectives in $\mathcal{H}$ are designated as \emph{normal}, while the remainder, $\mathcal{H}_r  = \mathcal{H} \setminus \mathcal{H}_n$, are \emph{regular}.  Every regular connective $h \in \mathcal{H}_r$ has arity $n_h = 1$. A number of notions can be formulated mores succinctly using this convention, see Remarks \ref{rem:cmpctpres:LE:Alg}, \ref{rem:compact:nrmalized:ops} and \ref{rem:cmpct:extended:language}.
\end{rem}

\subsection{Lattice expansions, and their canonical extensions}
The following definition captures the algebraic setting of the present paper, which  generalizes the normal lattice expansions of \cite{ConPal13} and the regular distributive lattice expansions of \cite{PaSoZh15b}. In what follows, we will refer to these algebras simply as {\em lattice expansions}.
	
	\begin{defi}
		\label{def:DLE}
		For any tuple $(\mathcal{F}, \mathcal{G})$ of function symbols as above, a {\em  lattice expansion} (LE) is a tuple $\bba = (L, \mathcal{F}^\bbA, \mathcal{G}^\bbA)$ such that $L$ is a bounded  lattice, $\mathcal{F}^\bbA = \{f^\bbA\mid f\in \mathcal{F}\}$ and $\mathcal{G}^\bbA = \{g^\bbA\mid g\in \mathcal{G}\}$, such that every $f^\bbA\in\mathcal{F}^\bbA$ (resp.\ $g^\bbA\in\mathcal{G}^\bbA$) is an $n_f$-ary (resp.\ $n_g$-ary) operation on $\bbA$, and moreover,
\begin{enumerate} %An LE is {\em normal} if
\item every $f^\bbA\in\mathcal{F}_n^\bbA$ (resp.\ $g^\bbA\in\mathcal{G}_n^\bbA$) preserves finite (hence also empty) joins (resp.\ meets) in each coordinate with $\epsilon_f(i)=1$ (resp.\ $\epsilon_g(i)=1$) and reverses finite (hence also empty) meets (resp.\ joins) in each coordinate with $\epsilon_f(i)=\partial$ (resp.\ $\epsilon_g(i)=\partial$).
\item every $f^\bbA\in\mathcal{F}_r^\bbA$ (resp.\ $g^\bbA\in\mathcal{G}_r^\bbA$) preserves finite nonempty joins (resp.\ meets) if $\epsilon_f=1$ (resp.\ $\epsilon_g=1$) and reverses finite nonempty meets (resp.\ joins) if $\epsilon_f=\partial$ (resp.\ $\epsilon_g =\partial$).
 \end{enumerate}
%
     %\footnote{\label{footnote:DLE vs DLO} Normal LEs are sometimes referred to as {\em  lattices with operators} (LOs). This terminology derives from the setting of Boolean algebras with operators, in which operators are understood as operations which preserve finite (hence also empty) joins in each coordinate. Thanks to the Boolean negation, operators are typically taken as primitive connectives, and all the other operations are reduced to these. However, this terminology results somewhat ambiguous in the lattice setting, in which primitive operations are typically maps which are operators if seen as $\bbA^\epsilon\to \bbA^\eta$ for some order-type $\epsilon$ on $n$ and some order-type $\eta\in \{1, \partial\}$. Rather than speaking of lattices with $(\varepsilon, \eta)$-operators, we then speak of normal LEs.}
Let $\mathbb{LE}$ be the class of LEs. Sometimes we will refer to certain LEs as $\mathcal{L}_\mathrm{LE}$-algebras when we wish to emphasize that these algebras have a compatible signature with the logical language we have fixed.
	\end{defi}

\begin{rem}\label{rem:cmpctpres:LE:Alg}
	Following Remark \ref{rem:cmpctpres:signatures}, for any signature $\mathcal{H}$, an LE algebra can be defined as a structure $\mathbb{A} = (L, \mathcal{H}^{\mathbb{A}})$ such that $L$ is a bounded lattice and $\mathcal{H}^{\mathbb{A}} = \{h^{\mathbb{A}} \mid h \in \mathcal{H}\}$ is a family of operations on $L$ such that $h^{\mathbb{A}}$ takes $n_h$ arguments, and $h^{\mathbb{A}}(a_1, \ldots, a (\vee^{\epsilon_h(i)})^{\alpha(h)} b, \ldots, a_{n_h}) = h^{\mathbb{A}}(a_1, \ldots, a, \ldots, a_{n_h}) \vee^{\alpha_h} h^{\mathbb{A}}(a_1, \ldots, b, \ldots, a_{n_h})$. Moreover, if $h \in \mathcal{H}_n$, then we have that $h^{\mathbb{A}}(a_1, \ldots, (\bot^{\epsilon_h(i)})^{\alpha(h)}, \ldots, a_{n_h}) = \bot^{\alpha_h}$. This clearly makes possible analogous compact presentations of the identities below as well as the axioms and rules in Definition \ref{def:DLE:logic:general}.
\end{rem}

In the remainder of the paper,
we will abuse notation and write e.g.\ $f$ for $f^\bbA$ when this causes no confusion.
%Normal LEs constitute the main semantic environment of the present paper. Henceforth, since every LE is assumed to be normal, the adjective will be typically dropped.
The class of all LEs is equational, and can be axiomatized by the usual lattice identities and the following equations for any $f\in \mathcal{F}$ (resp.\ $g\in \mathcal{G}$) and $1\leq i\leq n_f$ (resp.\ for each $1\leq j\leq n_g$):
	\begin{itemize}
		\item if $\varepsilon_f(i) = 1$, then $f(p_1,\ldots, p\vee q,\ldots,p_{n_f}) = f(p_1,\ldots, p,\ldots,p_{n_f})\vee f(p_1,\ldots, q,\ldots,p_{n_f})$; moreover if $f\in \mathcal{F}_n$, then $f(p_1,\ldots, \bot,\ldots,p_{n_f}) = \bot$,
		\item if $\varepsilon_f(i) = \partial$, then $f(p_1,\ldots, p\wedge q,\ldots,p_{n_f}) = f(p_1,\ldots, p,\ldots,p_{n_f})\vee f(p_1,\ldots, q,\ldots,p_{n_f})$; moreover if $f\in \mathcal{F}_n$, then  $f(p_1,\ldots, \top,\ldots,p_{n_f}) = \bot$,
		\item if $\varepsilon_g(j) = 1$, then $g(p_1,\ldots, p\wedge q,\ldots,p_{n_g}) = g(p_1,\ldots, p,\ldots,p_{n_g})\wedge g(p_1,\ldots, q,\ldots,p_{n_g})$; moreover if $g\in \mathcal{G}_n$, then  $g(p_1,\ldots, \top,\ldots,p_{n_g}) = \top$,
		\item if $\varepsilon_g(j) = \partial$, then $g(p_1,\ldots, p\vee q,\ldots,p_{n_g}) = g(p_1,\ldots, p,\ldots,p_{n_g})\wedge g(p_1,\ldots, q,\ldots,p_{n_g})$; moreover if $g\in \mathcal{G}_n$, then  $g(p_1,\ldots, \bot,\ldots,p_{n_g}) = \top$.
	\end{itemize}
	Each language $\mathcal{L}_\mathrm{LE}$ is interpreted in the appropriate class of LEs. In particular, for every LE $\bba$, each operation $f^\bba\in \mathcal{F}_n^\bbA$ (resp.\ $g^\bba\in \mathcal{G}_n^\bbA$) is finitely join-preserving (resp.\ meet-preserving) in each coordinate when regarded as a map $f^\bba: \bba^{\varepsilon_f}\to \bba$ (resp.\ $g^\bba: \bba^{\varepsilon_g}\to \bba$), and each operation $f^\bba\in \mathcal{F}_r^\bbA$ (resp.\ $g^\bba\in \mathcal{G}_r^\bbA$) preserves nonempty joins (resp.\ nonempty meets) in each coordinate when regarded as a map $f^\bba: \bba^{\varepsilon_f}\to \bba$ (resp.\ $g^\bba: \bba^{\varepsilon_g}\to \bba$).
%Typically, lattice-based logics of this kind are not expressive enough to allow an implication-like term to be defined out of the primitive connectives. Therefore the entailment relation cannot be recovered from the set of tautologies, hence the deducibility has to be defined in terms of sequents. This motivates the following:	
	%It is well known  that the generic LE-logic lacks the deduction theorem. Hence, the consequence relation of these logics cannot be uniformly captured in terms of theorems, but rather in terms of sequents, which motivates the following definition:
	\begin{defi}
		\label{def:DLE:logic:general}
		For any language $\mathcal{L}_\mathrm{LE} = \mathcal{L}_\mathrm{LE}(\mathcal{F}, \mathcal{G})$, the {\em basic}, or {\em minimal} $\mathcal{L}_\mathrm{LE}$-{\em logic} is a set of sequents $\phi\vdash\psi$, with $\phi,\psi\in\mathcal{L}_\mathrm{LE}$, which contains the following axioms:
		\begin{itemize}
			\item Sequents for lattice operations:
			\begin{align*}
				&p\vdash p, && \bot\vdash p, && p\vdash \top, & &  &\\
				&p\vdash p\vee q, && q\vdash p\vee q, && p\wedge q\vdash p, && p\wedge q\vdash q, &
			\end{align*}
			\item Sequents for $f\in \mathcal{F}$ and $g\in \mathcal{G}$:
			\begin{align*}
				&f(p_1,\ldots, p\vee q,\ldots,p_{n_f}) \vdash f(p_1,\ldots, p,\ldots,p_{n_f})\vee f(p_1,\ldots, q,\ldots,p_{n_f}),~\mathrm{for}~ \varepsilon_f(i) = 1,\\
				&f(p_1,\ldots, p\wedge q,\ldots,p_{n_f}) \vdash f(p_1,\ldots, p,\ldots,p_{n_f})\vee f(p_1,\ldots, q,\ldots,p_{n_f}),~\mathrm{for}~ \varepsilon_f(i) = \partial,\\
				& g(p_1,\ldots, p,\ldots,p_{n_g})\wedge g(p_1,\ldots, q,\ldots,p_{n_g})\vdash g(p_1,\ldots, p\wedge q,\ldots,p_{n_g}),~\mathrm{for}~ \varepsilon_g(i) = 1,\\
				& g(p_1,\ldots, p,\ldots,p_{n_g})\wedge g(p_1,\ldots, q,\ldots,p_{n_g})\vdash g(p_1,\ldots, p\vee q,\ldots,p_{n_g}),~\mathrm{for}~ \varepsilon_g(i) = \partial,
            \end{align*}
\item Additional sequents for $f\in \mathcal{F}_n$ and $g\in \mathcal{G}_n$:
			\begin{align*}
				& f(p_1,\ldots, \bot,\ldots,p_{n_f}) \vdash \bot,~\mathrm{for}~ \varepsilon_f(i) = 1,\\
				& f(p_1,\ldots, \top,\ldots,p_{n_f}) \vdash \bot,~\mathrm{for}~ \varepsilon_f(i) = \partial,\\
				&\top\vdash g(p_1,\ldots, \top,\ldots,p_{n_g}),~\mathrm{for}~ \varepsilon_g(i) = 1,\\
				&\top\vdash g(p_1,\ldots, \bot,\ldots,p_{n_g}),~\mathrm{for}~ \varepsilon_g(i) = \partial,\\
			\end{align*}
		\end{itemize}
		and is closed under the following inference rules:
		\begin{displaymath}
			\frac{\phi\vdash \chi\quad \chi\vdash \psi}{\phi\vdash \psi}
			\quad
			\frac{\phi\vdash \psi}{\phi(\chi/p)\vdash\psi(\chi/p)}
			\quad
			\frac{\chi\vdash\phi\quad \chi\vdash\psi}{\chi\vdash \phi\wedge\psi}
			\quad
			\frac{\phi\vdash\chi\quad \psi\vdash\chi}{\phi\vee\psi\vdash\chi}
		\end{displaymath}
		\begin{displaymath}
			 \frac{\phi\vdash\psi}{f(p_1,\ldots,\phi,\ldots,p_n)\vdash f(p_1,\ldots,\psi,\ldots,p_n)}{~(\varepsilon_f(i) = 1)}
		\end{displaymath}
		\begin{displaymath}
			 \frac{\phi\vdash\psi}{f(p_1,\ldots,\psi,\ldots,p_n)\vdash f(p_1,\ldots,\phi,\ldots,p_n)}{~(\varepsilon_f(i) = \partial)}
		\end{displaymath}
		\begin{displaymath}
			 \frac{\phi\vdash\psi}{g(p_1,\ldots,\phi,\ldots,p_n)\vdash g(p_1,\ldots,\psi,\ldots,p_n)}{~(\varepsilon_g(i) = 1)}
		\end{displaymath}
		\begin{displaymath}
			 \frac{\phi\vdash\psi}{g(p_1,\ldots,\psi,\ldots,p_n)\vdash g(p_1,\ldots,\phi,\ldots,p_n)}{~(\varepsilon_g(i) = \partial)}.
		\end{displaymath}
		The minimal LE-logic is denoted by $\mathbf{L}_\mathrm{LE}$. For any LE-language $\mathcal{L}_{\mathrm{LE}}$, by an {\em $\mathrm{LE}$-logic} we understand any axiomatic extension of the basic $\mathcal{L}_{\mathrm{LE}}$-logic in $\mathcal{L}_{\mathrm{LE}}$.
	\end{defi}
	
	For every LE $\bba$, the symbol $\vdash$ is interpreted as the lattice order $\leq$. A sequent $\phi\vdash\psi$ is valid in $\bba$ if $h(\phi)\leq h(\psi)$ for every homomorphism $h$ from the $\mathcal{L}_\mathrm{LE}$-algebra of formulas over $\mathsf{PROP}$ to $\bba$. The notation $\mathbb{LE}\models\phi\vdash\psi$ indicates that $\phi\vdash\psi$ is valid in every LE. Then, by means of a routine Lindenbaum-Tarski construction, it can be shown that the minimal LE-logic $\mathbf{L}_\mathrm{LE}$ is sound and complete with respect to its correspondent class of algebras $\mathbb{LE}$, i.e.\ that any sequent $\phi\vdash\psi$ is provable in $\mathbf{L}_\mathrm{LE}$ iff $\mathbb{LE}\models\phi\vdash\psi$. %Moreover, it is not hard to see that every consistent DLE-logic is characterized by the class of algebras for it.

\begin{defi}\label{def:normalization}
For every $\mathcal{L}_{\mathrm{LE}}$-algebra $\bba$ and all $f\in \mathcal{F}_r$ and $g\in \mathcal{G}_r$, the {\em normalizations} of $f^\bba$ and $g^\bba$ are the operations defined as follows: if $\epsilon_f = \epsilon_g = 1$,
\begin{center}
\begin{tabular}{c c c}
$\Diamond_f(u):  =\begin{cases}
f(u) & \mbox{if } u\neq \bot\\
\bot & \mbox{if } u = \bot\\
\end{cases}
$
& $\quad\quad$&
$\Box_g(u): = \begin{cases}
g(u) & \mbox{if } u\neq \top\\
\top & \mbox{if } u = \top\\
\end{cases}$\\
\end{tabular}
\end{center}
if $\epsilon_f = \epsilon_g = \partial$,

\begin{center}
\begin{tabular}{c c c}
${\lhd}_f(u):  =\begin{cases}
f(u) & \mbox{if } u\neq \top\\
\bot & \mbox{if } u = \top\\
\end{cases}
$
& $\quad\quad$&
${\rhd}_g(u): = \begin{cases}
g(u) & \mbox{if } u\neq \bot\\
\top & \mbox{if } u = \bot\\
\end{cases}$\\
\end{tabular}
\end{center}
\end{defi}

\begin{lem}\label{lemma:LE algebras interpret normalizations}
For every $\mathcal{L}_{\mathrm{LE}}$-algebra $\bba$ and all $f\in \mathcal{F}_r$ and $g\in \mathcal{G}_r$,
\begin{enumerate}
\item if $\epsilon_f = 1$ then $\Diamond_f$ preserves finite (hence also empty) joins and $f(u) = f(\bot)\vee \Diamond_fu$ for every $u\in \bba$;
\item if $\epsilon_g = 1$ then $\Box_g$ preserves finite (hence also empty) meets and $g(u) = g(\top)\wedge \Box_gu$ for every $u\in \bba$;
\item if $\epsilon_f = \partial$ then ${\lhd}_f$ reverses finite (hence also empty) meets and $f(u) = f(\top)\vee {\lhd}_fu$ for every $u\in \bba$;
\item if $\epsilon_g = \partial$ then ${\rhd}_g$ reverses finite (hence also empty) joins and $g(u) = g(\bot)\wedge {\rhd}_gu$ for every $u\in \bba$.
\end{enumerate}
\end{lem}
\begin{proof}
1. If $u\vee v = \bot$, then $u =  \bot = v$, and  the claim immediately follows by definition of $\Diamond_f$. If $u\vee v \neq \bot$, then $\Diamond_f(u\vee v) = f(u\vee v) = f(u)\vee f(v)$. If $u \neq  \bot \neq v$ then the claim immediately follows by definition of $\Diamond_f$. If $u  =  \bot$ and $v\neq \bot$, then $u\leq v$ and the claim  follows by definition of $\Diamond_f$ and the monotonicity of $f$. Analogously if  $u \neq \bot$ and $v\neq \bot$. The second part of the claim immediately follows from the definition of $\Diamond_f$. The remaining items are order-variants and their proof is omitted.
\end{proof}
	
	\begin{rem}\label{rem:compact:nrmalized:ops}
		Using the formulation adumbrated in Remarks \ref{rem:cmpctpres:signatures} and \ref{rem:cmpctpres:LE:Alg}, Definition \ref{def:normalization} can be equivalently given by specifying that, for all regular connectives $h \in \mathcal{H}_r$, the normalization of $h^{\mathbb{A}}$ is the operation $\heartsuit_h$ given by 
		\[
		\heartsuit_h(u) = \left\{%
		\begin{array}{ll}
		h(u) &u \neq (\bot^{\alpha_h})^{\epsilon_h(1)} \\
		\bot^{\alpha_h} &u = (\bot^{\alpha_h})^{\epsilon_h(1)}
		\end{array} 
		\right.
		\]
		Using this formulation, Lemma \ref{lemma:LE algebras interpret normalizations} can be given by stating that, for all finite, possibly empty subsets $S$ of elements of an LE-algebra $\mathbb{A}$, $\heartsuit_h((\bigvee^{\alpha_h})^{\epsilon_h(1)} S) = \bigvee^{\alpha_h}_{s \in S} \heartsuit_h(s)$, and moreover $h(u) = h((\bot^{\alpha_h})^{\epsilon_h(1)}) \vee^{\alpha_h} \heartsuit_h(u)$.
	\end{rem}

	\subsection{The expanded language $\mathcal{L}_\mathrm{LE}^*$}
	\label{ssec:expanded tense language}
	We now introduce an expansion of LE-languages, which adds connectives intended to be interpreted as the normalized counterparts of all regular connectives, as well as connectives to be interpreted as the left and right residuals of interpretations of all normal connectives, including the introduced normalized counterparts.  Formally, any given language $\mathcal{L}_\mathrm{LE} = \mathcal{L}_\mathrm{LE}(\mathcal{F}, \mathcal{G})$ can be associated with the language $\mathcal{L}_\mathrm{LE}^* = \mathcal{L}_\mathrm{LE}(\mathcal{F}^*, \mathcal{G}^*)$, where $\mathcal{F}^*\supseteq \mathcal{F}$ and $\mathcal{G}^*\supseteq \mathcal{G}$ are obtained by expanding $\mathcal{L}_\mathrm{LE}$ in two steps as follows: as to the first step, let $\mathcal{F}'\supseteq \mathcal{F}_n$ and $\mathcal{G}'\supseteq \mathcal{G}_n$ be obtained by adding:
\begin{enumerate}
		\item for each $f\in\mathcal{F}_r$ s.t.\ $\varepsilon_f = 1$, the unary connective $\Diamond_f$ with $\varepsilon_{\Diamond_f} = 1$, the intended interpretation of which is the normalization of $f$ (cf.\ Definition \ref{def:normalization});
\item for each $f\in\mathcal{F}_r$ s.t.\ $\varepsilon_f = \partial$, the unary connective ${\lhd}_f$ with $\varepsilon_{{\lhd}_f} = \partial$, the intended interpretation of which is the normalization of $f$ (cf.\ Definition \ref{def:normalization});
		\item for each $g\in\mathcal{G}_r$ s.t.\ $\varepsilon_g = 1$, the unary connective $\Box_g$ with $\varepsilon_{{\Box}_g} = 1$, the intended interpretation of which is the normalization of $g$ (cf.\ Definition \ref{def:normalization});
\item for each $g\in\mathcal{G}_r$ s.t.\ $\varepsilon_g = \partial$, the unary connective ${\rhd}_g$ with $\varepsilon_{{\rhd}_g} = \partial$, the intended interpretation of which is the normalization of $g$ (cf.\ Definition \ref{def:normalization}).
	\end{enumerate}
As to the second step, let $\mathcal{F}_n^*\supseteq \mathcal{F}'$ and $\mathcal{G}_n^*\supseteq \mathcal{G}'$ be obtained by adding:
	\begin{enumerate}
		\item for $f\in\mathcal{F}'$ and  $0\leq i\leq n_f$, the $n_f$-ary connective $f^\sharp_i$, the intended interpretation of which is the right residual of $f\in\mathcal{F}_n$ in its $i$th coordinate if $\varepsilon_f(i) = 1$ (resp.\ its Galois-adjoint if $\varepsilon_f(i) = \partial$);
		\item for $g\in\mathcal{G}'$ and  $0\leq i\leq n_g$, the $n_g$-ary connective $g^\flat_i$, the intended interpretation of which is the left residual of $g\in\mathcal{G}_n$ in its $i$th coordinate if $\varepsilon_g(i) = 1$ (resp.\ its Galois-adjoint if $\varepsilon_g(i) = \partial$).
		% $ g^\flat_j$ for each and $g\in \mathcal{G}$, where and $0\leq j\leq n_g$ ($f^\sharp_i$ is the right residual of $f$ in the $i$-th coordinate, and $g^\flat_j$ is the left residual of $g$ in the $j$-th coordinate).
		\footnote{We reserve the symbols  $\Diamond$ and $\lhd$ to  denote unary connectives in $\mathcal{F}_n$ such that $\varepsilon_{\Diamond} = 1$ and $\varepsilon_{\lhd} = \partial$ and  $\Box$ and $\rhd$ to denote unary connectives in $\mathcal{G}_n$ such that $\varepsilon_{\Box} = 1$ and $\varepsilon_{\rhd} = \partial$. The adjoints of $\Box$, $\Diamond$, $\lhd$ and $\rhd$ are denoted $\Diamondblack$, $\blacksquare$, $\blhd$ and $\brhd$, respectively. For every $f\in \mathcal{F}_r$ and $g\in \mathcal{G}_r$, we let $\blacksquare_f$ and $\Diamondblack_g$ denote the right and left adjoint of $\Diamond_f$ and $\Box_g$ respectively if $\varepsilon_f = \varepsilon_g = 1$, and ${\blacktriangleleft}_f$ and ${\blacktriangleright}_g$ denote the Galois-adjoints of ${\lhd}_f$ and ${\rhd}_g$ if $\varepsilon_f = \varepsilon_g = \partial$.}
	\end{enumerate}
	We stipulate that
	$f^\sharp_i\in\mathcal{G}_n^*$ if $\varepsilon_f(i) = 1$, and $f^\sharp_i\in\mathcal{F}_n^*$ if $\varepsilon_f(i) = \partial$. Dually, $g^\flat_i\in\mathcal{F}_n^*$ if $\varepsilon_g(i) = 1$, and $g^\flat_i\in\mathcal{G}_n^*$ if $\varepsilon_g(i) = \partial$. The order-type assigned to the additional connectives is predicated on the order-type of their intended interpretations. That is, for any $f\in \mathcal{F}'$ and $g\in\mathcal{G}'$,
	%each $g^\flat_j\in\mathcal{F}$, for each coordinate $i$ in $f$ or $g$,
	\begin{enumerate}
		\item if $\epsilon_f(i) = 1$, then $\epsilon_{f_i^\sharp}(i) = 1$ and $\epsilon_{f_i^\sharp}(j) = (\epsilon_f(j))^\partial$ for any $j\neq i$.
		\item if $\epsilon_f(i) = \partial$, then $\epsilon_{f_i^\sharp}(i) = \partial$ and $\epsilon_{f_i^\sharp}(j) = \epsilon_f(j)$ for any $j\neq i$.
		\item if $\epsilon_g(i) = 1$, then $\epsilon_{g_i^\flat}(i) = 1$ and $\epsilon_{g_i^\flat}(j) = (\epsilon_g(j))^\partial$ for any $j\neq i$.
		\item if $\epsilon_g(i) = \partial$, then $\epsilon_{g_i^\flat}(i) = \partial$ and $\epsilon_{g_i^\flat}(j) = \epsilon_g(j)$ for any $j\neq i$.
	\end{enumerate}
	Finally,\footnote{\label{footnote: notation residuals} For instance, if $f$ and $g$ are binary connectives such that $\varepsilon_f = (1, \partial)$ and $\varepsilon_g = (\partial, 1)$, then $\varepsilon_{f^\sharp_1} = (1, 1)$, $\varepsilon_{f^\sharp_2} = (1, \partial)$, $\varepsilon_{g^\flat_1} = (\partial, 1)$ and $\varepsilon_{g^\flat_2} = (1, 1)$.Warning: notice that this notation heavily depends from the connective which is taken as primitive, and needs to be carefully adapted to well known cases. For instance, consider the  `fusion' connective $\circ$ (which, when denoted  as $f$, is such that $\varepsilon_f = (1, 1)$). Its residuals
$f_1^\sharp$ and $f_2^\sharp$ are commonly denoted $/$ and
$\backslash$ respectively. However, if $\backslash$ is taken as the primitive connective $g$, then $g_2^\flat$ is $\circ = f$, and
$g_1^\flat(x_1, x_2): = x_2/x_1 = f_1^\sharp (x_2, x_1)$. This example shows
that, when identifying $g_1^\flat$ and $f_1^\sharp$, the conventional order of the coordinates is not preserved, and depends of which connective
is taken as primitive.}
	let $\mathcal{F}^*:  = \mathcal{F}_n^*\uplus \mathcal{F}_r$ and $\mathcal{G}^*:  = \mathcal{G}_n^*\uplus \mathcal{G}_r$.
	
\begin{rem}\label{rem:cmpct:extended:language}
	Following Remarks \ref{rem:cmpctpres:signatures}, \ref{rem:cmpctpres:LE:Alg} and \ref{rem:compact:nrmalized:ops}, the language $\mathcal{L}_{LE}$ can alternatively be defined as follows: let $\mathcal{H}'$ be obtained from $\mathcal{H}$ by adding $\heartsuit_h$ for each $h \in \mathcal{H}_r$, with intended interpretation as given in Remark \ref{rem:compact:nrmalized:ops}. Next, obtain $\mathcal{H}^{*}$ by adding: (1) for each $h \in \mathcal{H}'$ with $\alpha_{h} = 1$, the $n_h$-ary connectives $h_i^{\sharp}$ for $1 \leq i \leq n_{h}$ where $\alpha_{h_i^{\sharp}} = (\alpha^{\partial}_h)^{\epsilon_{h}(i)}$ and $\epsilon_{h_i^{\sharp}}(i) = \epsilon_{h}(i)$ while $\epsilon_{h_i^{\sharp}}(j) = ((\epsilon_{h}(j))^{\partial})^{\epsilon_{h}(i)}$ for $j \neq i$; (2) for each $h \in \mathcal{H}'$ with $\alpha_{h} = \partial$, the $n_h$-ary connectives $h_i^{\flat}$ for $1 \leq i \leq n_{h}$ where $\alpha_{h_i^{\flat}} = (\alpha^{\partial}_h)^{\epsilon_{h}(i)}$ and $\epsilon_{h_i^{\flat}}(i) = \epsilon_{h}(i)$ while $\epsilon_{h_i^{\flat}}(j) = ((\epsilon_{h}(j))^{\partial})^{\epsilon_{h}(i)}$ for $j \neq i$.
\end{rem}
	
\begin{defi}
		For any language $\mathcal{L}_\mathrm{LE}(\mathcal{F}, \mathcal{G})$, the {\em basic expanded} $\mathcal{L}_\mathrm{LE}$-{\em logic} is defined by specializing Definition \ref{def:DLE:logic:general} to the language $\mathcal{L}_\mathrm{LE}^* = \mathcal{L}_\mathrm{LE}(\mathcal{F}^*, \mathcal{G}^*)$ %a set of sequents $\phi\vdash\psi$ with $\phi,\psi\in\mathcal{L}_\mathrm{DLE}^*$, which contains the axioms of the DLE-logic $\mathbf{L}_\mathrm{DLE}$, and is closed under rules for DLE-logics plus
		and closing under the following residuation rules for each $f\in \mathcal{F}'$ and $g\in \mathcal{G}'$:
			$$
			\begin{array}{cc}
			\AxiomC{$f(\varphi_1,\ldots,\phi,\ldots, \varphi_{n_f}) \vdash \psi$}
			\doubleLine
			\LeftLabel{$(\epsilon_f(i) = 1)$}
			\UnaryInfC{$\phi\vdash f^\sharp_i(\varphi_1,\ldots,\psi,\ldots,\varphi_{n_f})$}
			\DisplayProof
			&
			\AxiomC{$\phi \vdash g(\varphi_1,\ldots,\psi,\ldots,\varphi_{n_g})$}
			\doubleLine
			\RightLabel{$(\epsilon_g(i) = 1)$}
			\UnaryInfC{$g^\flat_i(\varphi_1,\ldots, \phi,\ldots, \varphi_{n_g})\vdash \psi$}
			\DisplayProof
			\end{array}
			$$
			$$
			\begin{array}{cc}
			\AxiomC{$f(\varphi_1,\ldots,\phi,\ldots, \varphi_{n_f}) \vdash \psi$}
			\doubleLine
			\LeftLabel{$(\epsilon_f(i) = \partial)$}
			 \UnaryInfC{$f^\sharp_i(\varphi_1,\ldots,\psi,\ldots,\varphi_{n_f})\vdash \phi$}
			\DisplayProof
			&
			\AxiomC{$\phi \vdash g(\varphi_1,\ldots,\psi,\ldots,\varphi_{n_g})$}
			\doubleLine
			\RightLabel{($\epsilon_g(i) = \partial)$}
			\UnaryInfC{$\psi\vdash g^\flat_i(\varphi_1,\ldots, \phi,\ldots, \varphi_{n_g})$}
			\DisplayProof
			\end{array}
			$$
		The double line in each rule above indicates that the rule should be read both top-to-bottom and bottom-to-top.
		Let $\mathbf{L}_\mathrm{LE}^*$ be the minimal expanded $\mathcal{L}_\mathrm{LE}$-logic. %\footnote{\label{ftn: definition basic logic for expanded language} Hence, for any language $\mathcal{L}_\mathrm{DLE}$, there are in principle two logics associated with the expanded language $\mathcal{L}_\mathrm{DLE}^*$, namely the {\em minimal} $\mathcal{L}_\mathrm{DLE}^*$-logic, which we denote by $\mathbf{L}_\mathrm{DLE}^{\underline{*}}$, and which is obtained by instantiating Definition \ref{def:DLE:logic:general} to the language $\mathcal{L}_\mathrm{DLE}^*$, and the bi-intuitionistic `tense' logic $\mathbf{L}_\mathrm{DLE}^*$, defined above. The logic $\mathbf{L}_\mathrm{DLE}^*$ is the natural logic on the language $\mathcal{L}_\mathrm{DLE}^*$, however it is useful to introduce a specific notation for $\mathbf{L}_\mathrm{DLE}^{\underline{*}}$, given that all the results holding for the minimal logic associated with an arbitrary DLE-language can be instantiated to the expanded language $\mathcal{L}_\mathrm{DLE}^*$ and will then apply to $\mathbf{L}_\mathrm{DLE}^{\underline{*}}$.}
For any language $\mathcal{L}_{\mathrm{LE}}$, by a {\em expanded $\mathrm{LE}$-logic} we understand any axiomatic extension of the basic expanded  $\mathcal{L}_{\mathrm{LE}}$-logic in $\mathcal{L}^*_{\mathrm{LE}}$.
	\end{defi}
	
	The algebraic semantics of $\mathbf{L}_\mathrm{LE}^*$ is given by the class of $\mathcal{L}_\mathrm{LE}$-algebras defined as tuples $\bba = (L, \mathcal{F}^*, \mathcal{G}^*)$ such that $L$ is a lattice, and moreover,
		\begin{enumerate}
			
			\item for every $f\in \mathcal{F}'$ s.t.\ $n_f\geq 1$, all $a_1,\ldots,a_{n_f}\in L$ and $b\in L$, and each $1\leq i\leq n_f$,
			\begin{itemize}
				\item
				if $\epsilon_f(i) = 1$, then $f(a_1,\ldots,a_i,\ldots a_{n_f})\leq b$ iff $a_i\leq f^\sharp_i(a_1,\ldots,b,\ldots,a_{n_f})$;
				\item
				if $\epsilon_f(i) = \partial$, then $f(a_1,\ldots,a_i,\ldots a_{n_f})\leq b$ iff $a_i\leq^\partial f^\sharp_i(a_1,\ldots,b,\ldots,a_{n_f})$.
			\end{itemize}
			\item for every $g\in \mathcal{G}'$ s.t.\ $n_g\geq 1$, any $a_1,\ldots,a_{n_g}\in L$ and $b\in L$, and each $1\leq i\leq n_g$,
			\begin{itemize}
				\item if $\epsilon_g(i) = 1$, then $b\leq g(a_1,\ldots,a_i,\ldots a_{n_g})$ iff $g^\flat_i(a_1,\ldots,b,\ldots,a_{n_g})\leq a_i$.
				\item
				if $\epsilon_g(i) = \partial$, then $b\leq g(a_1,\ldots,a_i,\ldots a_{n_g})$ iff $g^\flat_i(a_1,\ldots,b,\ldots,a_{n_g})\leq^\partial a_i$.
			\end{itemize}
		\end{enumerate}
		It is also routine to prove using the Lindenbaum-Tarski construction that $\mathbf{L}_\mathrm{LE}^*$ (as well as any of its sound axiomatic extensions) is sound and complete w.r.t.\ the class of  expanded $\mathcal{L}_\mathrm{LE}$-algebras (w.r.t.\ the suitably defined equational subclass, respectively). % and that every consistent DLE$^*$-logic is characterized by its algebras.

\subsection{Canonical extensions, constructively}
Canonical extensions provide a purely algebraic encoding of Stone-type dualities, and indeed, the existence of the canonical extensions of  the best-known varieties of LEs can be proven via  preexisting dualities. However, alternative, purely algebraic constructions are available, such as those of \cite{GeHa01,DunnGP05}. These constructions are in fact more general, in that their definition does not rely on principles such as Zorn's lemma. In what follows we will adapt them to the setting of LEs introduced above.
%The way the algebraic and the relational semantics of any classical modal logic are linked to one another is very well known: every BAO can be associated with its ultrafilter frame, and with every Kripke frame is associated its complex algebra. To close this triangle, the J\'onsson-Tarski expansion of Stone representation theorem states that every BAO $\bba$ canonically embeds in the complex algebra of its ultrafilter frame. This complex algebra, which is called the {\em perfect}, or {\em canonical extension} of $\bba$, has several additional properties, both intrinsic to it (for instance, it is a {\em powerset} algebra, and not just an algebra of sets) and also relative to its embedded subalgebra. These properties can be expressed purely algebraically, hence independently of the ultrafilter frame construction, and characterize the canonical extension up to an isomorphism fixing the embedded algebra. Analogously well behaved constructions can be performed also for normal (distributive) lattice expansions ((D)LEs), of which we will not give a full account here. Interestingly, whereas the counterparts of the ultrafilter frames look rather different from their Boolean versions, the canonical extension of an LE is defined exactly as the one of a BAO, and the definition is based on the following:
\begin{defi}\label{Def:Canon:Ext}
Let $\bba$ be a (bounded) sublattice of a complete lattice $\bba'$.
\begin{enumerate}
\item  $\bba$ is {\em dense} in $\bba'$ if every element of $\bba'$ can be expressed both as a join of meets and as
a meet of joins of elements from $\bba$.
\item $\bba$ is {\em compact} in $\bba'$ if, for all $S, T \subseteq \bba$, if $\bigwedge S\leq \bigvee T$ then $\bigwedge S'\leq \bigvee T'$ for some finite $S'\subseteq S$ and $T'\subseteq T$.
\item The {\em canonical extension} of a lattice $\bba$ is a complete lattice $\bbas$ containing $\bba$
as a dense and compact sublattice.
\end{enumerate}
\end{defi}
 Let $\kbbas$ and $\obbas$ denote the meet-closure and the join-closure of $\bba$ in $\bbas$ respectively.  The elements of $\kbbas$ are referred to as {\em closed} elements, and elements of $\obbas$ as {\em open} elements.

\begin{thm}[Propositions 2.6 and 2.7 in \cite{GeHa01}]
The canonical extension of a bounded lattice $\bba$ exists and is unique up to any isomorphism fixing $\bba$.
\end{thm}
\begin{proof}
We expand on the existence, since it is relevant to the present paper. Let $\mathbf{I}$ and $\mathbf{F}$ be the collections of the ideals and filters of $\bba$ respectively. Consider the polarity $(\mathbf{F}, \mathbf{I}, \leq)$, where $F\leq I$ iff $F\cap I \neq \varnothing$ for every $F\in \mathbf{F}$ and $I\in \mathbf{I}$. As is well known (cf.\ \cite{DaPr90}), this polarity induces a Galois connection $(u: \mathcal{P}(\mathbf{F})\to \mathcal{P}(\mathbf{I}), \ell: \mathcal{P}(\mathbf{I})\to \mathcal{P}(\mathbf{F}))$, with  $u$ and $\ell$ defined by the  assignments $\mathcal{X}\mapsto \{I \mid F\leq I \mbox{ for all } F\in \mathcal{X}\}$, and $\mathcal{Y}\mapsto \{F \mid F\leq I \mbox{ for all } I\in \mathcal{Y}\}$, respectively. Hence the maps $\ell\circ u$ and $u\circ \ell$ are closure operators on $\mathcal{P}(\mathbf{F})$ and $\mathcal{P}(\mathbf{I})$ respectively. The collections of Galois-stable sets of $\ell\circ u$ and $u\circ \ell$ form complete $\bigcap$-subsemilattices $\mathcal{G}_F$ and $\mathcal{G}_I$ of $\mathcal{P}(\mathbf{F})$ and $\mathcal{P}(\mathbf{I})$ respectively. These semilattices  are then complete lattices, and are dually order-isomorphic to each other via the appropriate restrictions of $u$ and $\ell$. The maps $\bba\to \mathcal{G}_F$ and $\bba\to \mathcal{G}_I$ defined by the assignments $a\mapsto \ell \circ u(a{\uparrow})$ and $a\mapsto u\circ \ell(a{\downarrow})$ are dense and compact order-embeddings of $\bba$.
\end{proof}
In meta-theoretic settings in which Zorn's lemma is available, the fact that $\mathbf{F}$ and $\mathbf{I}$ are closed under taking unions of $\subseteq$-chains guarantees that the canonical extension of a lattice $\bba$ is a {\em perfect} lattice. That is, in addition to being complete,
%
%Early on, we mentioned the   intrinsic special properties of canonical extensions. Just like the canonical extension of a BA can be shown to be a {\em perfect} BA (i.e.\ a BA  isomorphic to the powerset algebra of some set), the canonical extension of any lattice is a perfect lattice \cite[Corollary 2.10]{DGP}:
%\begin{defi}
%\label{def: perfect lattice}
%A  lattice $\bba$ is {\em perfect} if %it satisfies one of the following, equivalent, conditions:
%\begin{enumerate}
%\item
%$\bba$ is complete, and
is both completely join-generated by the set $\jir$ of the completely
join-irreducible elements of $\bbas$, and completely meet-generated by the set $\mir$ of
the completely meet-irreducible elements of $\bbas$. In our present, constructive setting, canonical extensions are not perfect in general, since in general they do not have `enough' join-irreducibles and meet-irreducibles, as specified above.
%\item  $\bba$ is isomorphic to the complete lattice of the Galois-stable sets of some RS-frame.
%\end{enumerate}
%\end{defi}

The canonical extension of an LE $\bba$ will be defined as a suitable expansion of the canonical extension of the underlying lattice of $\bba$.
Before turning to this definition, recall that
taking the canonical extension of a lattice  commutes with
taking order duals and products, namely:
${(\bba^\partial)}^\delta = {(\bbas)}^\partial$ and ${(\bba_1\times \bba_2)}^\delta = \bba_1^\delta\times \bba_2^\delta$ (cf.\ \cite[Theorem 2.8]{DunnGP05}).
Hence,  ${(\bba^\partial)}^\delta$ can be  identified with ${(\bbas)}^\partial$,  ${(\bba^n)}^\delta$ with ${(\bbas)}^n$, and
${(\bba^\varepsilon)}^\delta$ with ${(\bbas)}^\varepsilon$ for any order type $\varepsilon$. Thanks to these identifications,
in order to extend operations of any arity and which are monotone or antitone in each coordinate from a lattice $\bba$ to its canonical extension, treating the case
of {\em monotone} and {\em unary} operations suffices:
\begin{defi}
For every unary, order-preserving operation $f : \bba \to \bba$, the $\sigma$-{\em extension} of $f$ is defined firstly by declaring, for every $k\in \kbbas$,
$$f^\sigma(k):= \bigwedge\{ f(a)\mid a\in \bba\mbox{ and } k\leq a\},$$ and then, for every $u\in \bbas$,
$$f^\sigma(u):= \bigvee\{ f^\sigma(k)\mid k\in \kbbas\mbox{ and } k\leq u\}.$$
The $\pi$-{\em extension} of $f$ is defined firstly by declaring, for every $o\in \obbas$,
$$f^\pi(o):= \bigvee\{ f(a)\mid a\in \bba\mbox{ and } a\leq o\},$$ and then, for every $u\in \bbas$,
$$f^\pi(u):= \bigwedge\{ f^\pi(o)\mid o\in \obbas\mbox{ and } u\leq o\}.$$
\end{defi}
Note that the operations $\sigma$ and $\pi$ are order dual. It is easy to see that the $\sigma$- and $\pi$-extensions of $\varepsilon$-monotone maps are $\varepsilon$-monotone. More remarkably,the $\sigma$-extension of a map which sends finite (resp.\ finite nonempty) joins or meets in the domain to finite (resp.\ finite nonempty) joins in the codomain sends {\em arbitrary} (resp.\ {\em arbitrary} nonempty) joins
or meets in the domain to {\em arbitrary} (resp.\ {\em arbitrary} nonempty) joins in the codomain. Dually, the $\pi$-extension of a map which sends finite (resp.\ finite nonempty) joins or meets in the domain to finite (resp.\ finite nonempty) meets in the
codomain sends {\em arbitrary} (resp.\ {\em arbitrary} nonempty) joins
or meets in the domain to {\em arbitrary} (resp.\ {\em arbitrary} nonempty) meets in the codomain (cf.\ \cite[Lemma 4.6]{GeHa01}; notice that the proof given there holds in a constructive meta-theory).
Therefore, depending on the properties of the original operation, it is more convenient to use one or the other extension. This justifies the following
\begin{defi}
%For every LE $\bba = (\bbb, \mathcal{F}, \mathcal{G})$, the {\em canonical extension} of $\bba$
The canonical extension of an
$\mathcal{L}_\mathrm{LE}$-algebra $\bbA = (L, \mathcal{F}^\bbA, \mathcal{G}^\bbA)$ is the   $\mathcal{L}_\mathrm{LE}$-algebra
$\bbA^\delta: = (L^\delta, \mathcal{F}^{\bbA^\delta}, \mathcal{G}^{\bbA^\delta})$ such that $f^{\bbA^\delta}$ and $g^{\bbA^\delta}$ are defined as the
$\sigma$-extension of $f^{\bbA}$ and as the $\pi$-extension of $g^{\bbA}$ respectively, for all $f\in \mathcal{F}$ and $g\in \mathcal{G}$.
\end{defi}
The canonical extension of an LE $\bba$ is a {\em quasi-perfect} LE:
\begin{defi}
\label{def:perfect LE}
An LE $\bbA = (L, \mathcal{F}^\bbA, \mathcal{G}^\bbA)$ is quasi-perfect if $L$ is a complete lattice, %(cf.\ Definition \ref{def: perfect lattice}),
 the following infinitary distribution laws are satisfied for each $f\in \mathcal{F}_n$, $g\in \mathcal{G}_n$, $1\leq i\leq n_f$ and $1\leq j\leq n_g$: for every $S\subseteq L$,
\begin{center}
\begin{tabular}{c c }
$f(x_1,\ldots, \bigvee S, \ldots, x_{n_f}) =\bigvee \{ f(x_1,\ldots, x, \ldots, x_{n_f}) \mid x\in S \}$  & if $\varepsilon_f(i) = 1$\\

$f(x_1,\ldots, \bigwedge S, \ldots, x_{n_f}) =\bigvee \{ f(x_1,\ldots, x, \ldots, x_{n_f}) \mid x\in S \}$  & if $\varepsilon_f(i) = \partial$\\

$g(x_1,\ldots, \bigwedge S, \ldots, x_{n_g}) =\bigwedge \{ g(x_1,\ldots, x, \ldots, x_{n_f}) \mid x\in S \}$  & if $\varepsilon_g(i) = 1$\\

$f(x_1,\ldots, \bigvee S, \ldots, x_{n_g}) =\bigwedge \{ g(x_1,\ldots, x, \ldots, x_{n_f}) \mid x\in S \}$  & if $\varepsilon_g(i) = \partial$,\\

\end{tabular}
\end{center}
and analogous identities hold for every $f\in \mathcal{F}_r$ and  $g\in \mathcal{F}_r$,  restricted to $S\neq \varnothing$.
\end{defi}
Before finishing the present subsection, let us spell out and further simplify the definitions of the extended operations.
First of all, we recall that taking  order-duals interchanges closed and open elements:
$K({(\bbas)}^\partial) = O(\bbas)$ and $O({(\bbas)}^\partial) =\kbbas$;  similarly, $K({(\bba^n)}^\delta) =\kbbas^n$, and $O({(\bba^n)}^\delta) =\obbas^n$. Hence,  $K({(\bbas)}^\epsilon) =\prod_i K(\bbas)^{\epsilon(i)}$ and $O({(\bbas)}^\epsilon) =\prod_i O(\bbas)^{\epsilon(i)}$ for every LE $\bba$ and every order-type $\epsilon$ on any $n\in \mathbb{N}$, where
\begin{center}
\begin{tabular}{cc}
$K(\bbas)^{\epsilon(i)}: =\begin{cases}
K(\bbas) & \mbox{if } \epsilon(i) = 1\\
O(\bbas) & \mbox{if } \epsilon(i) = \partial\\
\end{cases}
$ &
$O(\bbas)^{\epsilon(i)}: =\begin{cases}
O(\bbas) & \mbox{if } \epsilon(i) = 1\\
K(\bbas) & \mbox{if } \epsilon(i) = \partial.\\
\end{cases}
$\\
\end{tabular}
\end{center}
Denoting by $\leq^\epsilon$ the product order on $(\bbas)^\epsilon$, we have for every $f\in \mathcal{F}$, $g\in \mathcal{G}$,  $\overline{u}\in (\bbas)^{n_f}$ and $\overline{v}\in (\bbas)^{n_g}$,
\begin{align*}
  f^\sigma (\overline{k}) & := \bigwedge\{ f( \overline{a})\mid \overline{a}\in (\bbas)^{\epsilon_f}\mbox{ and } \overline{k}\leq^{\epsilon_f} \overline{a}\} \\
  f^\sigma (\overline{u}) & := \bigvee\{ f^\sigma( \overline{k})\mid \overline{k}\in K({(\bbas)}^{\epsilon_f})\mbox{ and } \overline{k}\leq^{\epsilon_f} \overline{u}\} \\
  g^\pi (\overline{o}) & := \bigvee\{ g( \overline{a})\mid \overline{a}\in (\bbas)^{\epsilon_g}\mbox{ and } \overline{a}\leq^{\epsilon_g} \overline{o}\}\\
  g^\pi (\overline{v}) & := \bigwedge\{ g^\pi( \overline{o})\mid \overline{o}\in O({(\bbas)}^{\epsilon_g})\mbox{ and } \overline{v}\leq^{\epsilon_g} \overline{o}\}
%
%\begin{tabular}{cc}
%$\Diamond^\sigma u:= \bigvee\{ \Diamond k\mid k\in \kbbas\mbox{ and } k\leq u\}$ & $\Box^\pi u:= \bigwedge\{ \Box o\mid o\in \obbas\mbox{ and } u\leq o\}$\\
%${\lhd}^\sigma u:= \bigvee\{ {\lhd}o\mid o\in \obbas\mbox{ and } u\leq o\}$ & ${\rhd}^\pi u:= \bigwedge\{ {\rhd}k\mid k\in \kbbas\mbox{ and } k\leq u\}$\\
%$u\circ^\sigma v:= \bigvee\{ k\circ k'\mid k, k'\in \kbbas, k\leq u \mbox{ and } k'\leq v\}$ & $u\star^\pi v:= \bigwedge\{o\star o'\mid o, o'\in \obbas, u\leq o \mbox{ and } v\leq o'\}$.\\
%\end{tabular}
\end{align*}

Notice that the algebraic completeness of the logics $\mathbf{L}_\mathrm{LE}$ and $\mathbf{L}_\mathrm{LE}^*$ and the canonical embedding of LEs into their canonical extensions immediately give completeness of $\mathbf{L}_\mathrm{LE}$ and $\mathbf{L}_\mathrm{LE}^*$ w.r.t.\ the appropriate class of perfect LEs.
		%\end{comment}

\subsection{Constructive canonical extensions are natural $\mathcal{L}_\mathrm{LE}^*$-algebras}

The aim of the present subsection is showing that the constructive canonical extension of any  $\mathcal{L}_{\mathrm{LE}}$-algebra $\bba$ supports the interpretation of  the  language $\mathcal{L}_\mathrm{LE}^*$ (cf.\ Section \ref{ssec:expanded tense language}). This will be done in two steps: %quasi-perfect
Firstly, we need to verify that taking the  normalization of any $f\in \mathcal{F}_r^{\bbas}$ and $g\in \mathcal{G}_r^{\bbas}$ commutes with taking  canonical extensions:

%\marginnote{edit here}
\begin{lem}\label{lemma: f and diamond f coincide on non bottom els}
For all $f\in \mathcal{F}_r$ and $g\in \mathcal{G}_r$,
\begin{enumerate}
\item if $\varepsilon_f = 1$, then $\Diamond_f^\sigma u = \Diamond_{f^\sigma} u$ for every $u\in \bbas$.
\item if $\varepsilon_g = 1$, then $\Box_g^\pi u = \Box_{g^\pi} u$ for every $u\in \bbas$.
\item if $\varepsilon_f = \partial$, then ${\lhd}_f^\sigma u = {\lhd}_{f^\sigma} u$ for every $u\in \bbas$.
\item if $\varepsilon_g = \partial$, then ${\rhd}_g^\pi u = {\rhd}_{g^\pi} u$ for every $u\in \bbas$.
\end{enumerate}
\end{lem}
\begin{proof}%\marginnote{to be edited}
1. By nonempty join-preservation, it is enough to show that if $k\in \kbbas$ and $k\neq \bot$, then $\Diamond_f^\sigma k = f^\sigma(k) = :\Diamond_{f^\sigma} u$. By denseness,  $\{k\mid k\in \kbbas \mbox{ and } k\leq u\}\neq \varnothing$. Recalling that in $\bbas$ the interpretation of any $f\in \mathcal{F}_r$ with $\epsilon_f = 1$ preserves arbitrary nonempty joins, the following chain of identities holds:
\begin{center}
\begin{tabular}{r c ll}
$\Diamond_f^\sigma k $ & $=$ & $\bigwedge\{\Diamond_f a \mid a\in \bba\mbox{ and } k\leq a\}$ & (def.\ of $\sigma$-extension)\\
& $=$ & $\bigwedge\{f(a) \mid a\in \bba\mbox{ and } k\leq a\}$ & ($k\neq \bot$)\\
& $=$ & $f^\sigma(k)$. & (def.\ of $\sigma$-extension)\\
\end{tabular}
\end{center}
 The remaining items are order-variants and hence their proof is omitted.
\end{proof}

Since $\bbas$ is a complete lattice, by general and well known order-theoretic facts, all the connectives in $\mathcal{F}'\supseteq \mathcal{F}_n$ and in $\mathcal{G}'\supseteq \mathcal{G}_n$ (cf.\ Subsection \ref{ssec:expanded tense language}), have (coordinatewise) {\em adjoints}. This implies that the constructive canonical extension of any $\mathcal{L}_{\mathrm{LE}}$-algebra supports the interpretation of the connectives in $\mathcal{F}^*$ and in $\mathcal{G}^*$ (cf.\ Subsection \ref{ssec:expanded tense language}), and can hence be endowed with the structure of an $\mathcal{L}_{\mathrm{LE}}^*$-algebra, as required.

\subsection{The language of constructive ALBA for LEs}\label{Subsec:Expanded:Land}

The expanded language  manipulated by ALBA includes the $\mathcal{L}_{\mathrm{LE}}^*$-connectives, as well as a denumerably infinite set of sorted variables $\mathsf{NOM}$ called {\em nominals}, and a denumerably infinite set of
sorted variables $\mathsf{CO\text{-}NOM}$, called {\em co-nominals}. The elements of $\mathsf{NOM}$ will be denoted with with $\nomi, \nomj$, possibly indexed, and those of
$\mathsf{CO\text{-}NOM}$ with $\cnomm, \cnomn$, possibly indexed.
While in the non-constructive setting nominals and co-nominals range over the completely join-irreducible and the completely meet-irreducible elements of perfect LEs, respectively, in the present, constructive setting, nominals and co-nominals will be interpreted as elements of $\kbbas$ and  $\obbas$, respectively.

Let us introduce the expanded language formally: the \emph{formulas} $\phi$ of $\mathcal{L}_\mathrm{LE}^{+}$ are given by the following recursive definition:
\begin{center}
\begin{tabular}{r c |c|c|c|c|c|cccc c c c}
$\phi ::= $ &$\nomj$ & $\cnomm$ & $\psi$ & $\phi\wedge\phi$ & $\phi\vee\phi$ & $f(\overline{\phi})$ &$g(\overline{\phi})$
\end{tabular}
\end{center}
with $\psi  \in \mathcal{L}_\mathrm{LE}$, $\nomj \in \mathsf{NOM}$ and $\cnomm \in \mathsf{CO\text{-}NOM}$,  $f\in \mathcal{F}^*$ and $g\in \mathcal{G}^*$.

As in the case of $\mathcal{L}_\mathrm{LE}$, we can form inequalities and quasi-inequalities based on $\mathcal{L}_\mathrm{LE}^{+}$.
%Let $\mathcal{L}_\mathrm{LE}^{+\leq}$ and $\mathcal{L}_\mathrm{LE}^{+\mathit{quasi}}$ respectively denote the set of inequalities between terms in $\mathcal{L}_\mathrm{LE}^+$, and  the set of quasi-inequalities formed out of $\mathcal{L}_\mathrm{LE}^{+\leq}$.
%Members of $\mathcal{L}_\mathrm{LE}^+$, $\mathcal{L}_\mathrm{LE}^{+\leq}$, and $\mathcal{L}_\mathrm{LE}^{+\mathit{quasi}}$
%
Formulas, inequalities and quasi-inequalities in $\mathcal{L}_\mathrm{LE}^{+}$ not containing any propositional variables (but possibly containing nominals and co-nominals) will be called \emph{pure}.

%\noindent Summing up, we will be working with six sets of syntactic objects, as reported in the following table:

%\begin{center}
%\begin{tabular}{| c || l | l |}
%\hline &Base language & Expanded Language\\
%
%\hline \hline Formulas / terms &${\mathcal{L}_\mathrm{LE}}$ &${\mathcal{L}_\mathrm{LE}^+}$\\
%
%\hline Inequalities &$\mathcal{L}_\mathrm{LE}^{\leq}$ &$\mathcal{L}_\mathrm{LE}^{+\leq}$\\
%
%\hline Quasi-inequalities &$\mathcal{L}_\mathrm{LE}^{\mathit{quasi}}$ &$\mathcal{L}_\mathrm{LE}^{+\mathit{quasi}}$\\
%
%\hline
%\end{tabular}
%\end{center}
In the previous section, we showed that constructive canonical extensions of ${\mathcal{L}_\mathrm{LE}}$-algebras can be naturally endowed with the structure of ${\mathcal{L}_\mathrm{LE}^*}$-algebras. Building on this fact, we can use constructive canonical extensions of ${\mathcal{L}_\mathrm{LE}}$-algebras as a semantic environment for the language ${\mathcal{L}_\mathrm{LE}^+}$ as follows.
If $\bba$ is an $\mathcal{L}_{\mathrm{LE}}$-algebra,  then an \emph{assignment} for ${\mathcal{L}_\mathrm{LE}^+}$ on $\bbas$ is a map $v: \mathsf{PROP} \cup \mathsf{NOM} \cup \mathsf{CO\mbox{-}NOM} \rightarrow \bbas$ sending propositional variables to elements of $\bbas$, sending nominals to $\kbbas$ and co-nominals to $\obbas$. An \emph{admissible assignment}\label{admissible:assignment} for ${\mathcal{L}_\mathrm{LE}^+}$  on $\bbas$ is an assignment $v$ for ${\mathcal{L}_\mathrm{LE}^+}$ on $\bbas$, such that $v(p) \in \bba$ for each $p \in \mathsf{PROP}$. In other words, the assignment $v$ sends propositional variables to elements of the subalgebra $\bba$, while nominals and co-nominals get sent to closed and open elements of $\bbas$, respectively. This means that the value of ${\mathcal{L}_\mathrm{LE}}$-terms under an admissible assignment will belong to $\bba$, whereas ${\mathcal{L}_\mathrm{LE}^+}$-terms in general will not.

\section{Inductive and Sahlqvist Inequalities}\label{Sec:Ind:Sahl:Ineq}

In this section we  introduce the notion of inductive  $\mathcal{L}_\mathrm{LE}(\mathcal{F}, \mathcal{G})$-inequalities, which, as discussed in the introduction, is one of the main tools of unified correspondence. Its characteristic feature is being parametric in every signature, and being formulated purely in terms of the order-theoretic properties of the interpretation of the logical connectives. However, in the present setting, unlike most of the others, the parametric signature witnesses the coexistence of normal and regular connectives, and hence we warn the reader that Definition \ref{Inducive:Ineq:Def} below, although verbatim the same as the analogous one given in e.g.\ \cite{ConPal13}, is actually a proper generalization of e.g.\ \cite[Definition 3.4]{ConPal13}, since the  definition of good branch on which it depends is not the same as the one given in \cite{ConPal13}. We deliberately keep the same terminology (good branch, inductive) and reduce the notational changes to a minimum, since the conceptual principles motivating these  notions  are the same in every setting, and the present definitions extend the ones relative to all more restrictive settings and project onto each of them. The constructive canonicity of inductive inequalities  will follow from their ALBA-reducibility  and the fact that all ALBA-reducible inequalities  are constructively canonical.

		\begin{defi}[Signed Generation Tree]
			\label{def: signed gen tree}
			The \emph{positive} (resp.\ \emph{negative}) {\em generation tree} of any $\mathcal{L}_\mathrm{LE}$-term $s$ is defined by labelling the root node of the generation tree of $s$ with the sign $+$ (resp.\ $-$), and then propagating the labelling on each remaining node as follows:
			\begin{itemize}
				%\item The root node $+s$ (resp.\ $-s$) is the root node of the positive (resp.\ negative) generation tree of $s$ signed with + (resp.\ $-$).
				\item For any node labelled with $ \lor$ or $\land$, assign the same sign to its children nodes.
				%\item If a node is labelled with $\lhd$, $\rhd$, assign the opposite sign to its child node.
				\item For any node labelled with $h\in \mathcal{F}\cup \mathcal{G}$ of arity $n_h\geq 1$, and for any $1\leq i\leq n_h$, assign the same (resp.\ the opposite) sign to its $i$th child node if $\varepsilon_h(i) = 1$ (resp.\ if $\varepsilon_h(i) = \partial$).
			\end{itemize}
			Nodes in signed generation trees are \emph{positive} (resp.\ \emph{negative}) if are signed $+$ (resp.\ $-$).
		\end{defi}
		
		Signed generation trees will be mostly used in the context of term inequalities $s\leq t$. In this context we will typically consider the positive generation tree $+s$ for the left-hand side and the negative one $-t$ for the right-hand side. We will also say that a term-inequality $s\leq t$ is \emph{uniform} in a given variable $p$ if all occurrences of $p$ in both $+s$ and $-t$ have the same sign, and that $s\leq t$ is $\epsilon$-\emph{uniform} in a (sub)array $\vec{p}$ of its variables if $s\leq t$ is uniform in $p$, occurring with the sign indicated by $\epsilon$, for every\footnote{\label{footnote:uniformterms}The following observation will be used at various points in the remainder of the present paper: if a term inequality $s(\vec{p},\vec{q})\leq t(\vec{p},\vec{q})$ is $\epsilon$-uniform in $\vec{p}$ (cf.\ discussion after Definition \ref{def: signed gen tree}), then the validity of $s\leq t$ is equivalent to the validity of $s(\overrightarrow{\top^{\epsilon(i)}},\vec{q})\leq t(\overrightarrow{\top^{\epsilon(i)}},\vec{q})$, where $\top^{\epsilon(i)}=\top$ if $\epsilon(i)=1$ and $\top^{\epsilon(i)}=\bot$ if $\epsilon(i)=\partial$.} $p$ in $\vec{p}$.
		
		For any term $s(p_1,\ldots p_n)$, any order type $\epsilon$ over $n$, and any $1 \leq i \leq n$, an \emph{$\epsilon$-critical node} in a signed generation tree of $s$ is a leaf node $+p_i$ with $\epsilon_i = 1$ or $-p_i$ with $\epsilon_i = \partial$. An $\epsilon$-{\em critical branch} in the tree is a branch from an $\epsilon$-critical node. The intuition, which will be built upon later, is that variable occurrences corresponding to $\epsilon$-critical nodes are \emph{to be solved for}, according to $\epsilon$.
		
		For every term $s(p_1,\ldots p_n)$ and every order type $\epsilon$, we say that $+s$ (resp.\ $-s$) {\em agrees with} $\epsilon$, and write $\epsilon(+s)$ (resp.\ $\epsilon(-s)$), if every leaf in the signed generation tree of $+s$ (resp.\ $-s$) is $\epsilon$-critical.
		In other words, $\epsilon(+s)$ (resp.\ $\epsilon(-s)$) means that all variable occurrences corresponding to leaves of $+s$ (resp.\ $-s$) are to be solved for according to $\epsilon$. We will also write $+s'\prec \ast s$ (resp.\ $-s'\prec \ast s$) to indicate that the subterm $s'$ inherits the positive (resp.\ negative) sign from the signed generation tree $\ast s$. Finally, we will write $\epsilon(\gamma) \prec \ast s$ (resp.\ $\epsilon^\partial(\gamma_h) \prec \ast s$) to indicate that the signed subtree $\gamma$, with the sign inherited from $\ast s$, agrees with $\epsilon$ (resp.\ with $\epsilon^\partial$).
		\begin{defi}
			\label{def:good:branch}
			Nodes in signed generation trees will be called \emph{$\Delta$-adjoints}, \emph{syntactically additive coordinatewise (SAC)}, \emph{syntactically right residual (SRR)}, and \emph{syntactically multiplicative in the product (SMP)}\footnote{This division reflects the order-theoretic properties of the interpretation of these connectives (join/meet preservation/reversal, adjunction, residuation) which will be exploited in the reduction phase of the constructive ALBA-algorithm, introduced below in Section \ref{sec:constructive ALBA}}, according to the specification given in Table~\ref{Join:and:Meet:Friendly:Table}.
			A branch in a signed generation tree $\ast s$, with $\ast \in \{+, - \}$, is called a \emph{good branch} if it is the concatenation of two paths $P_1$ and $P_2$, one of which may possibly be of length $0$, such that $P_1$ is a path from the leaf consisting (apart from variable nodes) only of PIA-nodes, and $P_2$ consists (apart from variable nodes) only of Skeleton-nodes. %\footnote{These classes are grouped together into the super-classes \emph{Skeleton} and \emph{PIA} as indicated in the table. This organization is motivated and discussed in \cite{CFPS} and \cite{CoGhPa13}.}
A branch is \emph{excellent} if it is good and in $P_1$ there are only SMP-nodes. A good branch is \emph{Skeleton} if the length of $P_1$ is $0$ (hence Skeleton branches are excellent), and  is {\em SAC}, or {\em definite}, if  $P_2$ only contains SAC nodes.
			\begin{table}[h]
				\begin{center}
                \bgroup
                \def\arraystretch{1.2}%  1 is the default, change whatever you need
					\begin{tabular}{| c | c |}
						\hline
						Skeleton &PIA\\
						\hline
						$\Delta$-adjoints &  (SMP) \\
						\begin{tabular}{ c c c c c c}
							$+$ &$\phantom{\vee}$ &$\vee$ &$\phantom{\lhd}$ & &\\
							$-$ & &$\wedge$\\
							\hline
						\end{tabular}
						&
						\begin{tabular}{c c c c }
							$+$ &$\wedge$ &$g\in \mathcal{G}$ & with $n_g = 1$ \\
							$-$ &$\vee$ &$f\in \mathcal{F}$ & with $n_f = 1$ \\
							\hline
						\end{tabular}
						\\
						 (SAC) & (SRR)\\
						\begin{tabular}{c c c c }
							$+$ &  &$f\in \mathcal{F}$ &\\
							$-$ &  &$g\in \mathcal{G}$ &  \\
						\end{tabular}
						&\begin{tabular}{c c c c}
							$+$ & &$g\in \mathcal{G}_n$ & with $n_g \geq 2$\\
							$-$ &  &$f\in \mathcal{F}_n$ & with $n_f \geq 2$\\
						\end{tabular}
						\\
						\hline
					\end{tabular}
                \egroup
				\end{center}
				\caption{Skeleton and PIA nodes for $\mathrm{LE}$.}\label{Join:and:Meet:Friendly:Table}
				%\vspace{-1em}
			\end{table}
		\end{defi}

		\begin{defi}[Inductive inequalities]
			\label{Inducive:Ineq:Def}
			For any order type $\epsilon$ and any irreflexive and transitive relation $\Omega$ on $p_1,\ldots p_n$, the signed generation tree $*s$ $(* \in \{-, + \})$ of a term $s(p_1,\ldots p_n)$ is \emph{$(\Omega, \epsilon)$-inductive} if
			\begin{enumerate}
				\item for all $1 \leq i \leq n$, every $\epsilon$-critical branch with leaf $p_i$ is good (cf.\ Definition \ref{def:good:branch});
				\item every $m$-ary SRR-node occurring in the critical branch is of the form $ h(\gamma_1,\dots,\gamma_{j-1},\beta, \allowbreak \gamma_{j+1}\ldots,\gamma_m)$, where for any $\ell \in\{1,\ldots,m\}\setminus j$: %$\gamma \in \{\gamma_1,\ldots,\gamma_{j-1},\gamma_{j+1},\ldots,\alpha_m\}$
\begin{enumerate}
\item $\epsilon^\partial(\gamma_\ell) \prec \ast s$ (cf.\ discussion before Definition \ref{def:good:branch}), and
%\item $\epsilon^\partial(\ast \gamma)$, and
%
\item $p_k \Omega p_i$ for every $p_k$ occurring in $\gamma_\ell$. 
\end{enumerate}
	\end{enumerate}
			We will refer to $\Omega$ as the \emph{dependency order} on the variables. An inequality $s \leq t$ is \emph{$(\Omega, \epsilon)$-inductive} if the signed generation trees $+s$ and $-t$ are $(\Omega, \epsilon)$-inductive. An inequality $s \leq t$ is \emph{inductive} if it is $(\Omega, \epsilon)$-inductive for some $\Omega$ and $\epsilon$.
		\end{defi}
		
		In what follows, we will find it useful to refer to formulas $\phi$ such that only PIA nodes occur in $+\phi$ (resp.\ $-\phi$) as {\em positive} (resp.\ {\em negative}) {\em PIA-formulas}, and to formulas $\xi$ such that only Skeleton nodes occur in $+\xi$ (resp.\ $-\xi$) as {\em positive} (resp.\ {\em negative}) {\em Skeleton-formulas}\label{page: positive negative PIA}.

\begin{defi}\label{Sahlqvist:Ineq:Def}
Given an order type $\epsilon$, the signed generation tree $\ast s$, $\ast \in \{-, + \}$, of a term $s(p_1,\ldots p_n)$ is \emph{$\epsilon$-Sahlqvist} if every $\epsilon$-critical branch is excellent. An inequality $s \leq t$ is \emph{$\epsilon$-Sahlqvist} if the trees $+s$ and $-t$ are both $\epsilon$-Sahlqvist.  An inequality $s \leq t$ is \emph{Sahlqvist} if it is $\epsilon$-Sahlqvist for some $\epsilon$.
\end{defi}
\begin{rem}\label{Rem:Suzuki}
In \cite{Suzuki11a}, Suzuki proves a constructive canonicity result in a setting of lattice expansions  given by  the signature $(h, k,\neg, l, r, c)$ such that $c$ is a constant, $-h$ and $+k$ are SMP of arbitrary arity and order-type,  $+l$ and $-r$ are SAC of arbitrary arity $n_l = n_r$ and order-type, and $\neg$ is a unary connective s.t.\ $\epsilon_{\neg} = \partial$, which is an involution ($\neg\neg x = x$) and moreover $+\neg$ and $-\neg$ are both SAC. In some cases, $l$ and $r$ are intended to be residual to each other in some coordinate, thus being promoted from SAC to SRR in that coordinate.
%\footnote{Notice that unary SMP-connectives are SAC; as remarked in \cite{Suzuki11a}, $-\vee$ and $+\wedge$ are concrete examples non-unary connectives fitting the description of $-h$ and $+h$ ( these are exactly the (non-unary) SMP-connectives). However,  other concrete examples of non-unary SMP-connectives are hard to come by. This is the reason why the setting of the present paper does not include  classes $\mathcal{F}_p$ and $\mathcal{G}_p$ specifically accounting for SMP-connectives.}  This result  covers  a syntactically defined class of sequents, described in terms of the  families defined by simultaneous recursion as follows:
\[ \cup\mbox{-terms}\quad s:: = b  \mid s\vee s\mid h(\overline{s}^{\varepsilon_h}) \mid l(\overline{s}^{\varepsilon_l})\]
\[ \cap\mbox{-terms}\quad  t:: = d\mid  t\wedge t \mid k(\overline{t}^{\varepsilon_k}) \mid r(\overline{t}^{\varepsilon_r})\]
\[ b\mbox{-terms}\quad b:: = p\mid \top \mid \bot \mid b\wedge b\mid k(\overline{b}^{\varepsilon_k})\mid  r(\overline{x}^{\varepsilon_r(i)})\] % b/ c\mid c\backslash b\mid c/ d\mid d\backslash c \]
\[ d\mbox{-terms}\quad  d:: = p\mid \top \mid \bot\mid d\vee d\mid h(\overline{d}^{\varepsilon_h}) \mid l(\overline{y}^{\varepsilon_l(i)}),\] % d\circ c \mid c\circ d\]
where
\begin{itemize}
\item[--] the $i$th element of the $n_h$-tuple $\overline{s}^{\varepsilon_h}$ (resp.\ $\overline{d}^{\varepsilon_h}$)  is a $\cup$-term (resp.\ a $d$-term) if $\varepsilon_h(i) = 1$ and is a $\cap$-term (resp.\ a $b$-term) if $\varepsilon_h(i) = \partial$;
\item[--] the $i$th element of the $n_l$-tuple $\overline{s}^{\varepsilon_l}$   is a $\cup$-term  if $\varepsilon_l(i) = 1$ and is a $\cap$-term  if $\varepsilon_l(i) = \partial$;
\item[--] the $i$th element of the $n_k$-tuple $\overline{t}^{\varepsilon_k}$ (resp.\ $\overline{b}^{\varepsilon_k}$) is a $\cap$-term (resp.\ \ $b$-term) if $\varepsilon_k(i) = 1$ and is a $\cup$-term (resp.\ a $d$-term) if $\varepsilon_g(i) = \partial$;
\item[--] the $i$th element of the $n_r$-tuple $\overline{t}^{\varepsilon_r}$   is a $\cap$-term  if $\varepsilon_r(i) = 1$ and is a $\cup$-term  if $\varepsilon_r(i) = \partial$;
\item[--] $\overline{x}^{\varepsilon_r(i)}$ is an $n_r$-tuple such that every coordinate but the $i$th is a constant term, and the $i$th element is a $b$-term if $\varepsilon_r(i) = 1$, and  is a $d$-term if $\varepsilon_r(i) = \partial$;
\item[--] $\overline{y}^{\varepsilon_l(i)}$ is an $n_l$-tuple such that every coordinate but the $i$th is a constant term, and the $i$th element is a $d$-term if $\varepsilon_l(i) = 1$, and  is a $b$-term if $\varepsilon_l(i) = \partial$.
    \end{itemize}
The inequalities proven to be canonical in \cite[Theorem 5.10]{Suzuki13}   are of the form $\phi\leq \psi$, such that  \[\phi = s(\overline{x},  \overline{a_1(\overline{x})}/\overline{z})\quad \quad \psi = t(\overline{x},  \overline{a_2(\overline{x})}/\overline{z})\] where $s(\overline{x}, \overline{z})$ is a $\cup$-term, $t(\overline{x}, \overline{z})$ is a $\cap$-term, and  there exists some order-type $\epsilon$ on $\overline{x}$ such that $\epsilon^{\partial}(a_1(\overline{x}))\prec +\phi$ and $\epsilon^{\partial}(a_2(\overline{x}))\prec -\psi$ for each $a_1$ in $\overline{a_1}$ and each $a_2$ in $\overline{a_2}$.

We would now like to compare these inequalities with the Sahlqvist and inductive inequalities as defined above. The first point to note is that we do not include any non-unary SMP nodes other than $+\wedge$ and $-\vee$ in Table~\ref{Join:and:Meet:Friendly:Table}. The reason for this is that, apart from conjunction and disjunction, we are not aware of any non-unary connective occurring in a well-known logic, the interpretation of which has this order-theoretic property. This being said, it is possible to lift the restriction that $n_g = 1$ and $n_f = 1$ in the SMP cell of Table~\ref{Join:and:Meet:Friendly:Table} without affecting any of the results in the present paper. Doing so does mean we need to add n-ary versions of the residuation rules for regular connectives to ALBA, but this is straightforward and will be discussed in Remark \ref{Rem:Additional:Rule}.

Given this modification to Table~\ref{Join:and:Meet:Friendly:Table}, it is not difficult to see that the inequality $\phi\leq \psi$ described above falls under the definition of $\epsilon$-Sahlqvist inequalities (cf.\ Definition \ref{Sahlqvist:Ineq:Def}), when one takes into account the following considerations.  Firstly, any operation that is multiplicative in the product (as always, modulo some order-type), is also multiplicative coordinatewise, therefore all SMP nodes are SAC (of course, preservation of empty meets in the product does not imply preservation of empty meets coordinatewise, but that does not concern us now). Secondly, it follows that taking any operation that is coordinatewise multiplicative and holding all but one of its coordinates constant turns it, in effect, into one that is multiplicative in the (now one-fold) product. Therefore, instantiating all arguments of a SAC node save one with constants turns it into an SMP node.

Taking this into account, the trees $+b$ and $-d$ can be recognized as consisting entirely of PIA-nodes, and in particular  entirely of SMP nodes. Further, the generation trees $+s$ and $-t$ of $\cup$- and $\cap$-terms are constructed by taking trees consisting entirely of skeleton nodes and inserting subtrees $+b \prec +s, -t$ and $-d \prec +s, -t$ at leaves. Thus \emph{all} branches in $+s$ and $-t$ are excellent. Substituting $\overline{a_1(\overline{x})}$ and $\overline{a_2(\overline{x})}$ into $s$ and $t$ as indicated possibly introduces non-good branches into $+ s(\overline{x},  \overline{a_1(\overline{x})}/\overline{z})$ and $- t(\overline{x},  \overline{a_2(\overline{x})}/\overline{z})$, but these will be non-critical according to $\epsilon$ and therefore both trees will be $\epsilon$-Sahlqvist.

%SAC node is reduced to be {\em de facto} unary, since all its coordinates but one take constant terms as arguments. % subject to the restriction that  all the arguments but one are constant.
%With this restriction, the order-theoretic behaviour of the interpretation of these PIA terms becomes essentially the same as that of  Sahlqvist PIA-terms.

The same analysis applies to the scope of the Sahlqvist canonicity result of Ghilardi and Meloni \cite{GhMe97} (cf.\ \cite[Remark 12]{CoPaSoZh15}). Indeed, %Su's class is intermediate between the Sahlqvist and inductive class defined in the present paper, although the order-theoretic underpinning of Ghilardi-Meloni's definition is the very same as that of Definition \ref{Sahlqvist:Ineq:Def}.
Suzuki's treatment   extracts the syntactic definition from Ghilardi-Meloni's order-theoretic insights and transfers it to the setting of general lattice expansions.
\end{rem}

\begin{exa}\label{ex: frege as type 3}
			Let $\mathcal{F} = \varnothing = \mathcal{G}_r$, and  $\mathcal{G}_n = \{\rhu\}$, with $\rhu$ binary and of order-type $(\partial, 1)$.
			As observed in \cite{ConPal12}, the Frege inequality
			
			\[p\rhu(q\rhu r)\leq (p\rhu q)\rhu (p\rhu r)\] %\marginnote{Type 3 in the order-type run below, and Type 4 in the order-type 111.}
			
			\noindent is not Sahlqvist for any order type, but is $(\Omega, \epsilon)$-inductive, e.g.\ for $r <_\Omega p <_\Omega q$ and $\epsilon(p, q, r) =(1, 1, \partial)$,  as can be seen from the signed generation trees below. 		
		
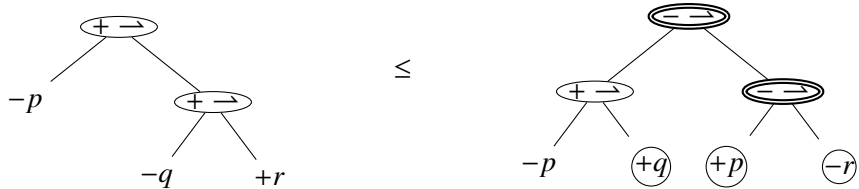
\begin{center}
				\begin{tikzpicture}
				\tikzstyle{level 1}=[level distance=1cm, sibling distance=2.5cm]
				\tikzstyle{level 2}=[level distance=1cm, sibling distance=1.5cm]
				\tikzstyle{level 3}=[level distance=1 cm, sibling distance=1.5cm]

				\node[PIA] {$+\rightharpoonup$}
				child{node{$-p$}}
				child{node[PIA]{$+\rightharpoonup$}
					child{node{$-q$}}
					child{node{$+r$}}}
				;
				
				\end{tikzpicture}
                \hspace{1cm}
                \begin{tikzpicture}
                \node (b) at (0, 0) [auto] [] {$\phantom{x}$};
				\node (p) at (0, 1.5) [auto] [] {$\leq$};
                \end{tikzpicture}
                \hspace{1cm}
				\begin{tikzpicture}
				\tikzstyle{level 1}=[level distance=1cm, sibling distance=2.5cm]
				\tikzstyle{level 2}=[level distance=1cm, sibling distance=1.5cm]
				\tikzstyle{level 3}=[level distance=1 cm, sibling distance=1.5cm]
				\node[Ske]{$-\rightharpoonup$}
				child{node[PIA]{$+\rightharpoonup$}
					child{node{$-p$}}
					child{node{\circled{$+q$}}}}
				child{node[Ske]{$-\rightharpoonup$}
					child{node{\circled{$+p$}}}
					child{node{\circled{$-r$}}}}
				;
				\end{tikzpicture}
			%\caption{Signed generation tree for $\Box (p \rightarrow  q) \to \Box (\Box p\to\Box q)$}
			%\label{fig:fisher-servi}
		\end{center}
In the picture above, the circled variable occurrences are the $\epsilon$-critical ones, the doubly circled nodes are the Skeleton ones and the single-circle ones are PIA. In the intuitionistic setting of \cite{GhMe97}, the Frege inequality with $\rhu$  interpreted as Heyting implication is a validity, and hence is trivially canonical. In the lattice/poset-based setting of \cite{Suzuki11a,Suzuki13}, the Frege inequality is not a validity, but because proposition variables occur on both coordinates of SRR nodes, the Frege inequality falls outside  the fragment accounted for in \cite{Suzuki11a,Suzuki13}. In Example \ref{ex: ALBA run on frege}, we will give a successful execution of ALBA on the Frege inequality.
\end{exa}

\section{Constructive ALBA}\label{sec:constructive ALBA}
In this section, we present the second main tool of unified correspondence, namely the algorithm ALBA, in a  version suitable to  the constructive setting. Like all versions of ALBA, this one proceeds in three stages: preprocessing, reduction/elimination, and output. Since the current setting is general (i.e.\ non necessarily distributive) lattice expansions, the algorithm discussed here is more similar to the one introduced in \cite{ConPal13} than to versions based on distributive lattices.  However, unlike the version of \cite{ConPal13}, the constructive ALBA includes special `residuation rules' for regular connectives along with residuation rules proper. This is due to the fact that we need to accommodate a signature with different order-theoretic properties than the one in \cite{ConPal13}, and, in particular, the residuation rules for normal connectives are not sound when applied to regular connectives. The other rules, although syntactically the same as in \cite{ConPal13},  take on  a different semantic meaning  in the present setting,  as we will see in the next section.

\medskip
  ALBA takes   $\mathcal{L}_\mathrm{LE}$-inequalities $\phi \leq \psi$ as input and then proceeds in three stages. The first stage preprocesses $\phi \leq \psi$ by eliminating all uniformly  occurring propositional variables, and applying distribution and splitting rules exhaustively. This produces a finite set of inequalities, $\phi'_i \leq \psi'_i$, $1 \leq i \leq n$.

Now ALBA forms the \emph{initial quasi-inequalities} $\bigamp S_i \Rightarrow \sf{Ineq}_i$, compactly represented as tuples $(S_i, \sf{Ineq}_i)$ referred as \emph{systems},  with each $S_i$ initialized to the empty set and $\sf{Ineq}_i$ initialized to $\phi'_i \leq \psi'_i$.

The second stage (called the reduction  stage) transforms $S_i$ and $\mathsf{Ineq}_i$ through the application of transformation rules, which are listed below. The aim is to eliminate all propositional variables from $S_i$ and $\mathsf{Ineq}_i$ in favour of  nominals and co-nominals. A system for which this has been done will be called \emph{pure} or \emph{purified}. The actual eliminations are effected through the Ackermann-rules, while the other rules are used to bring $S_i$ and $\mathsf{Ineq}_i$ into the appropriate shape which make these applications possible. Once all propositional variables have been eliminated, this phase terminates and returns the pure quasi-inequalities $\bigamp S_i \Rightarrow \mathsf{Ineq}_i$.

The third stage either reports failure if some system could not be purified, or else returns the conjunction of the pure quasi-inequalities $\bigamp S_i \Rightarrow \mathsf{Ineq}_i$, which we denote by $\mathsf{ALBA}(\phi \leq \psi)$.

We now outline each of the three stages in more detail:

\subsection{Stage 1: Preprocessing and initialization} ALBA receives an $\mathcal{L}_{\mathrm{LE}}$-inequality $\phi \leq \psi$ as input. It applies the following {\bf rules for elimination of monotone variables}  to $\phi \leq \psi$ exhaustively, in order to eliminate any propositional variables which occur uniformly:

\begin{prooftree}
\AxiomC{$\alpha(p) \leq \beta(p)$}\UnaryInfC{$\alpha(\bot) \leq \beta(\bot)$}
\AxiomC{$\gamma(p) \leq \delta(p)$}\UnaryInfC{$\gamma(\top) \leq \delta(\top)$}
\noLine\BinaryInfC{}
\end{prooftree}
for $\alpha(p) \leq \beta(p)$ positive and $\gamma(p) \leq \delta(p)$ negative in $p$, respectively (see footnote \ref{footnote:uniformterms}).

Next, ALBA exhaustively distributes $f\in \mathcal{F}$ over $+\vee$, and  $g\in \mathcal{F}$ over $-\wedge$, so as to bring occurrences of $+\vee$ and $-\wedge$ to the surface wherever this is possible, and then eliminate them via exhaustive applications of {\em splitting} rules.
\paragraph{Splitting-rules.}

\begin{prooftree} \AxiomC{$\alpha \leq \beta \wedge \gamma
$}\UnaryInfC{$\alpha \leq \beta \quad \alpha \leq \gamma$}
\AxiomC{$\alpha \vee \beta \leq \gamma$}\UnaryInfC{$\alpha \leq \gamma \quad \beta \leq \gamma$}
\noLine\BinaryInfC{}
\end{prooftree}

This gives rise to  a set of inequalities $\{\phi_i' \leq \psi_i'\mid 1\leq i\leq n\}$. Now ALBA forms the \emph{initial quasi-inequalities} $\bigamp S_i \Rightarrow \sf{Ineq}_i$, compactly represented as tuples $(S_i, \sf{Ineq}_i)$ referred as \emph{systems},  with each $S_i$ initialized to the empty set and $\sf{Ineq}_i$ initialized to $\phi'_i \leq \psi'_i$. Each initial system is passed separately to stage 2, described below, where we will suppress indices $i$.

\subsection{Stage 2: Reduction and elimination}

The aim of this stage is to eliminate all occurring propositional variables from a given system $(S, \mathsf{Ineq})$. This is done by means of the following \emph{approximation rules}, \emph{residuation rules}, \emph{splitting rules}, and \emph{Ackermann-rules}, collectively called \emph{reduction rules}. The terms and inequalities in this subsection are~from $\mathcal{L}_\mathrm{LE}^{+}$.

\paragraph{Approximation rules.} There are four approximation rules. Each of these rules functions by simplifying $\mathsf{Ineq}$ and adding an inequality to $S$.

\begin{description}
\item[Left-positive approximation rule.] $\phantom{a}$%If $\mathsf{Ineq} = \phi \leq \psi$ and $\phi = \phi'(\gamma / !x)$, and
\begin{center}
\AxiomC{$(S, \;\; \phi'(\gamma / !x)\leq \psi)$}
\RightLabel{$(L^+A)$}
\UnaryInfC{$(S\! \cup\! \{ \nomj \leq \gamma\},\;\; \phi'(\nomj / !x)\leq \psi)$}
\DisplayProof
\end{center}
with $+x \prec +\phi'(!x)$,  the branch of $+\phi'(!x)$ starting at $+x$ being SAC (cf.\ definition \ref{def:good:branch}), $\gamma$ belonging to the original language $\mathcal{L}_\mathrm{LE}$
%
    %\begin{itemize}
    %\item $S := S \cup \{ \nomj \leq \gamma\}$, and
    %
    %\item $\mathsf{Ineq} := \phi'(\nomj) \leq \psi$
    %\end{itemize}
    %
    and $\nomj$ being the first nominal variable not  occurring in $S$ or  $\phi'(\gamma / !x)\leq \psi$.
\item[Left-negative approximation rule.]$\phantom{a}$%If $\mathsf{Ineq} = \phi \leq \psi$ and $\phi = \phi'(\gamma / !x)$, and
\begin{center}
\AxiomC{$(S, \;\; \phi'(\gamma / !x)\leq \psi)$}
\RightLabel{$(L^-A)$}
\UnaryInfC{$(S\! \cup\! \{ \gamma\leq\cnomm\},\;\; \phi'(\cnomm / !x)\leq \psi)$}
\DisplayProof
\end{center}
%If $\mathsf{Ineq} = \phi \leq \psi$ and $\phi = \phi'(\gamma / !x)$, and
with $-x \prec +\phi'(!x)$, the branch of $+\phi'(!x)$ starting at $-x$ being SAC, $\gamma$ belonging to the original language $\mathcal{L}_\mathrm{LE}$ and
     $\cnomm$ being the first co-nominal not  occurring in $S$ or $\phi'(\gamma / !x)\leq \psi$.
\item[Right-positive approximation rule.]$\phantom{a}$%If $\mathsf{Ineq} = \phi \leq \psi$ and $\phi = \phi'(\gamma / !x)$, and
\begin{center}
\AxiomC{$(S, \;\; \phi\leq \psi'(\gamma / !x))$}
\RightLabel{$(R^+A)$}
\UnaryInfC{$(S\! \cup\! \{ \nomj \leq \gamma\},\;\; \phi\leq \psi'(\nomj / !x))$}
\DisplayProof
\end{center}
%If $\mathsf{Ineq} = \phi \leq \psi$ and $\psi = \psi'(\gamma / !x)$, and
with $+x \prec -\psi'(!x)$,  the branch of $-\psi'(!x)$ starting at $+x$ being SAC, $\gamma$ belonging to the original language $\mathcal{L}_\mathrm{LE}$ and
     $\nomj$ being the first nominal not  occurring in $S$ or $\phi\leq \psi'(\gamma / !x)$.
\item[Right-negative approximation rule.] $\phantom{a}$%If $\mathsf{Ineq} = \phi \leq \psi$ and $\phi = \phi'(\gamma / !x)$, and
\begin{center}
\AxiomC{$(S, \;\; \phi\leq \psi'(\gamma / !x))$}
\RightLabel{$(R^-A)$}
\UnaryInfC{$(S\! \cup\! \{ \gamma\leq \cnomm\},\;\; \phi\leq \psi'(\cnomm / !x))$}
\DisplayProof
\end{center}
%If $\mathsf{Ineq} = \phi \leq \psi$ and $\psi = \psi'(\gamma / !x)$, and
with $-x \prec -\psi'(!x)$, the branch of $-\psi'(!x)$ starting at $-x$ being SAC, $\gamma$ belonging to the original language $\mathcal{L}_\mathrm{LE}$ and  $\cnomm$ being the first co-nominal not  occurring in $S$ or $\phi\leq \psi'(\gamma / !x))$.
\end{description}

%\begin{rem}
\noindent The approximation rules above, as stated, will be shown to be sound both under admissible and under arbitrary assignments (cf.\ Proposition \ref{Rdctn:Rls:Crctnss:Prop}). However, their liberal application gives rise to topological complications in the proof of canonicity. Therefore, we will restrict the applications of approximation rules to nodes $!x$ giving rise to {\em maximal} SAC branches. Such applications will be called {\em pivotal}. Also, executions of ALBA in which approximation rules are applied only pivotally will be referred to as {\em pivotal}.\label{pivotal:approx:rule:application}
%\texttt{Define the notion of an OPTIMAL APPLICATION OF AN APPROXIMATION RULE, i.e., and application where the subtree we pull out os rooted at a pivot node. This is needed to preserve syntactic closed $\leq$ syntactic open shape, for the proof of canonicity. We have redefine the notion of a pivot node to be the first non-SJF node from the root on a branch.}\marginpar{\raggedright \tiny{Move this to the completeness section, if needed there. Otherwise, delete.}}
%\end{rem}

\paragraph{Residuation rules.} These rules operate on the inequalities in $S$, by rewriting a chosen inequality in $S$ into another inequality. For every $f\in \mathcal{F}_n$ and $g\in \mathcal{G}_n$, and any $1\leq i\leq n_f$ and $1\leq j\leq n_g$,

\begin{prooftree}
\AxiomC{$f(\phi_1,\ldots,\phi_i,\ldots,\phi_{n_f}) \leq \psi $}
\RightLabel{$\epsilon_f(i) = 1$}
\UnaryInfC{$\phi_i\leq f^\sharp_i(\phi_1,\ldots,\psi,\ldots,\phi_{n_f})$}
\AxiomC{$f(\phi_1,\ldots,\phi_i,\ldots,\phi_{n_f}) \leq \psi $}
\RightLabel{$\epsilon_f(i) = \partial$}
\UnaryInfC{$f^\sharp_i(\phi_1,\ldots,\psi,\ldots,\phi_{n_f})\leq \phi_i$}
\noLine\BinaryInfC{}
\end{prooftree}

\begin{prooftree}
\AxiomC{$\psi\leq g(\phi_1,\ldots,\phi_i,\ldots,\phi_{n_g})$}
\RightLabel{$\epsilon_g(i) = 1$}
\UnaryInfC{$g^\flat_i(\phi_1,\ldots,\psi,\ldots,\phi_{n_g})\leq \phi_i$}
\AxiomC{$\psi\leq g(\phi_1,\ldots,\phi_i,\ldots,\phi_{n_g})$}
\RightLabel{$\epsilon_g(i) = \partial$}
\UnaryInfC{$\phi_i\leq g^\flat_i(\phi_1,\ldots,\psi,\ldots,\phi_{n_g})$}
\noLine\BinaryInfC{}
\end{prooftree}
For every $f\in \mathcal{F}_r$ and $g\in \mathcal{G}_r$,
\begin{center}
\begin{tabular}{cc}
\AxiomC{$f(\phi)\leq\psi$}
\RightLabel{(if $\epsilon_f = 1$)}
\UnaryInfC{$f(\bot)\leq \psi\;\;\;\phi\leq\blacksquare_{f}\psi$}
\DisplayProof

&
\AxiomC{$ \phi\leq g(\psi)$}
\RightLabel{(if $\epsilon_g = 1$)}
\UnaryInfC{$\phi\leq g(\top)\;\;\; \Diamondblack_g\phi\leq\psi$}
\DisplayProof

\end{tabular}
\end{center}

\begin{center}
\begin{tabular}{cc}
\AxiomC{$f(\phi)\leq \psi$}
\RightLabel{(if $\epsilon_f = \partial$)}
\UnaryInfC{$f(\top)\leq\psi\;\;\; {\blhd} \psi\leq\phi$}
\DisplayProof

& \AxiomC{$\phi\leq g(\psi)$}
\RightLabel{(if $\epsilon_g = \partial$)}
\UnaryInfC{$\phi\leq g(\bot)\;\;\;\psi\leq {\brhd} \phi$}
\DisplayProof \\

\end{tabular}
\end{center}
In a given system, each of these rules replaces an instance of the upper inequality with the corresponding instances of the two lower inequalities.

The leftmost inequalities in each rule above will be referred to as the \emph{side condition}\label{def:sidecondition}.

\paragraph{Right Ackermann-rule.} $\phantom{a}$%If $\mathsf{Ineq} = \phi \leq \psi$ and $\phi = \phi'(\gamma / !x)$, and
\begin{center}
\AxiomC{$(\{ \alpha_i \leq p \mid 1 \leq i \leq n \} \cup \{ \beta_j(p)\leq \gamma_j(p) \mid 1 \leq j \leq m \}, \;\; \mathsf{Ineq})$}
\RightLabel{$(RAR)$}
\UnaryInfC{$(\{ \beta_j(\bigvee_{i=1}^n \alpha_i)\leq \gamma_j(\bigvee_{i=1}^n \alpha_i) \mid 1 \leq j \leq m \},\;\; \mathsf{Ineq})$}
\DisplayProof
\end{center}
where:
\begin{itemize}
\item $p$ does not occur in $\alpha_1, \ldots, \alpha_n$ or in $\mathsf{Ineq}$,
\item $\beta_{1}(p), \ldots, \beta_{m}(p)$ are positive in $p$, and
\item $\gamma_{1}(p), \ldots, \gamma_{m}(p)$ are negative in $p$.

\end{itemize}

\paragraph{Left Ackermann-rule.}$\phantom{a}$%If $\mathsf{Ineq} = \phi \leq \psi$ and $\phi = \phi'(\gamma / !x)$, and
\begin{center}
\AxiomC{$(\{ p \leq \alpha_i \mid 1 \leq i \leq n \} \cup \{ \beta_j(p)\leq \gamma_j(p) \mid 1 \leq j \leq m \}, \;\; \mathsf{Ineq})$}
\RightLabel{$(LAR)$}
\UnaryInfC{$(\{ \beta_j(\bigwedge_{i=1}^n \alpha_i)\leq \gamma_j(\bigwedge_{i=1}^n \alpha_i) \mid 1 \leq j \leq m \},\;\; \mathsf{Ineq})$}
\DisplayProof
\end{center}
where:
\begin{itemize}
\item $p$ does not occur in $\alpha_1, \ldots, \alpha_n$ or in $\mathsf{Ineq}$,
\item $\beta_{1}(p), \ldots, \beta_{m}(p)$ are negative in $p$, and
\item $\gamma_{1}(p), \ldots, \gamma_{m}(p)$ are positive in $p$.

\end{itemize}

\subsection{Stage 3: Success, failure and output}

If stage 2 succeeded in eliminating all propositional variables from each system, the algorithm returns the conjunction of these purified quasi-inequalities, denoted by $\mathsf{ALBA}(\phi \leq \psi)$. Otherwise, the algorithm reports failure and terminates.

\begin{exa}\label{ex: ALBA run on frege}
As mentioned in Example \ref{ex: frege as type 3},
		
		\[p\rhu(q\rhu r)\leq (p\rhu q)\rhu (p\rhu r)\]%\marginnote{Type 3 in the order-type run below, and Type 4 in the order-type 111.}

		%the order-type $\varepsilon_{p}=\varepsilon_{q}=\varepsilon_{r}=1$ and dependency order $p<_{\Omega}q<_{\Omega}r$) and .
		\noindent is  $(\Omega, \epsilon)$-inductive for $r <_\Omega p <_\Omega q$ and $\epsilon(p, q, r) =(1, 1, \partial)$. A pivotal execution of ALBA according to this choice of $\Omega$ and $\epsilon$, is given below. The variables $\nomi, \nomj,\nomh$ are nominals, $\cnomm$ is a co-nominal, $\bullet$ is the left adjoint of $\rhu$ in its second coordinate and $\lhu$ is the Galois-adjoint of $\rhu$ in its first coordinate.\footnote{Notice, however, that the Galois-adjoint of $\rhu$ in its first coordinate is, modulo inversion of coordinates, also the right adjoint of $\bullet$ in its first coordinate (cf.\ Footnote \ref{footnote: notation residuals}). The symbol $\lhu$ is more appropriate for the latter role.}
		
		%(which, as observed in cite{CoPa10}, is strictly right-primitive $(\Omega, \epsilon)$-inductive (e.g.\ for the order-type $\varepsilon_{p}=\varepsilon_{q}=\varepsilon_{r}=1$ and dependency order $p<_{\Omega}q<_{\Omega}r$) and is not Sahlqvist for any order type. In a Heyting algebra, the Frege axiom is a tautology. However, in weaker settings, such as the one of pre-Heyting algebras,
		
		%Clearly, the axiom above is also not primitive.\marginnote{Notice that the running below is not according to this order-type.}
		
		%For every perfect lattice $\bba$, let $\to$ be a binary operation on $\bba$ such that for every$u, v, w\in \bba$ $$u\bullet v\leq w \ \mbox{ iff } u\leq v\to w.$$ (here $p, q, r$ range on arbitrary elements of $\bba$,$m$ ranges in $\mty(\bba)$ and $i, j, h$ range in $\jty(\bba)$):

                \medskip
\noindent\begin{tabular}{@{}r@{}c l l@{}}
				& & $\forall p\forall q\forall r[p\rhu(q\rhu r)\leq (p\rhu q)\rhu (p\rhu r)]$ &\\
				& iff\! & $\forall p \forall q\forall r \forall \nomi\forall \nomj \forall \cnomm\forall \nomh[(\nomh\leq p\rhu(q\rhu r)\ \&\ \nomi\leq p\rhu q  \ \&\ \nomj\leq p  \ \& \ r\leq \cnomm)\Rightarrow \nomh\leq \nomi\rhu (\nomj\rhu \cnomm)]$&\\
& iff\! & $\forall p \forall q\forall \nomi\forall \nomj\forall \nomh \forall \cnomm[(\nomh\leq p\rhu(q\rhu \cnomm)\ \&\ \nomi\leq p\rhu q  \ \&\ \nomj\leq p)\Rightarrow \nomh\leq \nomi\rhu (\nomj\rhu \cnomm)]$&\\
& iff\! & $\forall q\forall \nomi\forall \nomj\forall \nomh \forall \cnomm[(\nomh\leq \nomj\rhu(q\rhu \cnomm)\ \&\ \nomi\leq \nomj\rhu q)\Rightarrow \nomh\leq \nomi\rhu (\nomj\rhu \cnomm)]$&\\
& iff\! & $\forall q\forall \nomi\forall \nomj\forall \nomh \forall \cnomm[(\nomj\bullet \nomh\leq q\rhu \cnomm\ \&\ \nomi\leq \nomj\rhu q)\Rightarrow \nomh\leq \nomi\rhu (\nomj\rhu \cnomm)]$&\\
& iff\! & $\forall q\forall \nomi\forall \nomj\forall \nomh \forall \cnomm[(q\bullet(\nomj\bullet \nomh)\leq \cnomm\ \&\ \nomi\leq \nomj\rhu q)\Rightarrow \nomh\leq \nomi\rhu (\nomj\rhu \cnomm)]$&\\
& iff\! & $\forall q\forall \nomi\forall \nomj\forall \nomh \forall \cnomm[(q\leq \cnomm\lhu (\nomj\bullet \nomh)\ \&\ \nomi\leq \nomj\rhu q)\Rightarrow \nomh\leq \nomi\rhu (\nomj\rhu \cnomm)]$&\\
& iff\! & $\forall \nomi\forall \nomj\forall \nomh \forall \cnomm[\nomi\leq \nomj\rhu (\cnomm\lhu (\nomj\bullet \nomh))\Rightarrow \nomh\leq \nomi\rhu (\nomj\rhu \cnomm)]$&\\	
& iff\! & $\forall \nomi\forall \nomj\forall \nomh \forall \cnomm[\nomi\leq \nomj\rhu (\cnomm\lhu (\nomj\bullet \nomh))\Rightarrow \nomi\bullet \nomh\leq  \nomj\rhu \cnomm]$&\\
& iff\! & $\forall \nomi\forall \nomj\forall \nomh \forall \cnomm[\nomi\leq \nomj\rhu (\cnomm\lhu (\nomj\bullet \nomh))\Rightarrow \nomi\leq  (\nomj\rhu \cnomm)\lhu \nomh]$&\\	
& iff\! & $\forall \nomj\forall \nomh \forall \cnomm[\nomj\rhu (\cnomm\lhu (\nomj\bullet \nomh))\leq  (\nomj\rhu \cnomm)\lhu \nomh].$&\\			
			\end{tabular}
\end{exa}

\begin{rem}\label{Rem:Additional:Rule}
As discussed in Remark \ref{Rem:Suzuki}, it would be possible to accommodate, both in the definition of the inductive formulas and in ALBA, $n$-ary connectives which preserve meets and joins in the product. We could add sets $\mathcal{F}_{p} \subseteq \mathcal{F}$  and $\mathcal{G}_{p} \subseteq \mathcal{G}$ of such connectives. In every $(\mathcal{F}, \mathcal{G})$-expanded lattice $\mathbb{A}$ we would have $f^{\mathbb{A}}(\bigvee^{\epsilon_f}_{i \in I} \overline{a}_i) = \bigvee_{i \in I} f(\overline{a}_i)$ for all nonempty subsets $\{\overline{a}_i \mid i \in I\} \subseteq \mathbb{A}^{\epsilon_f}$ and, similarly,  $g^{\mathbb{A}}(\bigwedge^{\epsilon_g}_{i \in I} \overline{a}_i) = \bigwedge_{i \in I} g(\overline{a}_i)$ for all nonempty subsets $\{\overline{a}_i \mid i \in I\} \subseteq \mathbb{A}^{\epsilon_g}$. The accompanying expanded languages would contain connectives $\triangle_f$ and $\triangledown_g$ for the normalizations of each $f \in \mathcal{F}_{p}$ and $g \in \mathcal{G}_{p}$. The intended interpretations of these are such that $f(\overline{a}) = f(\overline{\bot}^{\epsilon_f}) \vee \triangle_f(\overline{a})$ and $g(\overline{a}) = g(\overline{\top}^{\epsilon_g}) \wedge \triangledown_d(\overline{a})$. For each $f \in \mathcal{F}_{p}$ there would be adjoints $\blacktriangledown_i$, $1 \leq i \leq n_f$, such that $\triangle_f(\overline{a}) \leq b$ iff $\overline{a} \leq^{\epsilon_f} \overline{\blacktriangledown_i b}$ iff $a_i \leq^{\epsilon_f(i)} \blacktriangledown_i b$ for each $1 \leq i \leq n_f$. Dually, for each $g \in \mathcal{G}_{p}$ there would be adjoints $\blacktriangle_i$, $1 \leq i \leq n_g$, such that $a \leq \triangle_g(\overline{b})$ iff $\overline{\blacktriangle_i a} \leq^{\epsilon_g} \overline{b} $ iff $\blacktriangledown_i a \leq^{\epsilon_g(i)} b_i$ for each $1 \leq i \leq n_g$.

The appropriate residuation (or rather, \emph{adjunction}) rules for these connective to be added to ALBA would be as follows:

\begin{center}
\AxiomC{$f(\phi_1, \ldots, \phi_n)\leq\psi$}
\RightLabel{}
\UnaryInfC{$f(\overline{\bot}^{\epsilon_f})\leq \psi\;\;\;\phi_1 \leq^{\epsilon_f(1)} \blacktriangledown^{f}_1 \psi \;\; \cdots \;\; \phi_n \leq^{\epsilon_{f}(n_f)} \blacktriangledown^{f}_{n_f} \psi$}
\DisplayProof
\end{center}

\begin{center}
\AxiomC{$\phi \leq g(\psi_1, \ldots, \psi_n)$}
\RightLabel{}
\UnaryInfC{$\phi \leq g(\overline{\top}^{\epsilon_g})\;\;\;\blacktriangle^g_1 \phi \leq^{\epsilon_g(1)} \phi_1 \;\; \cdots \;\; \blacktriangle^g_{n_g} \phi \leq^{\epsilon_{g}(n_g)} \phi_{n_g}$}
\DisplayProof
\end{center}
\end{rem}

\section{Partial correctness of ALBA}\label{Sec:Crctns:ALBA}

In this section, we prove that ALBA is partially correct, in the sense that whenever it succeeds in eliminating all propositional variables from an inequality $\phi \leq \psi$, the conjunction of the quasi-inequalities returned is equivalent to $\phi \leq \psi$ on the constructive canonical extension $\bbas$ of any  $\mathcal{L}_\mathrm{LE}$-algebra $\bba$, for an arbitrarily fixed language $\mathcal{L}_\mathrm{LE}$.
Our treatment  follows the same structure as the proof of correctness in \cite{ConPal13}. Accordingly, we rely on \cite{ConPal13} whenever possible, and mostly expand on the parts addressing the differences.

Fix a  $\mathcal{L}_\mathrm{LE}$-algebra $\bba = (L, \mathcal{F}^\bba, \mathcal{G}^\bba)$ for the remainder of this section. We first give the statement of the correctness theorem and its proof, and subsequently prove the lemmas needed in the proof.

\begin{thm}[Partial Correctness]\label{Crctns:Theorem} If $\mathsf{ALBA}$ succeeds in reducing an $\mathcal{L}_\mathrm{LE}$-inequality $\phi \leq \psi$ and yields $\mathsf{ALBA}(\phi \leq \psi)$, then $\bbas \models \phi \leq \psi$ iff $\bbas \models \mathsf{ALBA}(\phi \leq \psi)$.
\end{thm}
\begin{proof}
Let $\phi_i \leq \psi_i$, $1 \leq i \leq n$ be the quasi-inequalities produced by preprocessing $\phi \leq \psi$. Consider the chain of statements (\ref{Crct:Eqn1}) to (\ref{Crct:Eqn5}) below. The proof will be complete if we can show that they are all equivalent.
\begin{eqnarray}
&&\bbas \models \phi \leq \psi\label{Crct:Eqn1}\\
&&\bbas \models \phi_i \leq \psi_i,  \quad 1 \leq i \leq n \label{Crct:Eqn2}\\
&&\bbas \models \bigamp \varnothing \Rightarrow  \phi_i \leq \psi_i, \quad 1 \leq i \leq n \label{Crct:Eqn3}\\
&&\bbas \models \bigamp S_i \Rightarrow \mathsf{Ineq}_i, \quad 1 \leq i \leq n \label{Crct:Eqn4}\\
& &\bbas \models \mathsf{ALBA}(\phi \leq \psi)\label{Crct:Eqn5}
\end{eqnarray}
For the equivalence of (\ref{Crct:Eqn1}) and (\ref{Crct:Eqn2}) we need to verify that the rules for the elimination of uniform variables, distribution and splitting preserve validity on $\bbas$. Distribution and splitting are immediate. As to elimination, if $\alpha(p) \leq \beta(p)$ is positive in $p$, then for all $a \in \bba$ it is the case that $\alpha(a) \leq \alpha(\bot)$ and  $\beta(\bot) \leq \beta(a)$. Hence if $\alpha(\bot) \leq \beta(\bot)$ then $\alpha(a) \leq \beta(a)$, and hence $\bbas \models \alpha(p) \leq \beta(p)$ iff $\bbas \models \alpha(\bot) \leq \beta(\bot)$. The case for $\gamma(p) \leq \delta(p)$ negative in $p$ is similar.

That (\ref{Crct:Eqn2}) and (\ref{Crct:Eqn3}) are equivalent is immediate. The bi-implication between (\ref{Crct:Eqn3}) and (\ref{Crct:Eqn4}) follows from proposition \ref{Rdctn:Rls:Crctnss:Prop}, while (\ref{Crct:Eqn4}) and (\ref{Crct:Eqn5}) are the same by definition.
\end{proof}

\begin{lem}[Distribution lemma]\label{Distribution:Lemma}
If  $\phi(!x), \psi(!x), \xi(!x), \chi(!x) \in \mathcal{L}_\mathrm{LE}^{+}$  and $\varnothing \neq \{ a_j \}_{j \in I} \subseteq \bbas$, then
\begin{enumerate}
\item $\phi(\bigvee_{j \in I} a_j) = \bigvee\{ \phi(a_j) \mid j \in I \}$, when $+x \prec +\phi(!x)$ and in $+\phi(!x)$ the branch ending in $+x$ is SAC;
\item $\psi(\bigwedge_{j \in I} a_j) = \bigvee\{ \psi(a_j) \mid j \in I \}$, when $-x \prec +\psi(!x)$ and in $+\psi(!x)$ the branch ending in $-x$ is SAC;
\item $\xi(\bigwedge_{j \in I} a_j) = \bigwedge\{ \xi(a_j) \mid j \in I \}$, when $-x \prec -\xi(!x)$ and in $-\xi(!x)$ the branch ending in $-x$ is SAC;
\item $\chi(\bigvee_{j \in I} a_j) = \bigwedge\{ \chi(a_j) \mid j \in I \}$, when $+x \prec -\chi(!x)$ and in $-\chi(!x)$ the branch ending in $+x$ is SAC.
\end{enumerate}
\end{lem}
\begin{proof}
The proof is very similar to \cite[Lemma 6.2]{ConPal13} and proceeds by simultaneous induction on $\phi$, $\psi$, $\xi$ and $\chi$. The base cases for $\bot$, $\top$, and $x$, when applicable, are trivial. We check the inductive cases for $\phi$, and list the other inductive cases, which all follow in a similar way.

\begin{description}
\item[$\phi$ of the form] ${ }$

\begin{description}
    \item[$f(\phi_1,\ldots,\phi_i(!x),\ldots,\phi_{n_f})$ with $f\in \mathcal{F}$ and $\epsilon_f(i) = 1$:] By the assumption of a unique occurrence of $x$ in $\phi$, the variable $x$ occurs in $\phi_i$ for exactly one index $1\leq i\leq n_f$. %Without loss of generality
        The assumption that $\epsilon_f(i) = 1$ implies that $+ x \prec + \phi_i$.
        Then
        \begin{center}
        \begin{tabular}{r c l l}
        $\phi(\bigvee_{j \in I} a_j)$ &$ =$&$ f(\phi_1,\ldots,\phi_i(\bigvee_{j \in I} a_j)\ldots,\phi_{n_f})$\\
        &$ =$&$  f(\phi_1,\ldots, \bigvee_{j \in I}\phi_i( a_j)\ldots,\phi_{n_f})$\\
         &$ =$&$  \bigvee_{j \in I}f(\phi_1,\ldots,\phi_i(a_j)\ldots,\phi_{n_f})$ & ($\varnothing \neq \{ a_j \}_{j \in I}$)\\
         &$ =$&$  \bigvee_{j \in I} \phi(a_j)$,\\
          \end{tabular}
          \end{center}where the second equality holds by the inductive hypothesis, since the branch of $+\phi$ ending in $+x$ is SAC, and it traverses $+\phi_i$.
\item[$f(\phi_1,\ldots,\psi_i(!x),\ldots,\phi_{n_f})$ with $f\in \mathcal{F}$ and $\epsilon_f(i) = \partial$:] By the assumption of a unique occurrence of $x$ in $\phi$, the variable $x$ occurs in $\psi_i$ for exactly one index $1\leq i\leq n_f$. The assumption that $\epsilon_f(i) = \partial$ implies that $- x \prec + \psi_i$. %Without loss of generality assume that $+ x \prec + \phi_i$.
    Then \begin{center}
        \begin{tabular}{r c l l}
        $\phi(\bigvee_{j \in I} a_i)$ &$ =$&$ f(\phi_1,\ldots,\psi_i(\bigvee_{j \in I} a_j)\ldots,\phi_{n_f})$\\
        &$ =$&$  f(\phi_1,\ldots, \bigwedge_{j \in I}\psi_i( a_j)\ldots,\phi_{n_f})$\\
        &$ =$&$  \bigvee_{j \in I}f(\phi_1,\ldots,\psi_i(a_j)\ldots,\phi_{n_f})$ &  ($\varnothing \neq \{ a_j \}_{j \in I}$)\\
         &$ =$&$  \bigvee_{j \in I} \phi(a_j)$,\\
          \end{tabular}
          \end{center}
     where the second equality holds by the inductive hypothesis, since the branch of $+\phi$ ending in $+x$ is SAC, and it traverses $-\psi_i$.
    %
   % \begin{description}
    %\item[$\phi' \circ \phi''$:] By the assumption of a unique occurrence of $x$ in $\phi$, the variable $x$ occurs in exactly one of $\phi'$ or $\phi''$. Without loss of generality assume that $+ x \prec + \phi'$. Then  $\phi(\bigvee_{i \in I} a_i) = \phi'(\bigvee_{i \in I} a_i) \circ \phi'' = \bigvee_{i \in I} \phi'( a_i) \circ \phi'' = \bigvee_{i \in I}( \phi'( a_i) \circ \phi'') = \bigvee_{i \in I} \phi(a_i)$ where the second equality holds by the inductive hypothesis, since the branch of $+\phi$ ending in $+x$ is SLR, and it traverses $+\phi'$.

    %\item[$\Diamond \phi'(x)$:] Thus $+x \prec +\phi'(x)$ and in $+\phi'(x)$ the branch ending in $+x$ is SLR. So by the inductive hypothesis $\phi'(\bigvee_{i \in I} a_i) = \bigvee_{i \in I} \phi'( a_i)$. Hence $\Diamond \phi'(\bigvee_{i \in I} a_i) = \Diamond \bigvee_{i \in I} \phi'( a_i) = \bigvee_{i \in I} \Diamond \phi'( a_i)$.

    %\item[$\lhd \phi'(x)$:] Thus $+x \prec -\phi'(x)$ and in $-\phi'(x)$ the branch ending in $+x$ is SLR. So by the inductive hypothesis $\phi'(\bigvee_{i \in I} a_i) = \bigwedge_{i \in I} \phi'( a_i)$. Hence $\lhd \phi'(\bigvee_{i \in I} a_i) = \lhd \bigwedge_{i \in I} \phi'( a_i) = \bigvee_{i \in I} \lhd \phi'( a_i)$.

    %\item[$\phi' \circ \phi''$:] Similar to the case for $\phi' \vee \phi''$ using the distributivity of $\circ$ over $\vee$.

    \end{description}
\item[$\psi$ of the form]  $f(\psi_1,\ldots,\psi_i(!x),\ldots,\psi_{n_f})$ with $f\in \mathcal{F}$ and $\epsilon_f(i) = 1$ or $f(\psi_1,\ldots,\xi_i(!x),\allowbreak \ldots,\phi_{n_f})$ with $f\in \mathcal{F}$ and $\epsilon_f(i) = \partial$.
\item[$\xi$ of the form] $g(\xi_1,\ldots,\xi_i(!x),\ldots,\xi_{n_g})$ with $g\in \mathcal{G}$ and $\epsilon_g(i) = 1$ or $g(\xi_1,\ldots,\chi_i(!x),\allowbreak \ldots,\xi_{n_g})$ with $g\in \mathcal{G}$ and $\epsilon_g(i) = \partial$. %$\Diamond \psi'(x)$, $\lhd \psi'(x)$, or $\phi'(x) \circ \phi''(x)$.

%\item[$\xi$ of the form]  %$\Box \xi'(x)$, $\rhd \xi'(x)$, or $\xi' \star \xi''$.

\item[$\chi$ of the form] $g(\chi_1,\ldots,\chi_i(!x),\ldots,\chi_{n_g})$ with $g\in \mathcal{G}$ and $\epsilon_g(i) = 1$ or $g(\chi_1,\ldots,\phi_i(!x), \allowbreak \ldots,\xi_{n_g})$ with $g\in \mathcal{G}$ and $\epsilon_g(i) = \partial$.
\qedhere
  %$\Box \chi'(x)$, $\rhd \chi'(x)$, or $\chi'(x) \star \chi''(x)$.
\end{description}
\end{proof}

%\marginpar{\raggedright \tiny{check notation for assignments ($v$ or $V$)? and $\mathcal{L}$ or LML?}}
\begin{lem}[Right Ackermann Lemma]\label{Ackermann:Right:Lemma}
Let $\alpha_1, \ldots, \alpha_n \in \mathcal{L}_\mathrm{LE}^{+}$ with $p \not \in \mathsf{PROP}(\alpha_i)$, $1 \leq i \leq n$, let $\beta_{1}(p), \ldots, \beta_{m}(p) \in \mathcal{L}_\mathrm{LE}^{+}$ be positive in $p$, and let $\gamma_{1}(p), \ldots, \gamma_{m}(p) \in \mathcal{L}_\mathrm{LE}^{+}$ be negative in $p$. Let $v$ be any assignment on the constructive canonical extension of an $\mathcal{L}_\mathrm{LE}$-algebra $\bba$. Then
\[
\bbas, v \models \beta_i(\alpha_1 \vee \cdots \vee \alpha_n / p) \leq \gamma_i(\alpha_1 \vee \cdots \vee \alpha_n / p), \textrm{ for all } 1 \leq i \leq m
\]
iff there exists a variant $v' \sim_p v$ such that
\[
\bbas, v' \models \alpha_i \leq p \textrm{ for all } 1 \leq i \leq n, \textrm{ and } \bba, v' \models \beta_i(p) \leq \gamma_i(p), \textrm{ for all } 1 \leq i \leq m.
\]
\end{lem}
\begin{proof}
For the implication from top to bottom, let $v'(p) = v(\alpha_1 \vee \cdots \vee \alpha_n)$. Since the $\alpha_i$ does not contain $p$, we have $v(\alpha_i) = v'(\alpha_i)$, and hence that $v'(\alpha_i) \leq v(\alpha_1) \vee \cdots \vee v(\alpha_n) = v(\alpha_1 \vee \cdots \vee \alpha_n) = v'(p)$. Moreover, $v'(\beta_i(p)) = v(\beta_i(\alpha_1 \vee \cdots \vee \alpha_n / p)) \leq v(\gamma_i(\alpha_1 \vee \cdots \vee \alpha_n / p)) = v'(\gamma_i(p))$.

For the implication from bottom to top, we make use of the fact that the $\beta_i$ are monotone (since positive) in $p$, while the $\gamma_i$ are antitone (since negative) in $p$. Since $v(\alpha_i) = v'(\alpha_i) \leq v'(p)$ for all $1 \leq n \leq n$, we have $v(\alpha_1 \vee \cdots \vee \alpha_n) \leq v'(p)$, and hence $v(\beta_i(\alpha_1 \vee \cdots \vee \alpha_n / p)) \leq v'(\beta_i(p)) \leq v'(\gamma_i(p)) \leq (\gamma_i(\alpha_1 \vee \cdots \vee \alpha_n / p))$.
\end{proof}

The proof of the following version of the lemma is similar.

\begin{lem}[Left Ackermann Lemma]\label{Ackermann:Left:Lemma}
Let $\alpha_1, \ldots, \alpha_n \in \mathcal{L}_\mathrm{LE}^{+}$ with $p \not \in \mathsf{PROP}(\alpha_i)$, $1 \leq i \leq n$, let $\beta_{1}(p), \ldots, \beta_{m}(p) \in \mathcal{L}_\mathrm{LE}^{+}$ be negative in $p$, and let $\gamma_{1}(p), \ldots, \gamma_{m}(p) \in \mathcal{L}_\mathrm{LE}^{+}$ be positive in $p$. Let $v$ be any assignment on the constructive canonical extension of an $\mathcal{L}_\mathrm{LE}$-algebra $\bba$. Then
\[
\bbas, v \models \beta_i(\alpha_1 \wedge \cdots \wedge \alpha_n / p) \leq \gamma_i(\alpha_1 \wedge \cdots \wedge \alpha_n / p), \textrm{ for all } 1 \leq i \leq m
\]
iff there exists a variant $v' \sim_p v$ such that
\[
\bbas, v' \models p \leq \alpha_i  \textrm{ for all } 1 \leq i \leq n, \textrm{ and } \bba, v' \models \beta_i(p) \leq \gamma_i(p), \textrm{ for all } 1 \leq i \leq m.
\]
\end{lem}

\begin{prop}\label{Rdctn:Rls:Crctnss:Prop}
If a system $(S,\mathsf{Ineq})$ is obtained from a system $(S_0,\mathsf{Ineq}_0)$ by the application of reduction rules, then
\[
\bbas \models \forall \overline{var_0} \: [\bigamp S_0 \Rightarrow \mathsf{Ineq}_0] \textrm{ \ iff \ } \bbas \models \forall \overline{var} \: [\bigamp S \Rightarrow \mathsf{Ineq}],
\]
where $\overline{var_0}$ and $\overline{var}$ are the vectors of all variables occurring in $\bigamp S_0 \Rightarrow \mathsf{Ineq}_0$ and $\bigamp S \Rightarrow \mathsf{Ineq}$, respectively.
\end{prop}
\begin{proof}
It is sufficient to verify that each rule preserves this equivalence, i.e., that if $S$ and $\mathsf{Ineq}$ are obtained from $S'$ and $\mathsf{Ineq}'$ by the application of a single transformation rule then
\[
\bbas \models \forall \overline{var'} \: [\bigamp S' \Rightarrow \mathsf{Ineq}'] \textrm{ \ iff \ } \bbas \models \forall \overline{var} \: [\bigamp S \Rightarrow \mathsf{Ineq}].
\]
\paragraph{Left-positive approximation rule:} Let $\mathsf{Ineq}$ be $ \phi'(\gamma / !x) \leq \psi$, with $+x \prec +\phi'(!x)$, and the branch of $+\phi'(!x)$ starting at $+x$ being SAC. Then, under any assignment to the occurring variables, $\phi'(\gamma / !x) = \phi'(\bigvee \{k \in \kbbas \mid k\leq \gamma \}) = \bigvee \{\phi'(k) \mid \gamma \geq k \in \kbbas\}$, where the latter equality holds by lemma \ref{Distribution:Lemma}.1, given that $\bot\in \{k \in \kbbas \mid k\leq \gamma \}\neq \varnothing$. But then
\[
\bba \models \forall \overline{var'} \: [\bigamp S' \Rightarrow \phi'(\gamma/!x) \leq \psi] \textrm{ \ iff \ } \bba \models \forall \overline{var'} \forall \nomj \: [\bigamp S' \& \; \nomj \leq \gamma \Rightarrow \phi'(\nomj) \leq \psi].
\]
The other approximation rules are justified in a similar manner, appealing to the other clauses of lemma \ref{Distribution:Lemma}.

The residuation rules for $f\in \mathcal{F}_n$ (resp.\ $g\in \mathcal{G}_n$) are justified by the fact that (the algebraic interpretation of) every such $f$ (resp.\ $g$) is a left (resp.\ right) residual in each positive coordinate and left (resp.\ right) Galois-adjoint in each negative coordinate.  % $\Diamond$, $\lhd$ and $\Box$, $\rhd$ are left and right adjoints, respectively, while $\circ$ and $\star$ are left and right residuals, respectively.
%\marginnote{make clear whether there are adjunction rules for the normalizations or not}
The residuation rules for $f\in \mathcal{F}_r$ (resp.\ $g\in \mathcal{G}_r$) are justified as follows: recall that by Lemma \ref{lemma: f and diamond f coincide on non bottom els},

\[\bbas \models f(p) = f(\bot) \lor \Diamond_f(p) \mbox{ if } \epsilon_f = 1 \quad\quad \bbas \models g(p) = g(\top) \land \Box_g(p) \mbox{ if } \epsilon_g = 1 \]
\[\bbas\models f(p) = f(\top) \lor {\lhd}_f(p) \mbox{ if } \epsilon_f = \partial \quad\quad \bbas \models g(p) = g(\bot) \land {\rhd}_g(p) \mbox{ if } \epsilon_g = \partial. \]

%which proves the soundness and invertibility of the rules above on any perfect DLR $\A$.
This, together with the the adjunction properties enjoyed by the normalized operations, guarantees that the  rules above are sound, in that they can be `derived' by means of equivalent substitution and adjunction for the normalized connectives. Below we provide  examples of how this is done for $f\in \mathcal{F}_r$ with  $\epsilon_f = \partial$ (left-hand side) and $\epsilon_f = 1$ (right-hand side).

\begin{center}
\begin{tabular}{cc}
\AxiomC{$f(\phi)\leq \psi$}
\UnaryInfC{$f(\top)\vee{\lhd}_f \phi\leq \psi$}
\UnaryInfC{$f(\top)\leq \psi\;\;\; {\lhd}_f\phi\leq \psi$}
\UnaryInfC{$f(\top)\leq\psi\;\;\; {\blhd} \psi\leq\phi$}
\DisplayProof
&
\AxiomC{$f(\phi)\leq \psi$}
\UnaryInfC{$f(\bot)\vee{\Diamond}_f\phi\leq \psi$}
\UnaryInfC{$f(\bot)\leq \psi\;\;\; {\Diamond}_f\phi\leq \psi$}
\UnaryInfC{$f(\bot)\leq\psi\;\;\;  \phi\leq \blacksquare_f\psi$}
\DisplayProof
\\
\end{tabular}
\end{center}
The Ackermann-rules are justified by the Ackermann lemmas \ref{Ackermann:Right:Lemma} and \ref{Ackermann:Left:Lemma}.
\end{proof}
\begin{rem}
\label{rem: towards constructive can}
The proof of the soundness of the approximation rules given above pivots on three ingredients: (1) the denseness of the canonical extension, guaranteeing e.g.\ that every element in $\bbas$ is the join of the closed elements below it; (2) the fact that the set of the closed elements below a given one is always nonempty, since $\bot$ always belongs to it; (3) the defining property of the SAC connectives, guaranteeing the preservation of joins of nonempty collections.  If  nominals and co-nominals are interpreted as completely join-irreducible and meet-irreducible elements respectively, as in the usual context in which ALBA aims at correspondence,  the set of the completely join-irreducible elements below a given element might be empty, and hence, these approximation rules are sound only at the price of strengthening the requirements on the Skeleton connectives different from $\Delta$-adjoints, and insisting that they be SLR, and not just SAC.  In \cite{PaSoZh15b}, more complicated approximation rules had to be given, precisely to account for correspondence in a context in which Skeleton connectives other than $\Delta$-adjoints are SAC.  % rule is not  the distributivity precisely because  we have not used the fact that nominals and co-nominals are interpreted as completely join-irreducible and meet-irreducible elements respectively. We only used the fact that completely join-irreducibles (resp.\ meet-irreducibles) completely join-generate (resp.\ meet-generate) the algebra $\bba$. This is a notable difference with the distributive setting of \cite{ConPal12}, where the soundness of the approximation rules essentially depends upon the complete join-primeness (resp.\ meet-primeness) of the interpretation of nominals and co-nominals respectively. This observation will be crucially put to use in the conclusions, where we discuss how the constructive canonicity theory of \cite{GhMe97,Suzuki13} can be derived from the results of the present paper.
\end{rem}

\begin{rem}\label{Adapt:To:Cncl:Ext:Remark}
In Section \ref{sec:canonicity} we will prove that, for each language $\mathcal{L}_\mathrm{LE}$, all $\mathcal{L}_\mathrm{LE}$-inequalities on which ALBA succeeds are canonical. For this it is necessary to show that ALBA transformations preserve validity with respect to {\em admissible} assignments on canonical extensions of LEs. Towards this, it is an easy observation that  all ALBA transformations other than the Ackermann rules preserve validity under admissible assignments.  The potential difficulty with the Ackermann rules stems from the fact that they,  unlike all other rules, involve changing the assignments to propositional variables.
\end{rem}

\section{ALBA successfully reduces all inductive inequalities}\label{Sec:Complete:For:Inductive:Section}
%The present section is based on \cite[Section 8]{CoPa11} to the present LE setting. We report on it for the sake of self-containedness.
The aim of the present section is showing that ALBA is successful on all inductive $\mathcal{L}_{\mathrm{LE}}$-inequalities, and that in fact {\em safe} and {\em pivotal}  executions suffice. The arguments presented here are largely a synthesis of those made in \cite{CoPaSoZh15} and \cite{ConPal13}. We need to provide a certain amount of detail to show that different parts of the ALBA runs shown to exist satisfy the properties of safety and pivotality, either separately or in combination.

\begin{defi}
An execution of ALBA is \emph{safe} if no side conditions (cf.\ Page \pageref{def:sidecondition}) introduced by applications of adjunction rules for  connectives in $\mathcal{F}_r\cup\mathcal{G}_r$ are further modified, except for receiving Ackermann substitutions.
\end{defi}

\begin{defi}
An $(\Omega, \epsilon)$-inductive inequality  is {\em definite} if in its critical branches, all Skeleton nodes are SAC nodes.
\end{defi}

\begin{lem}\label{Pre:Process:Lemma}
Let $\{ \phi_i \leq \psi_i\}$ be the set of inequalities obtained by
preprocessing an $(\Omega, \epsilon)$-inductive $\mathcal{L}_\mathrm{LE}$-inequality $\phi \leq
\psi$. Then each $\phi_i \leq \psi_i$ is a definite $(\Omega,
\epsilon)$-inductive inequality.
\end{lem}
\begin{proof}
Notice that the distribution during preprocessing only swap the order of Skeleton nodes on (critical) paths, and hence does not affect the goodness of critical branches. Moreover, PIA parts are entirely unaffected, and in particular the side conditions on SRR nodes of critical branches are maintained. Finally, notice that SAC nodes commute exhaustively with $\Delta$-adjoints, and hence all $\Delta$-adjoints are effectively surfaced and eliminated via splitting, thus producing definite inductive inequalities.
\end{proof}

The following definition is intended to capture the state of a system after approximation rules have been applied pivotally until no propositional variable remains in   $\mathsf{Ineq}$:

\begin{defi}\label{Stripped:Def}
Call a system $(S, \mathsf{Ineq})$ \emph{$(\Omega, \epsilon)$-stripped} if $\mathsf{Ineq}$ is pure, and for each $\xi \leq \chi \in S$ the following conditions hold:
\begin{enumerate}
\item one of  $-\xi$ and $+\chi$ is pure, and the other is $(\Omega, \epsilon)$-inductive;
%
%\item the non-pure side contains at most one critical branch;
%
\item every $\epsilon$-critical branch in $-\xi$ and $+\chi$ is PIA.
\end{enumerate}
\end{defi}

\begin{lem}\label{Stripping:Lemma}
For any definite $(\Omega, \epsilon)$-inductive inequality $\phi \leq \psi$ the system $(\emptyset, \phi \leq \psi)$ can be transformed into an $(\Omega, \epsilon)$-stripped system by the pivotal and safe application of approximation rules.
\end{lem}
\begin{proof}
By assumption, $+\phi$ and $-\psi$ are both definite $(\Omega, \epsilon)$-inductive.  Hence, for any propositional variable occurrences $p$, we can apply an approximation rule  on the first non SAC-node in the path which goes from the root to $p$. %(this is always possible thanks to the absence of $\Delta$-adjoints in the Skeleton). These applications are all pivotal. % Notice that  the nodes at which pivotal applications are effected form an antichain in the given generation tree.

Let us show that the resulting system $(S, \mathsf{Ineq})$ is $(\Omega, \epsilon)$-stripped. Clearly, the procedure reduces $\mathsf{Ineq}$ to a pure inequality. Each inequality in $S$ is generated by the application of some approximation rule, and is therefore  either of the form $\nomj \leq \alpha$ or $\beta \leq \cnomm$,  where $+\alpha$ and $-\beta$ are subtrees of $(\Omega, \epsilon)$-inductive trees,  and hence are $(\Omega, \epsilon)$-inductive, as required by item 1 of the definition.

% As to item 2,  if the non pure side contained at least two $\epsilon$-critical nodes, they would share a binary SRR ancestor, which would violate the definition of an $(\Omega, \epsilon)$-inductive tree.

Next, if in an inequality in $S$ contained  a critical variable occurrence such that its associated path has an SAC node, then, because such a path is by assumption good, this would contradict the fact that the inequality has been generated by a pivotal application of an approximation rule. This shows item 2.

Lastly, we note that no approximation rule can alter a previously introduced side condition and that therefore the rule applications described in this lemma are safe.
\end{proof}

\begin{defi}\label{Ackermann:Ready:Def}
An $(\Omega, \epsilon)$-stripped system $(S, \mathsf{Ineq})$ is \emph{Ackermann-ready} with respect to a propositional variable $p_i$ with $\epsilon_i = 1$ (respectively, $\epsilon_i = \partial$) if every inequality $\xi \leq \chi \in S$ is of one of the following forms:
\begin{enumerate}
\item $\xi \leq p$ where $\xi$ is pure (respectively, $p \leq \chi$ where $\chi$ is pure), or
\item $\xi \leq \chi$ where neither $- \xi$ nor $+ \chi$ contain any $+p_i$ (respectively, $-p_i$) leaves.
\end{enumerate}
\end{defi}

Note that the right or left Ackermann-rule (depending on whether $\epsilon_i = 1$ or $\epsilon_i = \partial$) is applicable to a system which is Ackermann-ready with respect to $p_i$. In fact, this would still have been the case had we weakened the requirement in clause 1 that $\xi$ and $\chi$ must be pure to simply require that they do not contain $p_i$.

\begin{lem}\label{Ackermann:Ready:LEmma}
If $(S, \mathsf{Ineq})$ is $(\Omega, \epsilon)$-stripped and $p_i$ is $\Omega$-minimal among propositional variables occurring in $(S, \mathsf{Ineq})$, then $(S, \mathsf{Ineq})$ can be transformed, through the application of residuation- and splitting-rules, into a system which is \emph{Ackermann-ready} with respect to $p_i$. Moreover, this can be done in such a way that
    \begin{enumerate}
    \item all side conditions introduced are pure inequalities, and
    \item if all side conditions occurring in $S$ are pure, then the all rule applications in this process are safe.
    \end{enumerate}
\end{lem}
\begin{proof}
If $\xi \leq \chi \in S$ and $- \xi$ and $+ \chi$ contain no $\epsilon$-critical $p_i$-nodes then this inequality already satisfies condition 2 of Definition \ref{Ackermann:Ready:Def}. So suppose that $- \xi$ and $+ \chi$ contain some $\epsilon$-critical $p_i$-node among them. This means $\xi \leq \chi$ is of the form $\alpha \leq \mathsf{Pure}$ with the $\epsilon$-critical $p_i$-node in $\alpha$ and $\mathsf{Pure}$ pure, or of the form $\mathsf{Pure} \leq \delta$ with $\mathsf{Pure}$ pure and the $\epsilon$-critical $p_i$-node in $\delta$. We can now prove by simultaneous induction on $\alpha$ and $\delta$ that these inequalities can be transformed into the form specified by clause 1 of definition \ref{Ackermann:Ready:Def}. %\footnote{W: Or we can just say they can be stripped off?}

The base cases are when $- \alpha = -p_i$ and $+\delta = +p_i$. Here the inequalities are in desired shape and no rules need be applied to them. We will only check a few of the inductive cases. If $- \alpha = - (\alpha_1 \vee \alpha_2)$, then %one of $-\alpha_1$ and $-\alpha_2$ contains the $\epsilon$-critical $p_i$-node and (since $- (\alpha_1 \vee \alpha_2)$ in $(\Omega, \epsilon)$-inductive, $-\vee$ is a binary SRJ-node, and $p_i$ is $\Omega$-minimal among the occurring variables) the other is pure. Since $\vee$ is commutative we may assume without loss of generality that the critical occurrence of $p_i$ is in $- \alpha_1$.
applying the $\vee$-splitting rule we transform $\alpha_1 \vee \alpha_2 \leq \mathsf{Pure}$ into $\alpha_1 \leq \mathsf{Pure}$ and $\alpha_2 \leq \mathsf{Pure}$. The resulting system is clearly still $(\Omega, \epsilon)$-stripped, and we may apply the inductive hypothesis to $\alpha_1 \leq \mathsf{Pure}$ and $\alpha_2 \leq \mathsf{Pure}$.
%\marginnote{here it needs to be edited}
If $- \alpha = - f(\overline{\alpha})$ for some $f \in \mathcal{F}_n$, then, as per definition of inductive inequalities and given that $p_i$ is by assumption $\Omega$-minimal, exactly one of the formulas in $\overline{\alpha}$  contains an $\epsilon$-critical node, and all the others (if any) are pure. Assume that the critical node is in $\alpha_j$ for $1\leq j\leq n_f$. Then, applying the appropriate $f$-residuation rule transforms $f(\overline{\alpha})\leq \mathsf{Pure}$ into either $\alpha_j\leq f^{\sharp}_j(\alpha_1, \ldots, \mathsf{Pure}, \ldots, \alpha_{n_f})$ %\mathsf{Pure}$
if $\epsilon_f = 1$ or %$\mathsf{Pure}
$f^{\sharp}_j(\alpha_1, \ldots, \mathsf{Pure}, \ldots, \alpha_{n_f}) \leq \alpha_j$ if $\epsilon_f = \partial$, yielding an $(\Omega, \epsilon)$-stripped system, and the inductive hypothesis is applicable.

%If $- \alpha = - \Diamond \alpha_1$ then applying the $\Diamond$-residuation rule transforms $\Diamond \alpha_1 \leq \mathsf{Pure}$ into $\alpha_1 \leq \blacksquare \mathsf{Pure}$, yielding once again an $(\Omega, \epsilon)$-stripped system, and the inductive hypothesis is applicable.
%The other cases for $-\alpha$ are when it is of the form $- \lhd \alpha_1$ or $- (\alpha_1 \circ \alpha_2)$. These follow similarly.

The case when $+ \delta = + g(\overline{\delta})$ with $g \in \mathcal{G}_n$ is similar.
%If $+ \delta = + g(\overline{\delta})$, then, as per definition of inductive inequalities and given that $p_i$ is by assumption $\Omega$-minimal, exactly one of the formulas in $\overline{\delta}$  contains an $\epsilon$-critical node, and all the others (if any) are pure. Assume that the critical node is in $\delta_j$ for $1\leq j\leq n_g$. Then, applying the appropriate $g$-residuation rule transforms $\mathsf{Pure}\leq g(\overline{\delta})$ into either $\mathsf{Pure}\leq \delta_j$  if $\epsilon_g = 1$ or $\delta_j\leq \mathsf{Pure}$  if $\epsilon_f = \partial$, yielding an $(\Omega, \epsilon)$-stripped system, and the inductive hypothesis is applicable.
%
%If $+ \delta = + \rhd \delta_1$, the applying the $\rhd$-residuation rule transforms $\mathsf{Pure} \leq \rhd \delta_1$ into $\delta_1 \leq \blacktriangleright \mathsf{Pure}$ yielding a system $(\Omega, \epsilon)$-stripped system, to which the inductive hypothesis (on $\alpha$, this time) is applicable.
%
%The other inductive cases for $+ \delta$ are when it is of the form $+ (\delta_1 \wedge \delta_2)$ or $+ \Box \delta_1$, and these follow similarly.
We still need to consider the cases when $- \alpha$ is $- f(\alpha')$ for some $f \in \mathcal{F}_r$, and when $+ \delta$ is $+ g(\delta')$ for some $g \in \mathcal{G}_r$. As these cases are order-dual, we will only treat the latter. If $\epsilon_g = 1$, then the residuation rule replaces $\mathsf{Pure} \leq g(\delta')$ with the two inequalities $\mathsf{Pure} \leq g(\top)$ and $\Diamondblack_{g} \mathsf{Pure} \leq \delta'$, to which the inductive hypothesis is applicable. Notice that the introduced side condition is pure. If $\epsilon_g = \partial$, then the residuation rule replaces $\mathsf{Pure} \leq g(\delta')$ with  $\mathsf{Pure} \leq g(\bot)$ and $\delta \leq {\brhd} \mathsf{Pure} \leq \delta'$, to which the inductive hypothesis is applicable. The introduced side condition is again pure.

In all the cases considered above, rules are applied only to non-pure inequalities, so the safety of these rule applications follow immediately from the assumption that all occurring side conditions are pure.
\end{proof}

\begin{lem}\label{App:Of:Ackermann:Lemma}
Applying the appropriate Ackermann-rule with respect to $p_i$ to an $(\Omega, \epsilon)$-stripped system which is Ackermann-ready with respect to $p_i$, again yields an $(\Omega, \epsilon)$-stripped system.
\end{lem}
\begin{proof}
Let $(S, \phi \le \psi)$ be an $(\Omega, \epsilon)$-stripped system which is Ackermann-ready with respect to $p_i$. We only consider the case in which the right Ackermann-rule is applied, the case for the left Ackermann-rule being dual. This means that $S = \{\alpha_k\leq p \mid 1 \leq k \leq n \} \cup \{\beta_j(p_i) \leq \gamma_{j}(p_i) \mid 1 \leq j \leq m \}$ where the $\alpha$s are pure and the $-\beta$s and $+\gamma$s contain no $+ p_i$ nodes. Let us denote the pure formula $\bigvee_{k=1}^{n} \alpha_k$ by $\alpha$. It is sufficient to show that for each $1 \leq j \leq m$, the trees $- \beta(\alpha / p_i)$ and $+ \gamma(\alpha / p_i)$ satisfy the conditions of Definition \ref{Stripped:Def}. Conditions 2  follows immediately once we notice that, since $\alpha$ is pure and is being substituted everywhere for variable occurrences corresponding to non-critical nodes, $- \beta(\alpha / p_i)$ and $+ \gamma(\alpha / p_i)$ have exactly the same $\epsilon$-critical paths as $- \beta(p_i)$ and $+ \gamma(p_i)$, respectively. Condition 1, namely that $- \beta(\alpha / p_i)$ and $+ \gamma(\alpha / p_i)$ are $(\Omega, \epsilon)$-inductive, also follows using additionally the observation that all new paths that arose from the substitution are variable free.
\end{proof}

\begin{thm}\label{thm:ALBA:Ind:Succss}
$\mathrm{ALBA}$ succeeds on all inductive inequalities, and safe, pivotal executions suffice.
\end{thm}
\begin{proof}
Let $\phi \leq \psi$ be an $(\Omega, \epsilon)$-inductive inequality. By Lemma \ref{Pre:Process:Lemma}, applying preprocessing yields a finite set of definite $(\Omega, \epsilon)$-inductive inequalities, each of which gives rise to  an initial system $(\emptyset, \phi' \leq \psi')$. Since preprocessing does not introduce side conditions, each initial system is free of them. By Lemma \ref{Stripping:Lemma}, pivotal applications of the approximation rules convert this system into an $(\Omega, \epsilon)$-stripped system, say $(S_1, \phi'' \leq \psi'')$. These systems are still free of side conditions, since approximation rules do not introduce side conditions. Now pick any $\Omega$-minimal variable occurring in $(S_1, \phi'' \leq \psi'')$, say $p$. By lemma \ref{Ackermann:Ready:LEmma}, the system can be made Ackermann-ready with respect to $p$ by applying residuation and splitting rules while introducing only pure side conditions. Now apply the appropriate Ackermann-rule to eliminate $p$ from the system. Since all occurring side conditions are pure, this is a safe rule application. By lemma \ref{App:Of:Ackermann:Lemma}, the result is again an $(\Omega, \epsilon)$-stripped system, now containing one propositional variable less, and to which lemma \ref{Ackermann:Ready:LEmma} can again be applied. This process is iterated until all occurring propositional variables are eliminated and a pure system is obtained.
\end{proof}

\section{Constructive canonicity}\label{sec:canonicity}
This section is devoted to proving that all inequalities on which constructive ALBA succeeds are {\em constructively canonical} (i.e.\ their validity is preserved under the constructive canonical extension). The arguments presented here amalgamate the constructive generalizations of the canonicity arguments used in \cite{CoPaSoZh15,PaSoZh15b} and \cite{ConPal13}. Indeed,   in the current setting, we need to appeal to (the constructive counterparts of) both the pivotality and safety of ALBA execution. %However, as we will see this generalization and merging is effected in a modular fashion

Fix an $\mathcal{L}_\mathrm{LE}$-algebra  $\bba$ and let $\bbas$ be its canonical extension. We write $\bbas \models_{\mathbb{A}} \phi \leq \psi$ to indicate that $\bbas, v \models \phi \leq \psi$ for all \emph{admissible} assignments $v$, as defined in Subsection \ref{Subsec:Expanded:Land}, page \pageref{admissible:assignment}. Recall that pivotal executions of ALBA are defined on page \pageref{pivotal:approx:rule:application} in Section \ref{sec:constructive ALBA}.

\begin{thm}\label{Thm:ALBA:Canonicity}
All $\mathcal{L}_\mathrm{LE}$-inequalities on which ALBA succeeds safely and pivotally are canonical.
\end{thm}
\begin{proof}
Let $\phi \leq \psi$ be an $\mathcal{L}_\mathrm{LE}$-inequality for which some execution of ALBA which is both safe and pivotal exists. The required canonicity proof is summarized in the following U-shaped diagram:\\

\begin{center}
\begin{tabular}{l c l}
$\bba \ \ \models \phi \leq \psi$ & &$\bbas \models \phi \leq \psi$\\
$\ \ \ \ \ \ \ \ \ \ \ \Updownarrow$ \\
$\bbas \models_{\bba} \phi \leq \psi$ & & \ \ \ \ \ \ \ \ \ \ \  $\Updownarrow $\\
$\ \ \ \ \ \ \ \ \ \ \ \Updownarrow$\\
$\bbas \models_{\bba} \mathsf{ALBA}(\phi \leq \psi)$
&\ \ \ $\Leftrightarrow$ \ \ \ &$\bbas \models \mathsf{ALBA}(\phi \leq \psi)$
\end{tabular}
\end{center}
The uppermost bi-implication on the left is given by the definition of validity on algebras and $\bba$ being  a subalgebra of $\bbas$. The lower bi-implication on the left is given by Proposition \ref{Propo:Crtns:CanExtns} below. The horizontal bi-implication follows from the facts that, by assumption, $\mathsf{ALBA}(\phi \leq \psi)$ is pure, and that, when restricted to pure formulas, the ranges of admissible and arbitrary assignments coincide. The bi-implication on the right is given by Theorem \ref{Crctns:Theorem}.
\end{proof}
Towards the proof of  Proposition \ref{Propo:Crtns:CanExtns}, the following definitions and lemmas will be useful:
\begin{defi}[Syntactically closed and open $\mathcal{L}_{\mathrm{LE}}^+$-formulas]\label{def:syn:closed:and:open}
\begin{enumerate}
\item A $\mathcal{L}_{\mathrm{LE}}^+$formula is \emph{syntactically closed} if all occurrences of nominals, $f^\sharp_i$ for any $f\in \mathcal{F}_n$ with $\varepsilon_f(i)=\partial$, $g_i^{\flat}$ for any $g\in \mathcal{G}_n$ with $\varepsilon_g(i)=1$, $\blacktriangleleft_{f}$ for any $f\in \mathcal{F}_r$ with $\varepsilon_f=\partial$,  $\Diamondblack_{g}$ for any $g\in \mathcal{G}_r$ with $\varepsilon_g=1$ are positive,
and all occurrences of co-nominals, $f^\sharp_i$ for any $f\in \mathcal{F}_n$ with $\varepsilon_f(i)=1$, $g_i^{\flat}$ for any $g\in \mathcal{G}_n$ with $\varepsilon_g(i)=\partial$, $\blacktriangleright_{g}$ for any $g\in \mathcal{G}_r$ with $\varepsilon_g=\partial$, and  $\blacksquare_{f}$ $f\in \mathcal{F}_r$ with $\varepsilon_f=1$ are negative;

\item A $\mathcal{L}_{\mathrm{LE}}^+$formula is \emph{syntactically closed} if all occurrences of nominals, $f^\sharp_i$ for any $f\in \mathcal{F}_n$ with $\varepsilon_f(i)=\partial$, $g_i^{\flat}$ for any $g\in \mathcal{G}_n$ with $\varepsilon_g(i)=1$, $\blacktriangleleft_{f}$ for any $f\in \mathcal{F}_r$ with $\varepsilon_f=\partial$,  $\Diamondblack_{g}$ for any $g\in \mathcal{G}_r$ with $\varepsilon_g=1$ are negative,
and all occurrences of co-nominals, $f^\sharp_i$ for any $f\in \mathcal{F}_n$ with $\varepsilon_f(i)=1$, $g_i^{\flat}$ for any $g\in \mathcal{G}_n$ with $\varepsilon_g(i)=\partial$, $\blacktriangleright_{g}$ for any $g\in \mathcal{G}_r$ with $\varepsilon_g=\partial$, and  $\blacksquare_{f}$ $f\in \mathcal{F}_r$ with $\varepsilon_f=1$ are positive.
\end{enumerate}
A $\mathcal{L}_{\mathrm{LE}}^+$formula which is both syntactically open and syntactically closed is called \emph{syntactically clopen}.

A system $S$ of $\mathcal{L}_\mathrm{LE}^{+}$-inequalities  is {\em compact-appropriate} if the left-hand side of each non-pure inequality in $S$ is syntactically closed and the  right-hand side of each non-pure inequality in $S$ is syntactically open.
\end{defi}

\begin{defi}
A system $S$ of $\mathcal{L}_\mathrm{LE}^{+}$-inequalities  is {\em topologically adequate} when the following conditions hold: for every $f\in \mathcal{F}_r$ and $g\in \mathcal{G}_r$,
\begin{enumerate}
\item if $\varepsilon_f = 1$ and $\phi\leq \blacksquare_f \psi$ is in $S$, then $f(\bot)\leq \psi$ is in $S$;
\item if $\varepsilon_g = 1$ and $\Diamondblack_g\phi\leq  \psi$ is in $S$, then $g(\top)\geq \phi$ is in $S$;
\item if $\varepsilon_f = \partial$ and ${\blacktriangleleft}_f\phi\leq \psi$ is in $S$, then $f(\top)\leq \phi$ is in $S$;
\item if $\varepsilon_g = \partial$ and $\phi\leq {\blacktriangleright}_g \psi$ is in $S$, then $g(\bot)\geq \psi$ is in $S$.
\end{enumerate}
\end{defi}

\begin{lem}
\label{prop:top:adequacy: invariant}
Topological adequacy is an invariant of safe executions of ALBA.%\marginnote{I think we do not need to add pivotal here, do you agree?}W: Agree.
\end{lem}
\begin{proof}
Preprocessing vacuously preserves the topological adequacy of any input inequality. The topological adequacy is vacuously satisfied up to the first application of an  adjunction rule  introducing any connective $\blacksquare_f$ for  $f\in \mathcal{F}_r$ with $\varepsilon_f = 1$, or $\Diamondblack_g$ for  $g\in \mathcal{G}_r$ with $\varepsilon_g = 1$, or $\blhd$ for  $f\in \mathcal{F}_r$ with $\varepsilon_f = \partial$, or $\brhd$ for  $g\in \mathcal{G}_r$ with $\varepsilon_g = \partial$. Each such application introduces two inequalities, one of which contains the new black connective, and the other one is exactly the side condition required by the definition of topological adequacy for the first inequality to be non-offending. Moreover, at any later stage, safe executions of ALBA do not modify the side conditions, except for substituting minimal valuations. This, together with the fact that ALBA does not contain any rules which allow to manipulate any of $\blacksquare_g, \Diamondblack_f, \blhd, \brhd$, guarantees the preservation of topological adequacy. Indeed,  if e.g.\ $f(\bot)\leq \psi$ and $\phi\leq \blacksquare_f\psi$ are both in a topologically  adequate quasi-inequality, then the variables occurring in $\psi$ in both inequalities have the same polarity, and in a  safe execution, the only way in which they could be modified is if they  both receive the same minimal valuations under applications of Ackermann rules. Hence,  after such an application, they would respectively be transformed into $f(\bot)\leq \psi'$ and $\phi'\leq \blacksquare_f\psi'$ for the {\em same} $\psi'$. Thus, the topological adequacy of the quasi-inequality is preserved.
\end{proof}

\begin{lem}
\label{lemma:good:shape: invariant}
Compact-appropriateness is an invariant of ALBA$^e$ executions.
\end{lem}
\begin{proof}
The proof proceeds by checking, case-by-case, that each rule of ALBA preserves compact-appropriateness, and is entirely analogous to the proof of \cite[Lemma 9.5]{ConPal12}.
\end{proof}

\begin{lem}\label{Syn:Shape:Lemma}
If the system $(S, \mathsf{Ineq})$ is obtained by running ALBA pivotally on some $\mathcal{L}_\mathrm{LE}$-inequality $\phi \leq \psi$, then $(S, \mathsf{Ineq})$ is compact-appropriate and topologically adequate.
\end{lem}
\begin{proof}
Since $\phi \leq \psi$ comes from the base language $\mathcal{L}_\mathrm{LE}$, it is immediate that $\phi$ and $\psi$ are both syntactically clopen. Since preprocessing does not introduce any symbols not in $\mathcal{L}_\mathrm{LE}$, in any inequality $\phi' \leq \psi'$ resulting from the preprocessing, $\phi'$ and $\psi'$ are both syntactically clopen. Thus, the claim holds for each initial system $(\varnothing, \phi' \leq \psi')$. In order to complete the proof, it now remains to check that each reduction rule preserves the desired syntactic shape. This is straightforward for the residuation rules.
%for all $f\in \mathcal{F}_n$ and $g\in \mathcal{G}_n$. The same claim about the residuation rules for $f\in \mathcal{F}_r$ and $g\in \mathcal{G}_r$ follows from Lemma \ref{lem:idealtoidealmaps}.
By way of illustration we will consider the right-negative approximation rule and the righthand Ackermann-rule.

The right-negative approximation rule transforms a system $(S, s \leq t'(\gamma / !x))$ with  $-x \prec -t'(!x)$ into $(S \cup \{\gamma \leq \cnomm \}, s \leq t'(\cnomm/!x))$. If $\gamma$ belongs to the original language, it is both syntactically clopen. Hence, $\gamma \leq \cnomm$ is of the right shape. Moreover,  $-x \prec -t'(!x)$ implies that $x$ occurs positively in $t'$, which by assumption is syntactically open. Hence, $ t'(\cnomm/!x)$ is syntactically open. If $\gamma$ does not belong to $\mathcal{L}_\mathrm{LE}$, then the assumption that all approximation rules are applied pivotally guarantees that $\gamma$ must be a conominal $\cnomn$ (which has been introduced by some previous application of the same approximation rule). Hence $\gamma\leq \cnomm$ is pure.

As for the righthand Ackermann rule, it transforms a system
\[
(\{\alpha_1 \leq p, \ldots, \alpha_n \leq p, \beta_1 \leq \gamma_1, \ldots, \beta_m \leq \gamma_m \}, \mathsf{Ineq})
\]
into
\[
(\{\beta_1(\alpha/p) \leq \gamma_1(\alpha / p), \ldots, \beta_m(\alpha / p) \leq \gamma_m(\alpha / p) \}, \mathsf{Ineq}),
\]
where $\alpha = \alpha_1 \vee \cdots \vee \alpha_n$. Firstly, note that all the pure inequalities among the $\beta_i \leq \gamma_i$ remain unaffected by the rule, and hence remain pure. For non-pure $\beta_i \leq \gamma_i$, we have by assumption that $\beta_i$ is syntactically closed and positive in $p$ while $\gamma_i$ is syntactically open and negative in $p$. Thus, in $\beta_i(\alpha / p)$ each occurrence of a symbol within any occurrence of the subformula $\alpha$ has the same polarity as it had in $\alpha$ before substitution. Hence, since $\alpha$ is syntactically closed, $\beta_i(\alpha /p)$ is syntactically closed. Similarly, in $\gamma_i(\alpha /p)$ any occurrence of a symbol within each occurrence of the subformula $\alpha$ has the opposite polarity from that which it had in $\alpha$ before substitution. Hence, $\gamma_i(\alpha /p)$ is syntactically open. \end{proof}

\begin{prop}[Correctness of safe and pivotal executions of \textrm{ALBA} under admissible assignments]\label{Propo:Crtns:CanExtns} If \textrm{ALBA} succeeds in reducing an $\mathcal{L}_\mathrm{LE}$-inequality $\phi \leq \psi$ through some safe and pivotal execution and yields $\mathsf{ALBA}(\phi \leq \psi)$, then $\bbas \models_{\bba} \phi \leq \psi$ iff $\bbas \models_{\bba} \mathsf{ALBA} (\phi \leq \psi)$.
\end{prop}
\begin{proof}
It has already been indicated in Remark \ref{Adapt:To:Cncl:Ext:Remark} that the proof is essentially the same as that of Theorem \ref{Crctns:Theorem}. The only difficulty that arises  is that the Ackermann-rules are generally not invertible under admissible assignments (cf.\ \cite[Example 9.1]{ConPal12}). However, by Propositions \ref{Right:Ack} and \ref{Left:Ack}, in the special case that the system is topologically adequate and the left and right hand sides of all non-pure inequalities involved in the application of an Ackermann-rule are, respectively, syntactically closed and open, the rule is sound and invertible
under admissible assignments. By Lemmas \ref{prop:top:adequacy: invariant} and \ref{Syn:Shape:Lemma}, these requirements are always satisfied when the rule is applied in safe and pivotal executions of ALBA.
\end{proof}

Our main theorem  follows as a corollary of Theorem \ref{thm:ALBA:Ind:Succss} and Proposition \ref{Propo:Crtns:CanExtns}.

\begin{thm}[Main]
All inductive $\mathcal{L}_\mathrm{LE}$-inequalities are constructively canonical.
\end{thm}

\section{Conclusions and further directions}\label{Sec:Conclusions}

\paragraph{About the setting of the present paper.}
We have chosen to work in the setting of arbitrary lattices expanded with normal and/or regular operations. However, all results in this paper generalize readily to posets expanded with operations of similar order-theoretic properties. Indeed, just like the canonical extension of a lattice, the canonical extension of a poset is a complete lattice into which the poset embeds in a dense and compact way \cite{DunnGP05}. Thus, all the properties we need to prove the soundness of (constructive) ALBA are still available in the poset case. This remains true even when one takes into account that, for posets, the notion of canonical extension is parametric in the choice of a definition of filters and ideals \cite{Mo14,gehrke2013delta1}.

On the other hand, the choice of the general (non-distributive) lattice setting, rather than the distributive lattice setting, is particularly advantageous for the formulation of a constructive version of ALBA. Indeed, the approximation rules employed in ALBA in the distributive setting \cite{ConPal12} make essential use of the fact that the canonical extension there is completely join-generated (resp.\ meet-generated) by its completely join-{\em prime} (resp.\ meet-{\em prime}) elements. % and, moreover, that these elements are completely join-prime (resp.\ meet-prime).
The approximation rules of distributive ALBA rely on the primeness of nominals and co-nominals and therefore fail utterly in a constructive environment, where this generation by sets of join- and meet-prime elements is unavailable. In the (non-constructive) non-distributive setting we still have complete join- and meet-generation by completely join- and meet-irreducible elements, but these elements need no longer be completely join- and meet-prime. Moreover, the soundness of the approximation rules of non-distributive ALBA \cite{ConPal13} relies only on the fact that completely join- and meet-irreducibles are complete join- and meet generators, and remain sound when we replace them with any other sets of complete join- and meet generators. Such a replacement is, in fact, all we needed to do in order to salvage the soundness of these rules for the constructive case.

Our choice to extend our treatment to {\em regular} modal operators  reflects the importance of these operators  from the viewpoint of applications in formal philosophy,  AI, and the social sciences. Indeed,  many influential and very diversely motivated formalisms  have appeared, also very recently (STIT logics \cite{STIT},  logics for conditionals \cite{stalnaker1968conditionals}, and for agent organizations \cite{dignum2012logic} are a sample from the three areas), in which necessitation might not be plausible. Also in classical modal logic, necessitation has been observed to be incompatible  w.r.t.\ certain core readings of the modal operators, such as the epistemic and the deontic (cf.\ Lemmon's work \cite{lemmon1957}). Regular modal logics cannot be accommodated semantically  by  standard relational models of normal modal logics, and a variety of competing semantic frameworks for regular logics exists in the literature, ranging from impossible worlds \cite{Kr65} to neighbourhood models \cite{HH03}. This poses the problem that a separate Sahlqvist-type theory would need to be developed for each of these semantic frameworks. Via duality, the results in the present paper immediately imply the canonicity of inductive formulas of logical frameworks including regular operators in each of these semantic frameworks, which sorts out the canonicity issue simultaneously in each such setting.
The logical systems mentioned above are mostly on a classical propositional base. Specifically concerning the lattice-based setting, an epistemic interpretation has been proposed very recently for the logic of lattice expansions with unary modal operators \cite{CFPPTW}. Specifically, this logic has been proposed  as a formalism for reasoning about {\em categories} (understood as `identity classes' of objects, and formalized as concepts in formal concept analysis \cite{ganter2012formal}) and the way they are perceived, known and understood by different agents. In view of this interpretation, the treatment of regular modal operators on general lattices developed in the present paper contributes to a more accurate and flexible modelling of the epistemic reasoning of agents concerning categories.

The relational semantics developed in \cite{CFPPTW} is based on enriched polarities (referred to as polarity-based frames or enriched formal contexts) upon which certain conditions are imposed in order to ensure a duality with perfect lattice expansions. These conditions are dropped in the subsequent paper \cite{TarkPaper},  leading to a constructive adjunction with complete (but not necessarily perfect) lattices and making the results of the present paper directly applicable. The epistemic interpretation of \cite{CFPPTW} and \cite{TarkPaper} is not inherent to polarity-based semantics, which can also support other interpretations, such as those proposed in e.g.\ in \cite{Rough:concepts}, where the semantic framework of polarity-based frames for lattice-based modal logic are used to generalize Rough Set Theory \cite{pawlak2002rough} to ``rough formal concept theory''. 

Alternative, graph-based semantics for non-distributive logics, also based on constructive adjunctions with complete lattice expansions, was introduced and studied in \cite{InfoEntrop:2019} and \cite{Non:dist:sem:to:meaning} where, in contrast to the polarity based semantics, formulas again denote propositions admitting truth values rather than denoting concepts, and where a natural hyper-constructivist reading of non-distributive logics emerges. Both the polarity and graph-based semantic paradigms can be naturally generalized to many-valued version, as has been done in \cite{competing:theories}, \cite{Socio-Political} and \cite{vague:categories}.   

\paragraph{Constructive canonicity and possibility semantics.} Possibility semantics \cite{Humberstone} is an alternative and more general interpretation of classical (modal) logic on relational structures the states of which are not constrained  to encode complete information. Thus, in a possibility model $M$, the consistent set $\{\phi\in \mathcal{L}\mid M, w\Vdash \phi\}$ does not need to be  maximal  for every state $w$ in $M$. In this setting, the completeness of classical (modal) logic can be proven without appeal to any of the equivalent formulations of the axiom of choice. In \cite{Ho16}, duality, correspondence and canonicity are developed for classical modal logic on possibility frames, where  {\em filter-canonicity} is understood as the preservation of the validity of formulas from filter-descriptive (cf.\ \cite[Definition 5.39]{Ho16}) to full possibility frames.
As observed by Holliday, the BAO dually arising from the filter frame\footnote{By definition, any filter frame is a full possibility frame.} (cf.\ \cite[Definition 5.30]{Ho16}) associated with any BAO $\mathbb{A}$ is (isomorphic to) the constructive canonical extension of $\mathbb{A}$ (cf.\ \cite[Theorem 5.46]{Ho16}). Hence, it follows from the duality results developed in the same paper that filter-canonicity corresponds algebraically  to the notion of constructive canonicity. Thus, again via duality, the constructive canonicity of inductive inequalities provides a straightforward route to proving that  inductive formulas of classical (modal) logic are filter-canonical (cf.\ \cite[Theorem 7.20]{Ho16}). The  algorithmic treatment of  canonicity and correspondence for possibility semantics is developed in \cite{YZ16}. %(more below).\marginnote{not yet added (but are we going to?)\\ W: You mean discuss it? I think rather not.}
%\marginnote{to be done}
%The setting of the present paper is lattice expansions, but we could have worked in poset expansions. Then the constructive canonical extension would have been performed by taking down-directed up-sets as filters and  up-directed down-sets as ideals of posets.

\paragraph{Further directions.} In \cite{GMPTZ}, the main tools of unified correspondence theory have been used to characterize the axiomatic extensions of basic DLE-logics which admit proper display calculi in terms of a proper subclass of inductive inequalities (referred to as {\em analytic inductive inequalities}), and to effectively compute analytic rules corresponding to these axioms. Very recently, the class of analytic inductive inequalities has been shown to have interesting formal-topological properties, which can be formulated in terms of a generalized  form of (both non-constructive and constructive) canonicity, referred to as {\em slanted canonicity}   \cite{de_rudder_slanted_2020}.
These results and insights have formed the base for the systematic development of analytic sequent calculi (cf.\ e.g.\ \cite{GP:linear,GP:lattice,BGPTW,Inquisitive}) endowed with a package of basic properties (soundness, completeness, conservativity, cut elimination and subformula property) which can be  established uniformly also thanks to  unified correspondence techniques. In particular, the canonicity of inductive inequalities is crucial for the uniform proof of soundness and conservativity of these calculi. Generalizing canonicity to a constructive setting opens the possibility of endowing these calculi with a category-theoretic semantics,  an amenable environment for addressing the long-standing issue of establishing a theory of isomorphisms of derivations. In particular, in \cite{GP:lattice}, a proper display calculus has been introduced for the basic environment of the logic of general (i.e.\ not necessarily distributive) lattices, which can serve as a base for investigating these issues for LE logics.

\bibliographystyle{abbrv}
\bibliography{Constructive_Revision2020}

\newcommand{\noop}[1]{}
\begin{thebibliography}{10}

\bibitem{STIT}
N.~Belnap and M.~Perloff.
\newblock Seeing to it that: a canonical form for agentives.
\newblock {\em Theoria}, 54(3):175--199, 1988.

\bibitem{BGPTW}
M.~Bilkova, G.~Greco, A.~Palmigiano, A.~Tzimoulis, and N.~Wijnberg.
\newblock The logic of resources and capabilities.
\newblock {\em The Review of Symbolic Logic}, 11(2):371--410, 2018.
\newblock Publisher: Cambridge University Press.

\bibitem{CoCr14}
W.~Conradie and A.~Craig.
\newblock Canonicity results for mu-calculi: an algorithmic approach.
\newblock {\em Journal of Logic and Computation}, 27(3):705--748, 2017.
\newblock Publisher: Oxford University Press.

\bibitem{competing:theories}
W.~Conradie, A.~Craig, A.~Palmigiano, and N.~Wijnberg.
\newblock Modelling competing theories.
\newblock In {\em Proceedings of the 11th Conference of the European Society
  for Fuzzy Logic and Technology (EUSFLAT 2019)}, pages 721--739. Atlantis
  Press, 2019/08.

\bibitem{InfoEntrop:2019}
W.~Conradie, A.~Craig, A.~Palmigiano, and N.~M. Wijnberg.
\newblock Modelling informational entropy.
\newblock In {\em International {Workshop} on {Logic}, {Language},
  {Information}, and {Computation}}, pages 140--160. Springer, 2019.

\bibitem{CCPZ}
W.~Conradie, A.~Craig, A.~Palmigiano, and Z.~Zhao.
\newblock Constructive canonicity for lattice-based fixed point logics.
\newblock In {\em International {Workshop} on {Logic}, {Language},
  {Information}, and {Computation}}, pages 92--109. Springer, 2017.

\bibitem{CoFoPaSo15}
W.~Conradie, Y.~Fomatati, A.~Palmigiano, and S.~Sourabh.
\newblock Algorithmic correspondence for intuitionistic modal mu-calculus.
\newblock {\em Theoretical Computer Science}, 564:30--62, 2015.

\bibitem{Rough:concepts}
W.~Conradie, S.~Frittella, K.~Manoorkar, S.~Nazari, A.~Palmigiano,
  A.~Tzimoulis, and N.~M. Wijnberg.
\newblock Rough concepts.
\newblock {\em Information Sciences}, To appear.

\bibitem{CFPPTW}
W.~Conradie, S.~Frittella, A.~Palmigiano, M.~Piazzai, A.~Tzimoulis, and
  N.~Wijnberg.
\newblock Categories: {H}ow {I} {L}earned to {S}top {W}orrying and {L}ove {T}wo
  {S}orts.
\newblock In J.~V\"a\"an\"anen, {\AA}.~Hirvonen, and R.~de~Queiroz, editors,
  {\em {L}ogic, {L}anguage, {I}nformation, and {C}omputation}, pages 145--164.
  Springer, 2016. ArXiv preprint 1604.00777.

\bibitem{TarkPaper}
W.~Conradie, S.~Frittella, A.~Palmigiano, M.~Piazzai, A.~Tzimoulis, and
  N.~Wijnberg.
\newblock Toward an epistemic-logical theory of categorization.
\newblock In {\em 16th conference on {T}heoretical {A}spects of {R}ationality
  and {K}nowledge (TARK 2017)}, volume 251 of {\em {E}lectronic {P}roceedings
  in {T}heoretical {C}omputer {S}cience}, pages 170--189, 2017.

\bibitem{CoGhPa14}
W.~Conradie, S.~Ghilardi, and A.~Palmigiano.
\newblock Unified correspondence.
\newblock In A.~Baltag and S.~Smets, editors, {\em Johan van Benthem on Logic
  and Information Dynamics}, volume~5 of {\em Outstanding Contributions to
  Logic}, pages 933--975. Springer International Publishing, 2014.

\bibitem{CoGoVa06}
W.~Conradie, V.~Goranko, and D.~Vakarelov.
\newblock Algorithmic correspondence and completeness in modal logic. {I}.
  {T}he core algorithm {SQEMA}.
\newblock {\em Logical Methods in Computer Science}, 2006.

\bibitem{ConPal12}
W.~Conradie and A.~Palmigiano.
\newblock Algorithmic correspondence and canonicity for distributive modal
  logic.
\newblock {\em Annals of Pure and Applied Logic}, 163(3):338 -- 376, 2012.

\bibitem{ConPal13}
W.~Conradie and A.~Palmigiano.
\newblock Algorithmic correspondence and canonicity for non-distributive
  logics.
\newblock {\em Annals of Pure and Applied Logic}, 170(9):923--974, 2019.

\bibitem{Socio-Political}
W.~Conradie, A.~Palmigiano, C.~Robinson, A.~Tzimoulis, and N.~Wijnberg.
\newblock Modelling socio-political competition.
\newblock {\em Fuzzy Sets and Systems}, 2020.
\newblock Publisher: Elsevier.

\bibitem{vague:categories}
W.~Conradie, A.~Palmigiano, C.~Robinson, A.~Tzimoulis, and N.~M. Wijnberg.
\newblock The logic of vague categories.
\newblock {\em arXiv preprint arXiv:1908.04816}, 2019.

\bibitem{Non:dist:sem:to:meaning}
W.~Conradie, A.~Palmigiano, C.~Robinson, and N.~Wijnberg.
\newblock Non-distributive logics: from semantics to meaning.
\newblock {\em arXiv preprint arXiv:2002.04257}, 2020.

\bibitem{ConPalSou}
W.~Conradie, A.~Palmigiano, and S.~Sourabh.
\newblock Algebraic modal correspondence: {S}ahlqvist and beyond.
\newblock {\em Journal of Logical and Algebraic Methods in Programming},
  91:60--84, 2017.

\bibitem{CoPaSoZh15}
W.~Conradie, A.~Palmigiano, S.~Sourabh, and Z.~Zhao.
\newblock Canonicity and relativized canonicity via pseudo-correspondence: an
  application of {ALBA}.
\newblock Submitted, ArXiv preprint 1511.04271.

\bibitem{CoPaZh15a}
W.~Conradie, A.~Palmigiano, and Z.~Zhao.
\newblock Sahlqvist via translation.
\newblock {\em Logical Methods in Computer Science}, 15(1), 2019.

\bibitem{CoRo14}
W.~Conradie and C.~Robinson.
\newblock On {Sahlqvist} theory for hybrid logics.
\newblock {\em Journal of Logic and Computation}, 27(3):867--900, Apr. 2017.

\bibitem{DaPr90}
B.~Davey and H.~Priestley.
\newblock {\em Introduction to Lattices and Order}.
\newblock Cambridge University Press, 2002.

\bibitem{de_rudder_slanted_2020}
L.~De~Rudder and A.~Palmigiano.
\newblock Slanted canonicity of analytic inductive inequalities.
\newblock {\em arXiv:2003.12355 [math]}, Mar. 2020.
\newblock arXiv: 2003.12355.

\bibitem{dignum2012logic}
V.~Dignum and F.~Dignum.
\newblock A logic of agent organizations.
\newblock {\em Logic Journal of IGPL}, 20(1):283--316, 2012.

\bibitem{DunnGP05}
J.~M. Dunn, M.~Gehrke, and A.~Palmigiano.
\newblock Canonical extensions and relational completeness of some
  substructural logics.
\newblock {\em Journal Symbolic Logic}, 70(3):713--740, 2005.

\bibitem{Inquisitive}
S.~Frittella, G.~Greco, A.~Palmigiano, and F.~Yang.
\newblock A multi-type calculus for inquisitive logic.
\newblock In J.~V{\"a}{\"a}n{\"a}nen, {\AA}.~Hirvonen, and R.~de~Queiroz,
  editors, {\em Logic, Language, Information, and Computation: 23rd
  International Workshop, WoLLIC 2016, Puebla, Mexico, August 16-19th, 2016.
  Proceedings}, LNCS 9803, pages 215--233. Springer, 2016. ArXiv preprint
  1604.00936.

\bibitem{FrPaSa14}
S.~Frittella, A.~Palmigiano, and L.~Santocanale.
\newblock Dual characterizations for finite lattices via correspondence theory
  for monotone modal logic.
\newblock {\em Journal of Logic and Computation}, 27(3):639--678, 2017.

\bibitem{ganter2012formal}
B.~Ganter and R.~Wille.
\newblock {\em Formal concept analysis: mathematical foundations}.
\newblock Springer Science \& Business Media, 2012.

\bibitem{GeHa01}
M.~Gehrke and J.~Harding.
\newblock Bounded lattice expansions.
\newblock {\em Journal of Algebra}, 238(1):345--371, 2001.

\bibitem{gehrke2013delta1}
M.~Gehrke, R.~Jansana, and A.~Palmigiano.
\newblock {$\Delta$}1-completions of a poset.
\newblock {\em Order}, 30(1):39--64, 2013.

\bibitem{GhMe97}
S.~Ghilardi and G.~Meloni.
\newblock Constructive canonicity in non-classical logics.
\newblock {\em Annals of Pure and Applied Logic}, 86(1):1--32, 1997.

\bibitem{GorankoV06}
V.~Goranko and D.~Vakarelov.
\newblock Elementary canonical formulae: Extending {S}ahlqvist's theorem.
\newblock {\em Annals of Pure and Applied Logic}, 141(1-2):180--217, 2006.

\bibitem{GMPTZ}
G.~Greco, M.~Ma, A.~Palmigiano, A.~Tzimoulis, and Z.~Zhao.
\newblock Unified correspondence as a proof-theoretic tool.
\newblock {\em Journal of Logic and Computation}, 28(7):1367--1442, Oct. 2018.
\newblock Publisher: Oxford Academic.

\bibitem{GP:linear}
G.~Greco and A.~Palmigiano.
\newblock Linear logic properly displayed.
\newblock Submitted. ArXiv preprint 1611.04181.

\bibitem{GP:lattice}
G.~Greco and A.~Palmigiano.
\newblock Lattice {Logic} {Properly} {Displayed}.
\newblock In J.~Kennedy and R.~J. de~Queiroz, editors, {\em Logic, {Language},
  {Information}, and {Computation}}, Lecture {Notes} in {Computer} {Science},
  pages 153--169, Berlin, Heidelberg, 2017. Springer.

\bibitem{HH03}
H.~H. Hansen.
\newblock Monotonic modal logics.
\newblock {\em Master's Thesis, ILLC, University of Amsterdam}, 2003.

\bibitem{Ho16}
W.~H. Holliday.
\newblock Possibility frames and forcing for modal logic.
\newblock 2015.

\bibitem{Humberstone}
I.~Humberstone.
\newblock From worlds to possibilities.
\newblock {\em Journal of Philosophical Logic}, 10(3):313--339, 1981.

\bibitem{Jonsson94}
B.~J\'onsson.
\newblock On the canonicity of {S}ahlqvist identities.
\newblock {\em Studia Logica}, 53:473--491, 1994.

\bibitem{JoTa52}
B.~J\'onsson and A.~Tarski.
\newblock Boolean algebras with operators.
\newblock {\em American Journal of Mathematics}, 74:127--162, 1952.

\bibitem{Kr65}
S.~A. Kripke.
\newblock Semantical analysis of modal logic {II}. {N}on-normal modal
  propositional calculi.
\newblock In J.~W. Addison, A.~Tarski, and L.~Henkin, editors, {\em The Theory
  of Models}. North Holland, 1965.

\bibitem{LeRoux:MThesis:2016}
C.~le~Roux.
\newblock {Correspondence theory in many-valued modal logics}.
\newblock Master's thesis, University of Johannesburg, South Africa, 2016.

\bibitem{lemmon1957}
E.~Lemmon.
\newblock New foundations for lewis modal systems.
\newblock {\em The Journal of Symbolic Logic}, 22(02):176--186, 1957.

\bibitem{MaZh16}
M.~Ma and Z.~Zhao.
\newblock Unified correspondence and proof theory for strict implication.
\newblock {\em Journal of Logic and Computation}, 27(3):921--960, 2017.

\bibitem{Mo14}
W.~Morton.
\newblock Canonical extensions of posets.
\newblock {\em Algebra universalis}, 72(2):167--200, 2014.

\bibitem{PaSoZh15a}
A.~Palmigiano, S.~Sourabh, and Z.~Zhao.
\newblock {J{\'o}nsson-style canonicity for ALBA-inequalities}.
\newblock {\em Journal of Logic and Computation}, 27(3):817--865, 2017.

\bibitem{PaSoZh15b}
A.~Palmigiano, S.~Sourabh, and Z.~Zhao.
\newblock Sahlqvist theory for impossible worlds.
\newblock {\em Journal of Logic and Computation}, 27(3):775--816, 2017.

\bibitem{pawlak2002rough}
Z.~Pawlak.
\newblock Rough set theory and its applications.
\newblock {\em Journal of Telecommunications and information technology}, pages
  7--10, 2002.

\bibitem{SaVa89}
G.~Sambin and V.~Vaccaro.
\newblock A new proof of {S}ahlqvist's theorem on modal definability and
  completeness.
\newblock {\em Journal of Symbolic Logic}, 54(3):992--999, 1989.

\bibitem{stalnaker1968conditionals}
R.~C. Stalnaker.
\newblock A theory of conditionals.
\newblock In {\em Ifs}, pages 41--55. Springer, 1968.

\bibitem{Suzuki11a}
T.~Suzuki.
\newblock Canonicity results of substructural and lattice-based logics.
\newblock {\em The Review of Symbolic Logic}, 4(01):1--42, 2011.

\bibitem{Suzuki13}
T.~Suzuki.
\newblock A {S}ahlqvist theorem for substructural logic.
\newblock {\em The Review of Symbolic Logic}, 6(02):229--253, 2013.

\bibitem{YZ16}
Z.~Zhao.
\newblock Algorithmic correspondence and canonicity for possibility semantics.
\newblock {\em arXiv preprint arXiv:1612.04957}, 2016.

\end{thebibliography}

\appendix

\section{Appendix}\label{sec:appendix}
Section \ref{ssec: conditional esakia lemma} below  adapts  parts of \cite[Section 4.1]{CoPaSoZh15} to the present LE setting. Sections \ref{ssec:intersection lemma} and \ref{ssec:topological ackermann} are based on \cite[Section 10]{CoPaSoZh15} and adapt it to the present  setting. We include them for the sake of self-containedness.

Throughout this section, we let $\bba$ and $\bbb$ be LEs, and $\bbas$ and $\bbbs$ be their constructive canonical extensions.

\subsection{Constructive conditional Esakia lemma}\label{ssec: conditional esakia lemma}
\begin{defi}
For all maps $f, g:\bbas\rightarrow \bbbs$,
\begin{enumerate}
\item
$f$ is {\em positive closed Esakia} if it preserves down-directed meets of closed elements of $\bbas$, that is: $$f(\bigwedge\{c_i : i\in I\})=\bigwedge\{f(c_i): i\in I\}$$ for any downward-directed collection $\{c_i : i\in I\}\subseteq \kbbas$;
\item
$g$ is {\em positive open Esakia} if it preserves upward-directed joins of open elements of $\bbas$, that is: $$g(\bigvee\{o_i : i\in I\})=\bigvee\{g(o_i): i\in I\}$$ for any upward-directed collection $\{o_i : i\in I\}\subseteq O(\ca)$.
\item
$f$ is {\em negative closed Esakia} if it reverses upward-directed joins of open elements  of $\bbas$, that is: $$f(\bigvee\{o_i : i\in I\})=\bigwedge\{f(o_i): i\in I\}$$ for any  upward-directed collection $\{o_i : i\in I\}\subseteq O(\ca)$;
\item
$g$ is {\em negative open Esakia} if it reverses down-directed meets of closed elements of $\bbas$, that is: $$g(\bigwedge\{c_i : i\in I\})=\bigvee\{g(c_i): i\in I\}$$ for any downward-directed collection $\{c_i : i\in I\}\subseteq \kbbas$.

\end{enumerate}
\end{defi}

\begin{lem}\label{lem:idealtoidealmaps}
Let $f,g:\bbas\to \bbbs$. For any $o \in \obbbs$ and $k \in \kbbbs$,%\marginnote{check the statement of this lemma}
\begin{enumerate}
\item
if $f$ is positive closed Esakia, $f(a)\in \kbbbs$ for every $a\in \bba$, and  $f(\bot)\leq o$, then $\blacksquare_f (o) = \bigvee\{ a\in \bba \mid a \leq \blacksquare_{f}(o) \}\in \obbas$. %The upper adjoint $\blacksquare_f$ of $\diam_f$ sends ideal elements to ideal elements.
\item
if $g$ is positive open Esakia, $g(a)\in \obbbs$ for every $a\in \bba$, and $k\leq g(\top)$, then $\Diamondblack_g (k) = \bigwedge\{ a\in \bba \mid \Diamondblack_{g}(k)\leq a \}\in  \kbbas$. %The upper adjoint $\blacksquare_f$ of $\diam_f$ sends ideal
\item
if $f$ is negative closed Esakia, $f(a)\in \kbbbs$ for every $a\in \bba$, and  $f(\top)\leq o$, then ${\blacktriangleleft}_f (o) = \bigwedge\{ a\in \bba \mid {\blacktriangleleft}_{f}(o)\leq a \}\in \kbbas$. %The upper adjoint $\blacksquare_f$ of $\diam_f$ sends ideal elements to ideal elements.
\item
if $g$ is negative open Esakia, $g(a)\in \obbbs$ for every $a\in \bba$, and $k\leq g(\bot)$, then ${\blacktriangleright}_g (k) = \bigvee\{ a\in \bba \mid a\leq {\blacktriangleright}_{g}(k) \}\in  \obbas$. %The upper adjoint $\blacksquare_f$ of $\diam_f$ sends ideal

\end{enumerate}
\end{lem}
\begin{proof}
1. To prove the statement, it
is enough to show that if $c\in \kbbas$ and $c\leq \blacksquare_f (o)$, then $c\leq a$ for some $a\in
\bba$ such that $a\leq {\blacksquare}_f (o)$.  By adjunction, $c\leq \blacksquare_f (o)$ is equivalent to $\Diamond_f(c)\leq o$. Then, by assumption,  $f(\bot)\vee \Diamond_f(c)\leq o$. The definition of $\Diamond_f$ implies that $f(c)\leq f(\bot)\vee \Diamond_f(c)\leq o$.
%
%Lemma \ref{lem:diamprops}.2, $\Diamond_f (s)  = \bigwedge \{\Diamond_f (a) \mid a \in A \textrm{ and }  s\leq a\}$. Note that by Lemma \ref{lem:diamprops}.3, each  $\Diamond_f (a) \in F (A^\delta)$,
Since $f$ is closed Esakia, $f(c) = \bigwedge\{f(a)\ |\ a\in \mathbb{A}\mbox{ and } c\leq a\}$. Moreover, by assumption, $f(a)\in \kbbbs$ for every $a\in \mathbb{A}$. Hence by compactness, $\bigwedge_{i=1}^n f(a_i) \leq o$ for some $a_1, \ldots, a_n\in \mathbb{A}$ s.t.\ $c \leq a_i$, $1\leq i\leq n$.
%$\bigwedge_{i=1}^n\Diamond_f (a_i) \leq i$ for some $a_i\in A$ s.t.\ $s \leq a_i$, $1\leq i\leq n$.
Let $a = \bigwedge_{i=1}^n a_i$. Clearly, $c\leq a$ and $a\in \mathbb{A}$; moreover, by the monotonicity of $f$ and Lemma \ref{lemma:LE algebras interpret normalizations}.1, we have $\Diamond_f (a)\leq f(a)\leq  \bigwedge_{i=1}^n f (a_i) \leq o$, and hence, by adjunction, $a\leq \blacksquare_f (o)$.

Clauses 2, 3 and 4 are order-variants of 1.
\end{proof}

\noindent The following Esakia-type result is similar to \cite[Proposition 30]{CoPaSoZh15}. There it was called {\em conditional Esakia lemma}, since, unlike the usual versions, it crucially relies on additional assumptions (on $f(\bot)$ and $g(\top)$).

\begin{prop}[Conditional Esakia Lemma]\label{prop:Esa}
 Let $f,g:\bbas\to \bbbs$ such that $f$ is closed Esakia and $f(a)\in \kbbbs$ for every $a\in \bba$, and $g$ is open Esakia and $g(a)\in \obbbs$ for every $a\in \bba$. For any upward-directed collection $\mathcal{O} \subseteq \obbbs$ and any downward-directed collection $\mathcal{C} \subseteq \kbbbs$,
\begin{enumerate}
\item
if $f(\bot)\leq \bigvee \mathcal{O}$, then $\blacksquare_f (\bigvee \mathcal{O}) = \bigvee\{ \blacksquare_f o \mid o \in \mathcal{O} \}$. Moreover, there exists some upward-directed subcollection  $\mathcal{O}'\subseteq \mathcal{O}$ such that $\bigvee\mathcal{O}' = \bigvee \mathcal{O}$, and $\blacksquare_f o\in O(\ca)$ for each $o\in \mathcal{O}'$, and $\bigvee\{ \blacksquare_f o \mid o \in \mathcal{O}' \} = \bigvee\{ \blacksquare_f o \mid o \in \mathcal{O} \}$.  %The upper adjoint $\blacksquare_f$ of $\diam_f$ sends ideal elements to ideal elements.
\item
if $g(\top)\geq \bigwedge \mathcal{C}$, then $\Diamondblack_g (\bigwedge \mathcal{C}) = \bigwedge\{ \Diamondblack_g c \mid c \in \mathcal{C} \}$. Moreover, there exists some downward-directed subcollection  $\mathcal{C}'\subseteq \mathcal{C}$ such that $\bigwedge\mathcal{C}' = \bigwedge \mathcal{C}$, and $\Diamondblack_g c\in K(\ca)$ for each $c\in \mathcal{C}'$, and $\bigwedge\{ \Diamondblack_g c \mid c \in \mathcal{C}' \} = \bigwedge\{ \Diamondblack_g c \mid c \in \mathcal{C} \}$.
\end{enumerate}
\end{prop}
\begin{proof}
1. For the first part, it is enough to show that, if $k\in \kbbas$ and $k\leq \blacksquare_f (\bigvee \mathcal{O}) $, then $k\leq \blacksquare_f (o_k) $ for some $o_k\in \mathcal{O} $. The assumption $k\leq \blacksquare_f (\bigvee \mathcal{O}) $ can be rewritten as $\Diamond_f k\leq \bigvee \mathcal{O} $, which together with $f(\bot)\leq \bigvee \mathcal{O}$ yields $f(\bot)\vee \Diamond_f k\leq  \bigvee \mathcal{O}$. By Lemma \ref{lemma:LE algebras interpret normalizations}.1, this inequality can be rewritten as $f(k)\leq \bigvee \mathcal{O} $. From  $f$ being closed Esakia and $f(a)\in \kbbbs$ for every $a\in \bba$, it follows that $f(k)\in \kbbbs$. By compactness, $f( k)\leq \bigvee_{i = 1}^n o_i $ for some $o_1,\ldots,o_n\in \mathcal{O}$. Since $\mathcal{O}$ is upward-directed,  $o_k \geq \bigvee_{i = 1}^n o_i$ for some $o_k\in \mathcal{O}$. Then  $f(\bot)\vee \Diamond_f k = f(k)\leq o_k$, which yields $\Diamond_f k\leq o_k$, which by adjunction can be rewritten as $k\leq \blacksquare_f (o_k) $ as required.

As to the second part of the statement, notice that the assumption  $f(\bot)\leq \bigvee \mathcal{O}$ is too weak to imply that  $f(\bot)\leq o$ for each $o\in \mathcal{O}$, and hence we cannot conclude, by way of   Lemma \ref{lem:idealtoidealmaps}, that $\blacksquare_f o \in O(\ca)$ for every $o\in \mathcal{O}$. However, let \[\mathcal{O}': = \{o\in \mathcal{O}\mid o\geq o_k \mbox{ for some } k\in K(\ca) \mbox{ s.t. } k\leq  \blacksquare_f (\bigvee \mathcal{O})\}.\]
Clearly, by construction we have that $\mathcal{O}'$ is upward-directed and $\bigvee\mathcal{O}' = \bigvee \mathcal{O}$, and moreover for each $o\in \mathcal{O}'$, we have that $f(\bot)\leq o_k\leq o$ for some  $k\in K(\ca)$ s.t.\ $k\leq  \blacksquare_f (\bigvee \mathcal{O})$. Hence, by Lemma \ref{lem:idealtoidealmaps}, $\blacksquare_f o\in O(\ca)$ for each $o\in \mathcal{O}'$. Moreover, the monotonicity of $\blacksquare_f$ and the previous part of the statement imply that  $\bigvee\{ \blacksquare_f o \mid o \in \mathcal{O}' \} = \blacksquare_f (\bigvee \mathcal{O}) = \bigvee\{ \blacksquare_f o \mid o \in \mathcal{O} \}$.

2. is order-dual.
\end{proof}

\subsection{Intersection lemmas for $\mathcal{L}_{\mathrm{LE}}^+$-formulas}\label{ssec:intersection lemma}

Recall that  $\blacksquare_f, \Diamondblack_g, \blhd, \brhd: \bbas \to \bbas$ respectively denote the adjoints of the normalization maps $\Diamond_f$ for any $f\in \mathcal{F}_r$ with $\varepsilon_f=1$, $\Box_g$ for any $g\in \mathcal{G}_r$ with $\varepsilon_g=1$, ${\lhd}_f$ for any $f\in \mathcal{F}_r$ with $\varepsilon_f=\partial$, ${\rhd}_g$ for any $g\in \mathcal{G}_r$ with $\varepsilon_g=\partial$. %In what follows, we will omit to mention the polarities of $f\in \mathcal{F}_r$ and $g\in \mathcal{G}_r$, since this information can be inferred from the notation used for the adjoint of their normalizations.
Then Lemma \ref{lem:idealtoidealmaps} immediately implies the following facts, which will be needed for the soundness of the topological Ackermann rule:
\begin{lem}\label{lem:idealtoideal}
\begin{enumerate}
\item
If $f\in \mathcal{F}_r$ with $\varepsilon_f=1$, and $o \in \obbas$ s.t.\ $f(\bot)\leq o$, then $\blacksquare_f (o) = \bigvee\{ a \ | \ a \leq \blacksquare_{f}(o) \} \in  \obbas$.
\item
If $g\in \mathcal{G}_r$ with $\varepsilon_g=1$, and $k\in \kbbas$ s.t.\ $g(\top)\geq k$, then $\Diamondblack_g (k) = \bigwedge\{ a \ | \ a \geq \Diamondblack_{g}(k) \} \in  \kbbas$.
\item
If $f\in \mathcal{F}_r$ with $\varepsilon_f=\partial$, and $o\in \obbas$ s.t.\ $f(\top)\leq o$, then $\blhd (o) = \bigwedge\{ a \ | \ a \geq \blhd(o) \} \in  \kbbas$.
\item
If $g\in \mathcal{G}_r$ with $\varepsilon_g=\partial$, and $k \in \kbbas$ s.t.\ $g(\bot)\geq k$, then $\brhd (k) = \bigvee\{ a \ | \ a \leq \brhd(k) \} \in  \obbas$.
\end{enumerate}
 %The upper adjoint $\blacksquare_\pi$ of $\diam_\pi$ sends ideal elements to ideal elements.
\end{lem}

\noindent In what follows,  we work under the assumption that the values of all parameters (propositional variables, nominals and conominals) occurring in the term functions mentioned in the statements of propositions and lemmas  are given by  admissible assignments.
%Moreover, we will omit to mention the polarities of $f\in \mathcal{F}_r$ and $g\in \mathcal{G}_r$, since this information can be inferred from the notation used for the adjoint of their normalizations.

\begin{lem}
\label{lemma:synct closed is sem closed}
Let $\phi(p)$ be syntactically closed, $\psi(p)$  syntactically open (cf.\ Definition \ref{def:syn:closed:and:open}), $c\in\kbbas$ and $o\in\obbas$.

\begin{enumerate}
\item If $\phi(p)$ is positive in $p$, $\psi(p)$ is negative in $p$, and for all $f\in \mathcal{F}_r$ and  $g\in \mathcal{G}_r$,
\begin{enumerate}
\item if $\varepsilon_f = 1$, then $f(\bot)\leq \psi'^{\ca}(c)$ for any subformula $\blacksquare_{f}\psi'(p)$ in $\phi(p)$ and  $\psi(p)$;
\item if $\varepsilon_g = 1$, then $g(\top)\geq \psi'^{\ca}(c)$ for any subformula $\Diamondblack_{g}\psi'(p)$ in $\phi(p)$ and  $\psi(p)$;
\item if $\varepsilon_f = \partial$, then $f(\top)\leq \psi'^{\ca}(c)$ for any subformula $\blhd\!\!\psi'(p)$ in $\phi(p)$ and  $\psi(p)$;
\item if $\varepsilon_g = \partial$, then $g(\bot)\geq \psi'^{\ca}(c)$ for any subformula $\brhd\!\!\psi'(p)$ in $\phi(p)$ and  $\psi(p)$;
\end{enumerate}
then $\phi(c)\in \kbbas$ and $\psi(c)\in \obbas$ for each $c\in \kbbas$.

\item If $\phi(p)$ is negative in $p$, $\psi(p)$ is positive in $p$, and  for all $f\in \mathcal{F}_r$ and  $g\in \mathcal{G}_r$,
\begin{enumerate}
\item if $\varepsilon_f = 1$, then $f(\bot)\leq \psi'^{\ca}(o)$ for any subformula $\blacksquare_{f}\psi'(p)$ in $\phi(p)$ and  $\psi(p)$;
\item if $\varepsilon_g = 1$, then $g(\top)\geq \psi'^{\ca}(o)$ for any subformula $\Diamondblack_{g}\psi'(p)$ in $\phi(p)$ and  $\psi(p)$;
\item if $\varepsilon_f = \partial$, then $f(\top)\leq \psi'^{\ca}(o)$ for any subformula $\blhd\!\!\psi'(p)$ in $\phi(p)$ and  $\psi(p)$;
\item if $\varepsilon_g = \partial$, then $g(\bot)\geq \psi'^{\ca}(o)$ for any subformula $\brhd\!\!\psi'(p)$ in $\phi(p)$ and  $\psi(p)$;
\end{enumerate}
then $\phi(o)\in K(\ca)$ and $\psi(o)\in O(\ca)$ for each $o\in O(\ca)$.
\end{enumerate}
\end{lem}
\begin{proof}
1. The proof proceeds by simultaneous induction on $\phi$ and $\psi$. It is easy to see that $\phi$ cannot be $\cnomm$, and the outermost connective of $\phi$ cannot be $f^{\sharp}_i$ for any $f\in\mathcal{F}_n$ with $\varepsilon_f(i) = 1$, or $g_j^{\flat}$ for any $g\in\mathcal{G}_n$ with $\varepsilon_g(j) = \partial$, or $\blacksquare_{f}$ for any $f\in \mathcal{F}_r$ with $\varepsilon_f = 1$, or $\blacktriangleright_{g}$ for any $g\in \mathcal{G}_r$ with $\varepsilon_g = \partial$. Similarly, $\psi$ cannot be $\nomi$, and the outermost connective of $\psi$ cannot be $g_j^{\flat}$ for any $g\in\mathcal{G}_n$ with $\varepsilon_g(j) = 1$, or $f^{\sharp}_i$ for any $f\in\mathcal{F}_n$ with $\varepsilon_f(i) = \partial$, or $\Diamondblack_{g}$ for any $g\in \mathcal{G}_r$ with $\varepsilon_g = 1$, or $\blacktriangleleft_{f}$ for any $f\in \mathcal{F}_r$ with $\varepsilon_f = \partial$.

The basic cases, that is, $\phi=\perp, \top, p, q, \nomi$ and $\psi=\perp, \top, p, q, \cnomm$ are straightforward.

Assume that $\phi(p)=\Diamondblack_{g}\phi'(p)$. Since $\phi(p)$ is positive in $p$, the subformula $\phi'(p)$ is syntactically closed and positive in $p$, and assumptions 1(a)-1(d) hold also for $\phi'(p)$.  Hence, by inductive hypothesis, $\phi'(c)\in \kbbas $ for any $c\in \kbbas$. %Hence, $\phi'(c) = \bigwedge\{a\in A\mid a\geq \phi'(c)\}$.
In particular, assumption 1(b) implies that $g(\top)\geq \phi'(c)$. Hence, by Lemma \ref{lem:idealtoideal}, $\Diamondblack_{g}\phi'(c)\in \kbbas$, as required. The case in which $\phi(p)$ is negative in $p$ is argued order-dually.

The cases in which $\phi(p)=\blacksquare_{f}\phi'(p), \blhd\!\!\phi'(p), \brhd\!\!\phi'(p)$ are similar to the one above.

The cases of the remaining connectives are treated as in \cite[Lemma 11.9]{ConPal12} and the corresponding proofs are omitted. Of course, in \cite{ConPal12} nominals and co-nominals are evaluated to completely join- and meet-irreducibles, respectively. However, the only property of completely join- and meet-irreducibles used there is the fact that they are, respectively, closed and open, and therefore the proof is directly transferable to the present setting.
%\marginnote{Make comment that that lemma works point free and does not make essential use of evaluation of nominals to join-irreducibles etc.}
\end{proof}

\begin{lem}[Intersection lemma]\label{MJ:Pres}

Let $\phi(p)$ be syntactically closed, $\psi(p)$  syntactically open, $\mathcal{C}\subseteq \kbbas$  downward-directed, $\mathcal{O}\subseteq \obbas$ upward-directed. Then

\begin{enumerate}
\item if $\phi(p)$ is positive in $p$, $\psi(p)$ is negative in $p$, and for all $f\in \mathcal{F}_r$ and  $g\in \mathcal{G}_r$,

\begin{enumerate}
\item if $\varepsilon_f = 1$, then  $f(\bot)\leq \psi'^{\ca}(\bigwedge\mathcal{C})$  for any subformula $\blacksquare_{f}\psi'(p)$ in $\phi(p)$ and  $\psi(p)$,
\item if $\varepsilon_g = 1$, then $g(\top)\geq \psi'^{\ca}(\bigwedge\mathcal{C})$ for any subformula $\Diamondblack_{g}\psi'(p)$ in $\phi(p)$ and $\psi(p)$,
\item if $\varepsilon_f = \partial$, then $f(\top)\leq \psi'^{\ca}(\bigwedge\mathcal{C})$ for any subformula $\blhd\!\!\psi'(p)$ in $\phi(p)$ and  $\psi(p)$,
\item if $\varepsilon_g = \partial$, then  $g(\bot)\geq \psi'^{\ca}(\bigwedge\mathcal{C})$ for any subformula $\brhd\!\!\psi'(p)$ in $\phi(p)$ and  $\psi(p)$,
\end{enumerate}

then

\begin{enumerate}
\item[(a)] $\phi^{\bbas}(\bigwedge\mathcal{C})=\bigwedge\{\phi^{\bbas}(c): c\in \mathcal{C}'\}$ for some down-directed subcollection $\mathcal{C}'\subseteq \mathcal{C}$ such that $\phi^{\bbas}(c)\in \kbbas$ for each $c\in \mathcal{C}'$.
\item[(b)] $\psi^{\bbas}(\bigwedge\mathcal{C})=\bigvee\{\psi^{\bbas}(c): c\in \mathcal{C}'\}$ for some  down-directed subcollection $\mathcal{C}'\subseteq \mathcal{C}$ such that $\psi^{\bbas}(c)\in \obbas$ for each $c\in \mathcal{C}'$.
\end{enumerate}
\item If $\phi(p)$ is negative in $p$, $\psi(p)$ is positive in $p$, and for all $f\in \mathcal{F}_r$ and  $g\in \mathcal{G}_r$,

\begin{enumerate}
\item if $\varepsilon_f = 1$, then $f(\bot)\leq \psi'^{\ca}(\bigvee\mathcal{O})$  for any subformula $\blacksquare_{f}\psi'(p)$ in $\phi(p)$ and  $\psi(p)$,
\item if $\varepsilon_g = 1$, then $g(\top)\geq \psi'^{\ca}(\bigvee\mathcal{O})$ for any subformula $\Diamondblack_{g}\psi'(p)$ in $\phi(p)$ and  $\psi(p)$,
\item if $\varepsilon_f = \partial$, then $f(\top)\leq \psi'^{\ca}(\bigvee\mathcal{O})$ for any subformula $\blhd\!\!\psi'(p)$ in $\phi(p)$ and  $\psi(p)$,
\item if $\varepsilon_g = \partial$, then $g(\bot)\geq \psi'^{\bbas}(\bigvee\mathcal{O})$ for any subformula $\brhd\!\!\psi'(p)$ in $\phi(p)$ and  $\psi(p)$,

\end{enumerate}

then

\begin{enumerate}

\item[(a)] $\phi^{\ca}(\bigvee\mathcal{O})=\bigwedge\{\phi^{\ca}(o): o\in \mathcal{O}'\}$ for some  up-directed subcollection $\mathcal{O}'\subseteq \mathcal{O}$ such that $\phi^{\ca}(o)\in K(\ca)$ for each $o\in \mathcal{O}'$.

\item[(b)] $\psi^{\ca}(\bigvee\mathcal{O})=\bigvee\{\psi^{\ca}(o): o\in \mathcal{O}'\}$ for some up-directed subcollection $\mathcal{O}'\subseteq\mathcal{O}$ such that $\psi^{\ca}(o)\in O(\ca)$ for each $o\in \mathcal{O}'$.
\end{enumerate}
\end{enumerate}
\end{lem}
\begin{proof}
The proof proceeds by simultaneous induction on $\phi$ and $\psi$. It is easy to see that $\phi$ cannot be $\cnomm$, and the outermost connective of $\phi$ cannot be $f^{\sharp}_i$ for any $f\in\mathcal{F}_n$ with $\varepsilon_f(i) = 1$, or $g_j^{\flat}$ for any $g\in\mathcal{G}_n$ with $\varepsilon_g(j) = \partial$, or $\blacksquare_{f}$ for any $f\in \mathcal{F}_r$ with $\varepsilon_f = 1$, or $\blacktriangleright_{g}$ for any $g\in \mathcal{G}_r$ with $\varepsilon_g = \partial$. Similarly, $\psi$ cannot be $\nomi$, and the outermost connective of $\psi$ cannot be $g_j^{\flat}$ for any $g\in\mathcal{G}_n$ with $\varepsilon_g(j) = 1$, or $f^{\sharp}_i$ for any $f\in\mathcal{F}_n$ with $\varepsilon_f(i) = \partial$, or $\Diamondblack_{g}$ for any $g\in \mathcal{G}_r$ with $\varepsilon_g = 1$, or $\blacktriangleleft_{f}$ for any $f\in \mathcal{F}_r$ with $\varepsilon_f = \partial$.

The basic cases in which $\phi=\perp, \top, p, q, \nomi$ and $\psi=\perp, \top, p, q, \cnomm$ are straightforward.

Assume that $\phi(p)=\Diamondblack_{g}\phi'(p)$ for some $g\in \mathcal{G}_r$ with $\varepsilon_g = 1$. Since $\phi(p)$ is positive in $p$, the subformula $\phi'(p)$ is syntactically closed and positive in $p$, and assumptions 1(a)-1(d) hold also for $\phi'(p)$.  Hence, by inductive hypothesis, $\phi'(\bigwedge\mathcal{C}) = \bigwedge\{\phi'(c)\mid c\in\mathcal{C}''\}$ for some down-directed subcollection $\mathcal{C}''\subseteq \mathcal{C}$ such that $\phi'(c)\in \kbbas$ for each $c\in \mathcal{C}''$. In particular, assumption 1(b) implies that $g(\top)\geq \phi'(\bigwedge\mathcal{C}) = \bigwedge\{\phi'(c)\mid c\in \mathcal{C}''\}$. Notice that $\phi'(p)$ being positive in $p$ and $\mathcal{C}''$ being down-directed imply that $\{\phi'(c)\mid c\in \mathcal{C}''\}$ is down-directed. Hence, by Proposition \ref{prop:Esa} applied to $\{\phi'(c)\mid c\in \mathcal{C}''\}$, we get that $\Diamondblack_{g}\phi'(\bigwedge\mathcal{C}) = \Diamondblack_{g}(\bigwedge\{\phi'(c)\mid c\in\mathcal{C}''\}) = \bigwedge\{\Diamondblack_{g}\phi'(c)\mid c\in\mathcal{C}''\}$. Moreover, there exists some down-directed subcollection $\mathcal{C}'\subseteq \mathcal{C}''$ such that $\Diamondblack_{g}\phi'(c)\in \kbbas$ for each $c\in \mathcal{C}'$ and $\bigwedge\{\Diamondblack_{g}\phi'(c)\mid c\in\mathcal{C}'\} = \bigwedge\{\Diamondblack_{g}\phi'(c)\mid c\in\mathcal{C}''\}$. This gives us $\Diamondblack_{g}\phi'(\bigwedge\mathcal{C}) =  \bigwedge\{\Diamondblack_{g}\phi'(c)\mid c\in\mathcal{C}'\}$, as required. The case in which $\phi(p)$ is negative in $p$ is argued order-dually.

The cases in which $\phi(p)=\blacksquare_{g}\phi'(p), \blhd\!\!\phi'(p), \brhd\!\!\phi'(p)$ are similar to the one above.

The cases of the remaining connectives are treated as in \cite[Lemma 11.10]{ConPal12} and the corresponding proofs are omitted.
\end{proof}

\subsection{Topological Ackermann Lemmas for $\mathcal{L}_{\mathrm{LE}}^+$}\label{ssec:topological ackermann}

\begin{defi}
A system $S$ of $\mathcal{L}_{\mathrm{LE}}^+$-inequalities  is {\em topologically adequate} when the following conditions hold for all $f\in \mathcal{F}_r$ and  $g\in \mathcal{G}_r$:
\begin{enumerate}
\item if $\varepsilon_f = 1$ and $\phi\leq \blacksquare_f \psi$ is in $S$, then $f(\bot)\leq \psi$ is in $S$;
\item if $\varepsilon_g = 1$ and $\Diamondblack_g\phi\leq  \psi$ is in $S$, then $g(\top)\geq \phi$ is in $S$;
\item if $\varepsilon_f = \partial$ and $\blacktriangleleft_f\!\!\phi\leq \psi$ is in $S$, then $f(\top)\leq \phi$ is in $S$;
\item if $\varepsilon_g = \partial$ and $\phi\leq \blacktriangleright_g\!\! \psi$ is in $S$, then $g(\bot)\geq \psi$ is in $S$.
\end{enumerate}
\end{defi}

\begin{prop}[Right-handed Topological Ackermann Lemma]\label{Right:Ack}

Let $S$ be a topologically adequate system of $\mathcal{L}_{\mathrm{LE}}^+$-inequalities which is the union of the following disjoint subsets:
\begin{itemize}
\item
$S_1$ consists only of inequalities in which $p$ does not occur;
\item
$S_2$ consists of inequalities of the type $\alpha\leq p$, where $\alpha$ is syntactically closed and $p$ does not occur in $\alpha$;
\item
$S_3$ consists of inequalities of the type $\beta(p)\leq \gamma(p)$ where $\beta(p)$ is syntactically closed and positive in $p$, and $\gamma(p)$ be syntactically open and negative in $p$,

\end{itemize}

Then the following are equivalent:
\vspace{1mm}
\begin{enumerate}
\item
\vspace{1mm}
$\beta^{\bbas}(\bigvee\alpha^{\bbas})\leq\gamma^{\bbas}(\bigvee\alpha^{\bbas})$ for all inequalities in $S_3$, where $\bigvee\alpha$ abbreviates $\bigvee\{\alpha\mid \alpha\leq p\in S_2\}$;

\vspace{1mm}
\item
There exists $a_0\in\bba$ such that $\bigvee\alpha^{\bbas}\leq a_0$ and $\beta^{\bbas}(a_0)\leq\gamma^{\bbas}(a_0)$ for all inequalities in $S_3$.

\end{enumerate}

\end{prop}

\begin{proof}

$(\Leftarrow)$ By the monotonicity of $\beta_i(p)$ and antitonicity of $\gamma_i(p)$ in $p$ for $1\leq i\leq n$, together with $\alpha^{\bbas}\leq a_0$, we have that $\beta_i^{\bbas}(\alpha^{\bbas})\leq\beta_i^{\bbas}(a_0)\leq\gamma_i^{\bbas}(a_0)\leq\gamma_i^{\bbas}(\alpha^{\bbas})$.\\

$(\Rightarrow)$ Since the quasi-inequality is topologically adequate, by Lemma \ref{MJ:Pres}.1, $\alpha^{\bbas}\in\kbbas$.

Hence, $\alpha^{\bbas}=\bigwedge\{a\in\bba : \alpha^{\bbas}\leq a\}$, making it the meet of a downward-directed set of clopen elements. Therefore, we can rewrite each inequality in $S_3$ as \[\beta^{\bbas}(\bigwedge\{a\in\bba : \alpha^{\bbas}\leq a\})\leq\gamma^{\bbas}(\bigwedge\{a\in\bba : \alpha^{\bbas}\leq a\}).\] Since $\beta$ is syntactically closed and positive in $p$, $\gamma$ is syntactically open and negative in $p$, again by topological adequacy, we can apply Lemma \ref{MJ:Pres} and get that \[\bigwedge\{\beta^{\bbas}(a): a\in\mathcal{A}_1\}\leq\bigvee\{\gamma_i^{\bbas}(b): b\in\mathcal{A}_2\}\] for some $\mathcal{A}_1, \mathcal{A}_2\subseteq \{a\in\bba : \alpha^{\bbas}\leq a\}$ such that $\beta^{\bbas}(a)\in K(\ca)$ for each $a\in \mathcal{A}_1$, and $\gamma^{\bbas}(b)\in O(\ca)$ for each $b\in \mathcal{A}_2$. By compactness,

\[\bigwedge\{\beta_i^{\bbas}(a): a\in\mathcal{A}'_1\}\leq\bigvee\{\gamma_i^{\bbas}(b): b\in\mathcal{A}'_2\}\] for some finite subsets $\mathcal{A}'_1\subseteq\mathcal{A}_1$ and $\mathcal{A}'_2\subseteq\mathcal{A}_2$.  Then let $a^*=\bigwedge\{\bigwedge\mathcal{A}'_1\wedge \bigwedge\mathcal{A}'_2 \mid \beta\leq\gamma\in S_3\}.$ Clearly, $a^*\in\bba$, and
$\alpha^{\bbas}\leq a^*$.  By the monotonicity of $\beta(p)$ and the antitonicity of $\gamma(p)$ in $p$ for each $\beta\leq \gamma$ in $S_3$, we have $\beta^{\bbas}(a^*)\leq\beta^{\bbas}(a)$ and $\gamma_i^{\bbas}(b)\leq\gamma_i^{\bbas}(a^*)$ for all $a\in\mathcal{A}'_1$ and all $b\in\mathcal{A}'_2$. Therefore, \[\beta_i^{\bbas}(a^*)\leq\bigwedge\{\beta_i^{\bbas}(a): a\in\mathcal{A}'_1\}\leq\bigvee\{\gamma_i^{\bbas}(b): b\in\mathcal{A}'_2\}\leq\gamma_i^{\bbas}(a^*)\] for each  $\beta\leq \gamma$ in $S_3$.
\end{proof}

\begin{prop}[Left-handed Topological Ackermann Lemma]\label{Left:Ack}

Let $S$ be a topologically adequate system of $\mathcal{L}_{\mathrm{LE}}^+$-inequalities which is the union of the following disjoint subsets:
\begin{itemize}
\item
$S_1$ consists only of inequalities in which $p$ does not occur;
\item
$S_2$ consists of inequalities of the type $p\leq\alpha$, where $\alpha$ is syntactically open and $p$ does not occur in $\alpha$;
\item
$S_3$ consists of inequalities of the type $\beta(p)\leq \gamma(p)$ where $\beta(p)$ is syntactically closed and negative in $p$, and $\gamma(p)$ be syntactically open and positive in $p$,

\end{itemize}
Then the following are equivalent:
\vspace{1mm}
\begin{enumerate}
\item
\vspace{1mm}
$\beta^{\bbas}(\bigwedge\alpha^{\bbas})\leq\gamma^{\bbas}(\bigwedge\alpha^{\bbas})$ for all inequalities in $S_3$, where $\bigwedge\alpha$ abbreviates $\bigwedge\{\alpha\mid p\leq\alpha\in S_2\}$;

\vspace{1mm}
\item
There exists $a_0\in\bba$ such that $a_0\leq\bigwedge\alpha^{\bbas}$ and $\beta^{\bbas}(a_0)\leq\gamma^{\bbas}(a_0)$ for all inequalities in $S_3$.
\end{enumerate}
\end{prop}
\begin{proof}
The proof is similar to the proof of the right-handed Ackermann lemma and is omitted.
\end{proof}

\end{document}